\documentclass[11pt]{amsart}

\usepackage{ifpdf}          
\usepackage{etoolbox}       

\usepackage[margin=1in]{geometry}

\usepackage[foot]{amsaddr}  

\usepackage[numbers]{natbib} 

\usepackage{hyperref}
\hypersetup{colorlinks=true,allcolors=blue}

\providecommand{\doi}[1]{}
\renewcommand{\doi}[1]{\href{https://doi.org/#1}{https://doi.org/#1}}

\usepackage{comment}
\usepackage{graphicx}
\graphicspath{{figures/}}
\usepackage{booktabs}
\usepackage{multirow}
\usepackage{adjustbox}
\usepackage{rotating}
\usepackage{bm}
\usepackage[normalem]{ulem}
\usepackage{wrapfig}

\usepackage{pgfplots}
\usepackage{mathtools}
\usepackage{amssymb}
\usepackage[inline]{enumitem}
\usepackage{siunitx}
\newcommand{\tnum}[1]{\num[
  scientific-notation=true,
  round-mode=figures,
  round-precision=3,
  output-exponent-marker=\text{e}
]{#1}}
\usepackage{algorithm}
\usepackage{algpseudocode}

\AtBeginEnvironment{algorithmic}{\small}

\usepackage[table]{xcolor}

\usepackage{cleveref}

\crefname{equation}{}{}
\Crefname{equation}{}{}

\newcommand\fnum[1]{\num[scientific-notation=fixed,round-precision=2,round-mode=places]{#1}}

\newcommand\inum[1]{\num[scientific-notation=fixed,round-precision=0,round-mode=places,group-separator = {,}]{#1}}

\newcommand{\defeq}{\ensuremath{\mathrel{\mathop:}=}}

\newcommand\tabadjust{\centering\scriptsize}

\DeclareMathOperator*\minopt{\ensuremath{minimize}}
\DeclareMathOperator*\argmin{\ensuremath{arg\,min}}
\DeclareMathOperator*\argmax{\ensuremath{arg\,max}}

\newcommand{\hcA}{\widehat{\mathcal{A}}}
\newcommand{\halpha}{\widehat{\alpha}}

\ifpdf
  \DeclareGraphicsExtensions{.eps,.pdf,.png,.jpg}
\else
  \DeclareGraphicsExtensions{.eps}
\fi

\theoremstyle{plain}
\newtheorem{theorem}{Theorem}[section]
\newtheorem{lemma}[theorem]{Lemma}
\newtheorem{proposition}[theorem]{Proposition}

\theoremstyle{definition}
\newtheorem{assumption}[theorem]{Assumption}

\theoremstyle{remark}
\newtheorem{remark}[theorem]{Remark}

\crefname{assumption}{Assumption}{Assumptions}
\Crefname{assumption}{Assumption}{Assumptions}
\crefname{remark}{Remark}{Remarks}
\Crefname{remark}{Remark}{Remarks}
\crefname{fact}{Fact}{Facts}
\Crefname{fact}{Fact}{Facts}
\crefname{claim}{Claim}{Claims}
\Crefname{claim}{Claim}{Claims}

\usepackage{amsopn}

\makeatletter

\renewcommand{\subsection}{%
  \@startsection{subsection}{2}%
  {\z@}%
  {1.5ex \@plus .2ex}%
  {.8ex}%
  {\normalfont\bfseries\itshape}}

\renewcommand{\subsubsection}{%
  \@startsection{subsubsection}{3}%
  {\z@}%
  {1.25ex \@plus .2ex}%
  {-1em}%
  {\normalfont\bfseries\itshape}}

\renewcommand{\paragraph}{%
  \@startsection{paragraph}{4}%
  {\z@}%
  {1ex}%
  {-1em}%
  {\normalfont\bfseries\itshape}}

\makeatother

\title[TROM for Inverse Problems]{Tensor-Based Reduced-Order Modeling for
Optimization-Based Inverse Problems}

\author[Sahidul Islam]{Sahidul Islam\textsuperscript{$\dag$}}

\author[Andreas Mang]{Andreas Mang\textsuperscript{$\ddag$}}
\email{andreas.mang@tufts.edu}

\author[Maxim Olshanskii]{Maxim Olshanskii\textsuperscript{$\dag$}}
\email{maolshanskiy@uh.edu}

\address[$\dag$]{Department of Mathematics, University of Houston, Houston, TX 77204, USA}
\address[$\ddag$]{Department of Mathematics, Tufts University, Medford, MA 02155, USA}

\thanks{The work of the authors was partially supported by the National Science Foundation (NSF) under awards DMS-2145845 (AM) and DMS-2309197 (MO and SI). This material is based upon work supported by the NSF under Grant DMS-1929284 while AM was in residence at the Institute for Computational and Experimental Research in Mathematics in Providence, RI, during the ``Stochastic and Randomized Algorithms in Scientific Computing: Foundations and Applications'' semester program. We are grateful for the support of the Research Computing Data Core at the University of Houston.}

\keywords{inverse problems, reduced-order modeling, error analysis,
low-rank tensor approximation, Gauss--Newton method, regularized nonlinear
least squares, parameter estimation}

\subjclass[2020]{65M32, 65F55, 15A69, 65K10}

\begin{document}

\begin{abstract}
We develop a tensor reduced-order modeling (TROM) framework for optimization-based inverse problems governed by parameter-dependent dynamical systems. The approach approximates the parameter-to-observation map directly in tensor-train (TT) format, using either TT-SVD or TT-Cross compression, and integrates the resulting representation into a regularized nonlinear least-squares formulation. Beyond accelerating forward evaluations, the low-rank tensor structure is used to reformulate the inverse problem in reduced coordinates, assemble the Gauss--Newton quantities without forming the full observation-space Jacobian, and perform TROM-based objective minimization over the discrete parameter grid. This tensor optimization step can be used either as a stand-alone approximate minimization procedure or as a data-informed initialization for a subsequent Gauss--Newton solve. The method is studied for two inverse problems: an inverse heat-transfer problem in a heterogeneous medium, where the unknown parameters describe the locations and radii of multiple low-conductivity inclusions, and a FitzHugh--Nagumo parameter-estimation problem with a highly nonconvex optimization landscape. Numerical experiments assess the effects of reduced-order model approximation error, measurement noise, regularization, initialization, spatial discretization, and increasing parameter dimension. The results show that TROM can reproduce the behavior of full-order inversion at a substantially reduced online cost. The experiments also demonstrate that reduced-coordinate inversion, tensor-based optimization, and appropriate regularization improve robustness in higher-dimensional, noisy, and strongly nonconvex regimes. For the continuous TROM inverse problem, we also develop an error-to-inversion analysis. Under local strong convexity of the regularized FOM objective and parametric smoothness of the FOM observation map, the error between the parameters recovered with the full-order model and the TROM is bounded in terms of controlled uniform errors in the surrogate map and its Jacobian, together with local FOM and curvature quantities.
\end{abstract}

\maketitle

\section{Introduction}

Inverse problems for parameter-dependent forward models, including systems governed by partial differential equations (PDEs) and ordinary differential equations (ODEs), arise in many applications where internal material properties, geometric features, source terms, or model parameters must be inferred from indirect and noisy measurements~\cite{kaipio2005statistical,stuart2010inverse,isakov2017inverse,mang2018pdeconstrained}. In computational practice, solving such problems usually requires repeated evaluations of a high-fidelity forward model, often together with derivative or sensitivity information. This cost can become prohibitive when the forward model is large-scale, the parameter dimension is moderate or high, or the inverse problem must be solved repeatedly for different data realizations. These challenges have motivated extensive research on surrogate and reduced-order models (ROMs) for inverse problems, uncertainty quantification, and data assimilation; see, e.g., \cite{frangos2010surrogate,galbally2010nonlinear,cui2015datadriven,benner2015survey,quarteroni2016reduced,ghattas2021learning}. Such approaches replace the expensive parameter-to-observation map by a computationally efficient approximation while retaining sufficient accuracy for the inverse solve.

The need for fast and reliable inverse solvers is particularly pronounced in digital-twin settings, where a computational model is updated as new data become available to monitor, predict, or control the behavior of a physical system \cite{grieves2017digital,tao2018digital,kapteyn2021probabilistic,nasem2024foundational}. Surrogate models are of particular importance if inference needs to happen in (close to) real-time. In this context, inverse problems are not isolated offline computations: they are a recurring component of the model-update loop. Hence, the online stage must be inexpensive, stable with respect to noise and model error, and capable of handling repeated optimization tasks under limited computational budgets.

Classical ROMs typically exploit low-dimensional structure in the state, solution manifold, or input--output map, for example through projection-based model reduction or reduced basis methods~\cite{benner2015survey,quarteroni2016reduced,hesthaven2016certified}. These techniques have been successfully used in large-scale inverse problems~\cite{galbally2010nonlinear,lieberman2010parameter,cui2015datadriven}. However, inverse problems with several unknown parameters introduce an additional difficulty: the parameter-to-observation map may need to be approximated over a multidimensional parameter domain. Direct sampling over a tensor-product grid quickly becomes infeasible, while local surrogates may not provide the robustness needed for globalized or multi-start optimization.

Tensor reduced-order models (TROMs) provide a natural way to address this difficulty by exploiting low-rank structure in the joint dependence of the observable on the model parameters. In particular, low-rank tensor representations, such as tensor-train (TT) decomposition~\cite{oseledets2011tensor}, offer compact approximations of high-dimensional arrays and support efficient tensor arithmetic, interpolation, and optimization. Low-rank tensor methods are now well developed in numerical analysis and high-dimensional approximation~\cite{kolda2009tensor,hackbusch2012tensor,grasedyck2013literature}, and tensor-train cross (TT-Cross) techniques can construct compressed representations by sampling only a subset of tensor entries~\cite{oseledets2010tt,savostyanov2014quasioptimality,dolgov2015polynomial}. TROMs have been developed for parametric dynamical systems~\cite{mamonov2022interpolatory,mamonov2024tensorial,islam2026tensorial}. Nevertheless, the systematic use of \emph{TROMs in inverse modeling} remains comparatively unexplored, especially beyond the basic goal of accelerating forward evaluations.

In this paper, we investigate TROMs for deterministic numerical inversion. Our focus is not only on replacing the full-order model (FOM) with a fast surrogate for forward solves, but also on exploiting the tensor structure in inverse-problem-specific components: Gauss--Newton (GN) sensitivity assembly, reduced-coordinate inversion, and data-informed initialization (see C1--C5 below). These uses are important because the performance of nonlinear inverse solvers depends not only on the cost of residual evaluation, but also on the quality of sensitivity information, the choice of regularization, and the basin of attraction of the initial iterate.

We study the proposed approach on two deterministic inverse problems. The first is a parabolic heat-transfer problem in a heterogeneous medium with low-conductivity inclusions. The unknown parameter vector encodes the locations and radii of one, two, or three inclusions, leading to parameter dimensions $D=3,6,9$. The available data consist of time-dependent temperature observations at several sensors placed on the domain boundary. The second example is a FitzHugh--Nagumo parameter-estimation problem, where the unknown parameters enter a nonlinear excitable-dynamics model. Although this forward model is low-dimensional, the associated least-squares objective has a strongly nonconvex landscape and is sensitive to initialization~\cite{villalobos2026neural}.

Rather than constructing a reduced model for the full state, we approximate directly the parameter-to-observation map. The corresponding observation tensor is compressed in TT format using either TT-SVD~\cite{oseledets2011tensor} or interpolation techniques based on TT-Cross and AMEn-type ideas~\cite{oseledets2010tt,dolgov2015polynomial}. The resulting surrogate is then used inside a regularized nonlinear least-squares formulation of the inverse problem.

We furthermore develop an error-to-inversion analysis for the TROM inverse problem. Under local strong convexity of the regularized FOM objective and suitable parametric smoothness of the FOM observation map, the analysis bounds the perturbation of the recovered parameter caused by tensor compression and interpolation. It thereby separates the approximation accuracy of the TROM from the local stability of the inverse problem.

\paragraph{Contributions}
Existing TROMs for parametric dynamical systems are used mainly to accelerate forward solves; we instead exploit the tensor structure throughout the inverse solve:
\begin{enumerate}[leftmargin=2.2em,label={C\arabic*.}]
\item \emph{An inverse-problem-tailored TROM.} The reduced model approximates the para\-meter-to-ob\-ser\-vation map directly, rather than the full state.
\item \emph{Reduced-coordinate Gauss--Newton.} Beyond fast forward evaluation, the TT representation enables a reduced-coordinate formulation of the inverse problem and allows the GN quantities to be assembled directly from the low-rank structure of the parameter-to-observation map.
\item \emph{TT-based objective minimization.} We minimize the objective directly over the discrete parameter grid in TT format, yielding either a stand-alone approximate minimizer or a data-informed initial guess for a subsequent GN solve.
\item \emph{Error-to-inversion analysis.} For the  TROM inverse problem, we bound the perturbation of the recovered parameter in terms of the TROM map and Jacobian errors and local FOM stability quantities. We then estimate the two TROM errors from the tensor-compression and interpolation errors, using a max-fiber tensor norm to obtain bounds that remain meaningful under refinement of the parameter grid.
\item \emph{Comprehensive numerical study.} On two inverse problems, we assess robustness to TROM error and measurement noise, and the effects of regularization, initialization, spatial discretization, increasing parameter dimension, and nonconvex objective landscapes; compare TT-SVD with TT-Cross and TROM-based inversion with full-order inversion; quantify the benefit of TT-based optimization and initialization; and show that online TROM inversion is cheap relative to the offline surrogate construction.
\end{enumerate}

\paragraph{Other Related Work}
Surrogate models have long been used to make global optimization of expensive objectives tractable~\cite{jones1998efficient,jones2001taxonomy,shahriari2016taking}, and, closer to the present setting, low-rank tensor surrogates have been combined with global search and subsequent local refinement for parameter estimation and inverse problems~\cite{sozykin2022ttopt,krivorotko2020global,zheltkova2018tensor}. Those methods build a low-rank representation of a \emph{scalar objective}, which is tied to one data set and must be rebuilt, at the cost of a full-order solve per sampled entry, whenever the data change. Because we compress the parameter-to-observation \emph{map} instead, the objective tensor for any new observation vector is assembled in closed form from the same TT cores, so that a further inversion requires no additional full-order solves: only the reassembly of a single TT core, one rounding sweep, and the discrete search are repeated. The same distinction underlies the error analysis of \Cref{sec:error-to-inversion}, which bounds the recovered parameter in terms of the compression error of the map rather than of a data-specific objective. A complementary line of work manages the fidelity of an approximation \emph{during} a local solve rather than searching globally: trust-region and surrogate-management frameworks adapt, enrich, or certify the reduced model as the iterates move~\cite{alexandrov1998trust,booker1999rigorous,zahr2015progressive,yue2013accelerating}, with~\cite{forrester2009recent} as a survey, and certified reduced-basis versions supply a posteriori bounds for the reduced optimization problem in PDE-constrained optimization and parameter identification~\cite{dihlmann2015certified,qian2017certified,kartmann2024adaptive}. The TROM used here is instead built once, offline, and held fixed throughout the inversion.

Low-rank tensor techniques have also been used in deterministic and statistical inverse problems. For the latter, TT approximations have been used to represent high-dimensional probability densities and accelerate posterior sampling~\cite{dolgov2020approximation}. Related low-rank tensor reconstructions combined with transport maps have been developed for concentrated posterior densities arising in Bayesian inversion~\cite{eigel2022lowrank}. This line of work has been extended to deep, composed TT transport maps for high-dimensional and conditional Bayesian inference~\cite{cui2022deep,cui2023scalable}. Tensor methods have also been incorporated into ROM construction for parametric PDEs and dynamical systems, including interpolatory, projection-based, and non-intrusive TROMs~\cite{mamonov2022interpolatory,mamonov2024tensorial,mizan2026parametric,vijaywargiya2026structure,islam2026tensorial} and TT reduced-basis methods~\cite{mueller2026tensor}. Separated representations of the same kind underlie the proper generalized decomposition, which constructs low-rank approximations of parametric and time-dependent PDE solutions a priori, without requiring a precomputed sampling of the solution manifold~\cite{ladeveze2010latin,nouy2010priori}. As with the TROMs above, these representations approximate the state; here we compress the parameter-to-observation map instead. Another active direction is nonlinear manifold reduction, where encoder--decoder architectures can be advantageous for handling solution manifolds with slowly decaying Kolmogorov width~\cite{greif2019decay}; examples include convolutional-autoencoder ROMs~\cite{lee2020model}, deep-learning ROMs for parametrized PDEs~\cite{fresca2021comprehensive}, and neural-network ROMs for inverse problems and parameter identification~\cite{ivagnes2023towards}. Closely related to our use of the surrogate for Jacobian information, derivative-informed neural network surrogates of parameter-to-observation maps have been developed for outer-loop applications such as Bayesian inversion~\cite{olearyroseberry2022derivative}. Finally, several works show that inverse solvers can benefit from ROM constructions designed specifically for inference, not only from faster forward evaluations: examples include posterior-adaptive reduced bases for Bayesian inversion~\cite{cui2015datadriven} and goal-oriented inference methods that build reduced spaces around the prediction or identification task~\cite{lieberman2013goal}.

\paragraph{Notation}
Throughout, let $g : \mathcal{A}\subset\mathbb{R}^D \to \mathcal{U}$ denote the forward solution operator and let $q:\mathcal{U}\to\mathcal{Y}$ be the observation operator. Here $\mathcal{A}$ is a parameter domain, $\mathcal{U}$ and $\mathcal{Y}$ are Banach spaces. In the examples below, $g$ denotes the solution operator of either a parametric PDE or a parametric ODE system. The parameter-to-observation map is $f \defeq q\circ g$, $f:\mathcal{A}\to\mathcal{Y}$. Given noisy data $y_{\mathrm{obs}}\in\mathcal{Y}$, we assume $y_{\mathrm{obs}} = f(\alpha_{\mathrm{true}})+\varepsilon_{\mathrm{meas}}$, where $\varepsilon_{\mathrm{meas}}$ denotes measurement noise. In the inversion procedure, the full-order map $f$ is replaced by a tensor reduced-order approximation $\widehat f$, which introduces a TROM error $\varepsilon_{\mathrm{rom}}$, $\widehat f(\alpha) = f(\alpha)+\varepsilon_{\mathrm{rom}}(\alpha)$. Thus, the effective data misfit contains both measurement and ROM contributions. This distinction is used later in the choice of discrepancy targets and regularization parameters.

\paragraph{Outline}
The remainder of the paper is organized as follows.  \Cref{sec:ROM} describes the construction of the TROM, including TT-SVD and TT-Cross compression and fast online surrogate evaluation. \Cref{sec:Inv} presents the ROM-based inversion procedure. It also describes a deeper integration of the TROM into the inverse solve through reduced-coordinate TROM minimization, efficient assembly of the GN quantities, and the TROM-based construction of an initial guess.
\Cref{sec:error-to-inversion} develops the error-to-inversion estimate and bounds the required uniform TROM map and Jacobian errors in terms of interpolation and tensor-compression quantities.
\Cref{s:methods} formulates two examples of the parametric forward and inverse problems.
\Cref{s:results} reports the numerical study, and \Cref{s:conclusions} summarizes the main findings and discusses limitations and future directions.

\section{Reduced Order Model}\label{sec:ROM}

We construct a ROM directly for the parameter-to-observation map $f = q \circ g$, rather than for the full state associated with the underlying forward problem. This choice is natural for inverse problems in which the optimization procedure requires only repeated evaluations of $f(\alpha)$ and its parameter sensitivities.

After a suitable discretization, we assume that the observation space $\mathcal{Y}$ is finite-dimensional with dimension $N_{\mathrm{obs}}$, and write $f:\mathcal{A}\to \mathbb{R}^{N_{\mathrm{obs}}}$. We further assume that the admissible parameter set is a box, $\mathcal{A} = \bigotimes_{i=1}^D [\alpha^{\mathrm{min}}_i,\alpha^{\mathrm{max}}_i]\subset\mathbb R^D$.

\subsection{Offline Part}

The set of all possible observations, i.e., the image of $\mathcal{A}$ under $f$, is a parametric manifold in $\mathbb{R}^{N_{\mathrm{obs}}}$. We start with a discrete full-order representation of this manifold. To this end, we first discretize the parameter domain $\mathcal{A}$ using a Cartesian grid $\hcA$: we distribute $N_i$ nodes $\{\halpha_i^j\}_{j=1,\dots,N_i}$ within each of the intervals $[\alpha_i^{\mathrm{min}}, \alpha_i^{\mathrm{max}}]$ for $i=1,\dots,D$, and let
\begin{equation}
\label{eqn:grid}
\hcA = \left\{ \halpha = (\halpha_1,\dots,\halpha_D)^\mathsf{T}\,:\,
\halpha_i \in \{\halpha_i^n\}_{n=1}^{N_i}, ~ i = 1,\dots,D \right\} .
\end{equation}

\noindent All observations for the parameters in $\hcA$ are organized in the \emph{multi-dimensional} array
\begin{equation}\label{eqn:snapmulti}
\mathcal{T}(\,:\,,n_1,\dots,n_D) = f(\halpha_1^{n_1},\dots,\halpha_D^{n_D}),
\end{equation}

\noindent which is a tensor of order $D+1$ and size $N_{\mathrm{obs}} \times N_1 \times \dots \times N_D$. The first mode corresponds to the discretized observation vector, while the remaining modes correspond to the parameter grid. Thus, we may consider the $(D+1)$-way tensor $\mathcal{T} \in \mathbb{R}^{N_{\mathrm{obs}} \times N_1 \times \cdots \times N_D}$, as a discrete representation of the parametric observation manifold.

The tensor $\mathcal{T}$ is a full-order object, and for ROM purposes we seek a compressed (reduced-order) approximation of $\mathcal{T}$. In this paper, such an approximation is constructed in the TT format: $\mathcal{T} \approx \widehat{\mathcal{T}}$, where
\begin{equation}\label{TT}
\widehat{\mathcal{T}}(n_0,n_1,\ldots,n_D)
=
G_{1}(1,n_0,\,:\,)\,
G_{2}(\,:\,,n_1,\,:\,)\cdots
G_{D+1}(\,:\,,n_D,1),
\end{equation}

\noindent $n_i=1,\dots, N_i,\; i=0, \dots, D$, with $N_0 = N_{\mathrm{obs}}$ and third-order (3-way) core tensors  $G_{1}, G_{2}, \ldots, G_{D+1}$ such that $G_{k} \in \mathbb{R}^{r_k \times N_{k-1} \times r_{k+1}}$. The product of the $G$-terms in \eqref{TT} is understood as a product of $r_k \times r_{k+1}$ matrices. The TT-ranks $r_1 = 1,$ $r_2, \ldots, r_{D+1},$ $r_{D+2} = 1$ define the compression. We shall also write a TT representation~\eqref{TT} using the compact notation
\[
\widehat{\mathcal{T}} = G_{1}\circ G_{2} \circ\dots\circ G_{D+1}.
\]

To find a suitable TT-tensor $\widehat{\mathcal{T}}$, we adopt two approaches.
The first requires assembling the complete tensor $\mathcal{T}$, followed by a truncated singular value decomposition (SVD) applied to a sequence of matricizations of the tensor (TT-SVD approach).
In the second approach, a compressed representation of $\mathcal{T}$ is computed using the TT-Cross interpolation procedure, which in general requires access only to a (small) subset of the entries of $\mathcal{T}$.
The latter avoids expensive offline computations.

\subsubsection{TT-SVD for Order Reduction}

In this approach, the TT tensor $\widehat{\mathcal{T}}$ is constructed by computing truncated SVDs of a sequence of matricizations of $\mathcal{T}$, i.e., matrices of size $(N_0 \cdots N_k) \times (N_{k+1} \cdots N_{D})$ obtained by unfolding $\mathcal{T}$; see~\cite{oseledets2011tensor}. If the truncation criterion in each SVD is defined by the threshold $\varepsilon_{\mathrm{TT}}/\sqrt{D}$, $\varepsilon_{\mathrm{TT}} > 0$, on the relative magnitude of the singular values, then the resulting TT tensor satisfies
\begin{equation}\label{TTapprox}
\|{\mathcal{T}} - \widehat{\mathcal{T}}\|_F \le \varepsilon_{\mathrm{TT}} \|{\mathcal{T}}\|_F.
\end{equation}

\noindent Note that TT-SVD generally requires access to all entries of $\mathcal{T}$. For high parameter dimensions, computing the full tensor $\mathcal{T}$ may be prohibitively expensive.

\subsubsection{TT-Cross}\label{sec:Cross}
The TT-Cross method constructs a TT approximation $\widehat{\mathcal{T}}$ of $\mathcal{T}$ by adaptively sampling only a small subset of its entries, thereby avoiding full tensor assembly. It combines ideas from TT-Cross interpolation~\cite{oseledets2010tt} with alternating minimal energy (AMEn) enrichment~\cite{dolgov2014alternating}, iteratively updating TT cores while selecting informative multi-index sets via max-volume principles~\cite{goreinov2010submatrix,savostyanov2014quasioptimality}. We employ the \texttt{amen\_cross} routine from~\cite{dolgov2015polynomial}: at each iteration, local low-rank approximations are enriched by residual-based directions, and a prescribed tolerance $\varepsilon_{\mathrm{TTC}} > 0$ controls the stopping criterion and adaptive rank growth, so the ranks are determined automatically to meet the requested tolerance.

To reduce the offline computational cost, one may precede TT-Cross by a preprocessing step that reduces the first dimension, $N_{\mathrm{obs}}$, of the tensor to be approximated. One way of doing this is to find a matrix $P\in\mathbb{R}^{N_0\times N_{\mathrm{obs}}}$, with reduced dimension $N_0<N_{\mathrm{obs}}$, such that
\begin{equation}\label{eq:P}
\sup_{\halpha\in\hcA} \frac{\|(I-P^\mathsf{T}P)f(\halpha)\|}{\|f(\halpha)\|} \le
\varepsilon_{\mathrm{red}},
\end{equation}

\noindent for a sufficiently small tolerance $\varepsilon_{\mathrm{red}} > 0$. Here, we assume that the rows of $P$ are orthonormal. In practice, $P$ is constructed by a truncated SVD of a training subset of observations; see \Cref{sec:Pre}.  Alternatively, one may learn a reduced latent representation with a nonlinear encoder--decoder architecture.

The dimension reduction is performed by mapping the observations into $\mathbb{R}^{N_0}$, i.e., $f(\halpha) \mapsto P f(\halpha)$. This enables TT-Cross interpolation to be applied to a tensor of size
$N_0 \times N_1 \times \cdots \times N_D$,
instead of the original tensor of size
$N_{\mathrm{obs}} \times N_1 \times \cdots \times N_D$.

The resulting TT-Cross surrogate thus combines two levels of compression: \begin{enumerate*}[label=\it (\roman*)] \item a low-rank reduction in the observation space and \item a tensor-train approximation of the reduced coefficient map over the parameter space\end{enumerate*}. The TT tensor of the original size
$N_{\mathrm{obs}} \times N_1 \times \cdots \times N_D$
can be reconstructed by applying a tensor--matrix product with $P^\mathsf{T}$ along the first mode.

To simplify notation, we assume
\[
N_0=N_{\mathrm{obs}}
\]

\noindent and let $\widehat{\mathcal T}$ denote the TT tensor of size $N_{\mathrm{obs}} \times N_1 \times \cdots \times N_D$.

\subsubsection{Preprocessing step}\label{sec:Pre}

This is done with the help of SVD on a representative collection of observations. To this end, we randomly select training and test subsets of points from the discrete parameter domain:
\[
\mathcal{I}_{\mathrm{train}} \subset \hcA,
\quad
P_{\mathrm{train}} \defeq |\mathcal{I}_{\mathrm{train}}|,
\qquad
\mathcal{I}_{\mathrm{test}} \subset \hcA,
\quad
P_{\mathrm{test}} \defeq |\mathcal{I}_{\mathrm{test}}|.
\]

By solving the forward problem for each parameter in $\mathcal{I}_{\mathrm{train}}$, we assemble a matrix of observations
$
S \in \mathbb{R}^{N_{\mathrm{obs}} \times P_{\mathrm{train}}},
$
and compute a truncated SVD of $S$ with prescribed tolerance $\varepsilon_{\mathrm{red}}$, $S \approx U_S \Sigma V_S^\mathsf{T}$, $U_S \in \mathbb{R}^{N_{\mathrm{obs}} \times N_0}$. We let $P= U_S^\mathsf{T}$.  To ensure that $P$ satisfies \eqref{eq:P}, we validate the inequality on the held-out set $\mathcal{I}_{\mathrm{test}}$ and enrich the training set if necessary.

\subsection{Online Part}

Next, we describe the online evaluation of the TT-ROM model outlined above.

\subsubsection{ROM Evaluation}
\label{sec:core_contraction}

To define the fast (reduced-dimension) evaluation of the parameter-to-observation mapping $f = q \circ g$, we follow~\cite{mamonov2022interpolatory} and assume an interpolation procedure
\begin{equation}\label{eqn:bea}
  \chi_i \,:\, [\alpha_i^{\mathrm{min}}, \alpha_i^{\mathrm{max}}] \to \mathbb{R}^{N_i},\quad i=1,\dots,D,
\end{equation}

\noindent such that for any continuous function $g : \mathcal{A} \to \mathbb{R}$,
\begin{equation}\label{Interp}
  I(g)(\alpha)= \sum_{n_1=1}^{N_1}\cdots\sum_{n_D=1}^{N_D}
\big(\chi_1(\alpha_1)\big)_{n_1}
\cdots
\big(\chi_D(\alpha_D)\big)_{n_D}
  g\big(\halpha_1^{n_1},\ldots,\halpha_D^{n_D}\big)
\end{equation}

\noindent defines an interpolant of $g$.

One straightforward choice is Lagrange interpolation of order $p$. For any $\alpha_i \in [\alpha_i^{\mathrm{min}}, \alpha_i^{\mathrm{max}}]$, let $\widehat{\alpha}_i^{i_1}, \ldots, \widehat{\alpha}_i^{i_p}$ be the $p$ closest grid nodes to $\alpha_i$ on $[\alpha_i^{\mathrm{min}}, \alpha_i^{\mathrm{max}}]$, for $i=1,\ldots,D$. Then
\begin{equation}
\label{eqn:lagrange}
\big(\chi_i (\alpha_i)\big)_j =
\begin{cases}
  \displaystyle
  \prod\limits_{\substack{m = 1 \\ m \neq k}}^{p}
  \frac{\widehat{\alpha}_i^{i_m}-\alpha_i}{\widehat{\alpha}_i^{i_m}-\widehat{\alpha}_i^{j}},
  & \text{if } j = i_k \in \{i_1,\ldots,i_p\}, \\
  0, & \text{otherwise},
\end{cases}
\end{equation}

\noindent for $j=1,\dots,N_i$. Other suitable choices include continuous piecewise-polynomial interpolants and splines, depending on the desired degree of smoothness.

With the help of~\eqref{eqn:bea}, we define the ROM approximation to the observations as
\begin{equation}
\label{eqn:extractbt}
\widehat{f}(\alpha) = \widehat{\mathcal{T}}
  \times_2 \chi_1(\alpha_1) \times_3 \chi_2(\alpha_2) \cdots \times_{D+1} \chi_D(\alpha_D)
  \in \mathbb{R}^{N_{\mathrm{obs}}},
\end{equation}

\noindent where $\times_k$ denotes the $k$-mode tensor--vector product~\cite{kolda2009tensor}.

If $\alpha \in \hcA$ belongs to the sampling set, then $\chi_i(\alpha_i)$ encodes the position of $\alpha_i$ among the grid nodes on $[\alpha^{\mathrm{min}}_i, \alpha^{\mathrm{max}}_i]$. In this case, $\widehat{f}(\alpha)$ coincides with the observation vector for the given $\alpha$, up to the tensor compression error. For a general $\alpha\in \mathcal{A}$, the vector $\widehat{f}(\alpha)$ is obtained by interpolation of precomputed observations. An important point is that the computations in~\eqref{eqn:extractbt} can be carried out directly in a low-rank format in a fast and efficient manner, as explained below.

Given the TT cores $G_i$ and the interpolation procedures $\chi_i$, let us define the $\alpha$-specific \emph{core vector} $c_\chi(\alpha) \in \mathbb{R}^{r_2}$ as
\begin{equation}
\label{eqn:C_TT}
c_\chi (\alpha) = \prod_{i=1}^{D}  G_{i+1} \times_2 \chi_{i} (\alpha_i).
\end{equation}

\noindent Using the definition of the $k$-mode product, one finds that $\widehat{f}$ from~\eqref{eqn:extractbt}
can be computed as:
\begin{equation}
\label{eqn:f_TT}
\widehat{f} (\alpha) = G_1\, c_\chi(\alpha),
\end{equation}

\noindent where $G_1$ is interpreted as an $N_{\mathrm{obs}}\times r_2$ matrix rather than a $1\times N_{\mathrm{obs}}\times r_2$ tensor.

All computations in \eqref{eqn:C_TT} and \eqref{eqn:f_TT} involve only low-dimensional objects. Moreover, the vectors $\chi_i (\alpha_i)$ have only a few (in our case, $p$) nonzero entries, which makes the online evaluation fast and efficient. \emph{In summary, the surrogate is evaluated by local interpolation in each parameter direction followed by contraction of the TT cores, without reconstructing the full tensor.}

\section{Numerical Inversion and TROM integration} \label{sec:Inv}

We consider a classical variational approach: Given a vector of observations $y_{\mathrm{obs}}$, find geometric parameters $\alpha\in\mathbb{R}^D$ by minimizing the regularized least-squares functional over the parameter domain $\mathcal{A}\subset\mathbb{R}^D$:
\begin{equation}
\minopt_{\alpha\in\mathcal{A}}\left\{\tfrac{1}{2}\,\,
\|f(\alpha)-y_{\mathrm{obs}}\|_2^2 + \tfrac{\lambda}{2}\|\alpha-\alpha_{\mathrm{ref}}\|_2^2\right\}.
\label{eq:inverse_problem}
\end{equation}

\noindent The first term measures the discrepancy between the model prediction $f(\alpha)$ for a candidate $\alpha$ and the observations $y_{\mathrm{obs}}$, while the second term is a Tikhonov-type regularization; $\alpha_{\mathrm{ref}}\in\mathbb{R}^D$ is a reference parameter (prior mean), and $\lambda>0$ is the regularization parameter. Notice that we eliminate the constraint~\eqref{eq:heat} from the formulation and insert the parameter-to-observation map $f$ into the data mismatch term (reduced formulation). The map $f$ is evaluated either by the FOM or by the TROM surrogate ($f = \widehat f$ in this case), depending on the setting.

Problem~\eqref{eq:inverse_problem} is solved by a projected GN method. At iteration $k$, given the current iterate $\alpha_k$, we define the residual
$\eta_k = y_{\mathrm{obs}}-f(\alpha_k)$
and linearize the forward map:
\[
f(\alpha_k + s) \approx f(\alpha_k)+J_k s,
\qquad
J_k=\frac{\partial f}{\partial \alpha}(\alpha_k),
\]

\noindent for arbitrary $s \in \mathbb{R}^D$. Substituting this approximation into \eqref{eq:inverse_problem} gives the GN system
\begin{equation}
\bigl(J_k^\mathsf{T} J_k+\lambda I\bigr)s_k
=
J_k^\mathsf{T} \eta_k-\lambda(\alpha_k-\alpha_{\mathrm{ref}}).
\label{eq:gn_system}
\end{equation}

\noindent We solve~\eqref{eq:gn_system} for the search direction $s_k \in \mathbb{R}^D$; the trial update is then projected onto the admissible box:
$\alpha_{k+1} \gets \Pi_{\mathcal{A}}(\alpha_k+t_ks_k)$,
where the stepsize $t_k>0$ is chosen by Armijo backtracking to ensure sufficient decrease of the regularized objective~\cite[p.~33]{nocedal2006numerical}.

The iterations are terminated if we do not make sufficient progress between iterations: We stop the iteration if
\[
\frac{\|\alpha_{k+1}-\alpha_k\|_2}{\max(1,\|\alpha_k\|_2)} \le \varepsilon_{\mathrm{stop}},
\]

\noindent or the relative decrease in the objective satisfies
\[
\frac{|\Phi_k-\Phi_{k+1}|}{\max(1,\Phi_k)} \le \varepsilon_{\mathrm{stop}},\quad \text{with}~
\Phi_k = \tfrac{1}{2}\|f(\alpha_k)-y_{\mathrm{obs}}\|_2^2 + \tfrac{\lambda}{2}\|\alpha_k-\alpha_{\mathrm{ref}}\|_2^2.
\]

\noindent We select $\varepsilon_{\mathrm{stop}}=10^{-8}$ for all our experiments.

In the TROM setting, the matrices $J_k^\mathsf{T} J_k$ and vectors $J_k^\mathsf{T}\eta_k$ in~\eqref{eq:gn_system} are assembled using the TT contraction formulas for surrogate evaluation and implicit Jacobian construction described above, without forming a full Jacobian in the observation space (see \Cref{s:trom-jacobian}).
The regularization parameter $\lambda$ is chosen based on the discrepancy principle~\cite{engl1996regularization,hansen2010discrete,kaipio2005statistical} (see also \Cref{sec:regularization-parameter}).

\subsection{Reduced TROM objective}
In the TROM inversion, the forward map in the minimization problem~\eqref{eq:inverse_problem} is given by the surrogate $\widehat f$ in~\eqref{eqn:f_TT}. Consistently with~\eqref{eqn:f_TT}, we consider $G_1\in \mathbb{R}^{N_{\mathrm{obs}}\times r_2}$ as the matrix rather than $1\times N_{\mathrm{obs}}\times r_2$ tensor.

Without loss of generality, we assume that the columns of $G_1$ are orthonormal. For TT-SVD this is obtained when the TT tensor is left-orthogonalized; for a general TT representation, including TT-Cross, it can be enforced by a QR factorization of $G_1$ and absorption of the triangular factor into the second TT core.

Let $\Pi = G_1(G_1)^\mathsf{T}$, $\Pi^\perp = I-\Pi$, be the orthogonal projectors onto the range of $G_1$ and its orthogonal complement, respectively. Since $\widehat f(\alpha)\in\operatorname{range}(G_1)$ for all $\alpha$, we have the orthogonal decomposition
\[
\|\widehat f(\alpha)-y_{\mathrm{obs}}\|_2^2
=
\|\widehat f(\alpha)-\Pi y_{\mathrm{obs}}\|_2^2
+
\|\Pi^\perp y_{\mathrm{obs}}\|_2^2.
\]

\noindent The second term is independent of $\alpha$ and therefore does not affect the minimizer. Moreover, using the orthonormality of the columns of $G_1$,
\[
\|\widehat f(\alpha)-\Pi y_{\mathrm{obs}}\|_2^2
=
\|G_1c_\chi(\alpha)
-
G_1(G_1)^\mathsf{T}y_{\mathrm{obs}}\|_2^2
=
\|c_\chi(\alpha)-y_{\mathrm{red}}\|_2^2,
\]

\noindent where  $c_\chi(\alpha)$ is defined in \eqref{eqn:C_TT}, and
\begin{equation}\label{eq:y_red}
y_{\mathrm{red}}
=
(G_1)^\mathsf{T}y_{\mathrm{obs}}
\in\mathbb R^{r_2}.
\end{equation}

\noindent Hence, for the TROM forward solve, the minimization problem~\eqref{eq:inverse_problem} is equivalent to the reduced problem
\begin{equation}
\minopt_{\alpha\in\mathcal{A}}\,\,
\left\{
\tfrac{1}{2}
\|c_\chi(\alpha)-y_{\mathrm{red}}\|_2^2
+
\tfrac{\lambda}{2}\|\alpha-\alpha_{\mathrm{ref}}\|_2^2
\right\}.
\label{eq:inverse_problem_TROM}
\end{equation}

This has an important computational implication: once $y_{\mathrm{red}}$ is computed, the GN iteration runs in the reduced coordinate space of dimension $r_2$ rather than the full observation space of dimension $N_{\mathrm{obs}}$. In particular, the reduced residual is $y_{\mathrm{red}} - c_\chi(\alpha_k)\in\mathbb R^{r_2}$, and the corresponding reduced Jacobian has size $r_2\times D$. Thus, after projecting the full-order observation vector in \eqref{eq:y_red}, the online linear-algebra operations depend on the second TT rank $r_2$, which can be much smaller than $N_{\mathrm{obs}}$.

\subsection{Regularization parameter selection}\label{sec:regularization-parameter}

For the regularization parameter selection we consider the discrepancy principle. We construct the observations $y_{\mathrm{obs}}$ using forward simulations. We differentiate between noise contributions due to measurement noise and due to ROM approximation errors. We refer to \Cref{sec:discrepancy} for additional details.

\subsection{TROM Jacobian and Gauss–Newton quantities}\label{s:trom-jacobian}

We now describe the Jacobian of the reduced map $c_\chi$ appearing in~\eqref{eq:inverse_problem_TROM}. Since $\widehat f(\alpha)=G_1c_\chi(\alpha)$ and the columns of $G_1$ are orthonormal, \(c_\chi(\alpha)=G_1^\mathsf{T}\widehat f(\alpha)\). The TROM allows for its efficient computation  by exploiting the contraction structure in~\eqref{eqn:C_TT}. This provides the sensitivity information required for numerical optimization, while ensuring that the computational cost depends on the TT ranks and the reduced dimensions, rather than on the full observation dimension. Indeed, \eqref{eqn:C_TT} and~\eqref{eqn:f_TT} yield
\begin{equation}
\label{eqn:Jac}
\frac{\partial \widehat{f}}{\partial \alpha_d}(\alpha)
=
G_1
\left(
\prod_{i=1}^{d-1} G_{i+1} \times_2 \chi_i(\alpha_i)
\right)
G_{d+1} \times_2 \chi'_d(\alpha_d)
\left(
\prod_{i=d+1}^{D} G_{i+1} \times_2 \chi_i(\alpha_i)
\right),
\end{equation}

\noindent for $d = 1, \dots, D$. Like $\chi_d(\alpha_d)$, the vector $\chi'_d(\alpha_d)$ has only a few nonzero entries, computed explicitly from~\eqref{eqn:lagrange}.

To evaluate~\eqref{eqn:Jac} more efficiently, we introduce prefix and suffix products. Let
$B_d(\alpha_d) = G_{d+1} \times_2 \chi_d(\alpha_d)$,
$\dot{B}_d(\alpha_d) = G_{d+1} \times_2 \chi'_d(\alpha_d)$,
and define
\begin{equation}
\label{eqn:interface}
B^{<d}(\alpha)=\prod_{i<d}B_i(\alpha_i),
\qquad
B^{>d}(\alpha)=\prod_{i>d}B_i(\alpha_i),
\qquad d = 1,\dots,D,
\end{equation}

\noindent with the empty products $B^{<1}=I_{r_2}$ and $B^{>D}=1$, evaluated by the recursions $B^{<d+1}=B^{<d}B_d(\alpha_d)$ and $B^{>d-1}=B_d(\alpha_d)B^{>d}$.
With this notation, we obtain
\[
\frac{\partial c_\chi}{\partial \alpha_d}(\alpha)
=
B^{<d} \dot{B}_d(\alpha_d) B^{>d}
\in \mathbb{R}^{r_2},
\qquad d = 1, \dots, D.
\]

We estimate the cost as follows. Once the matrices $B_d(\alpha_d)$ are formed, the prefix and suffix products $B^{<d}$, $B^{>d}$ are computed once and reused across parameter directions. If the TT ranks are bounded by $r_{\mathrm{max}}\in\mathbb{N}$, then forming all prefixes and suffixes costs $O(D r_{\mathrm{max}}^3)$. For each parameter direction, assembling one Jacobian column costs $O(r_{\mathrm{max}}^{{2}})$. Therefore, the total cost of assembling the ``reduced'' Jacobian
\[
J = J(\alpha) =
\left[
\frac{\partial c_\chi}{\partial \alpha_1}(\alpha),
\dots,
\frac{\partial c_\chi}{\partial \alpha_D}(\alpha)
\right]
\in  \mathbb{R}^{r_2 \times D}.
\]

\noindent is simply $O\left(D r_{\mathrm{max}}^3\right)$.

If explicit assembling of the Jacobian of the surrogate model is the goal, then one may proceed with
$\frac{\partial \widehat{f}}{\partial \alpha_d}(\alpha)
= G_1
\frac{\partial c_\chi}{\partial \alpha_d}(\alpha)$,
$d = 1, \dots, D$. In this case assembling one Jacobian column costs $O(N_{\mathrm{obs}}r_2+r_{\mathrm{max}}^{{2}})$, so that the total cost increases to $O(D r_{\mathrm{max}}^{3}+D N_{\mathrm{obs}}r_2)$.

In the GN method the Jacobian enters through the quantities $J^\mathsf{T}J$ and $J^\mathsf{T}\eta$, which adds $O\left(r_2 D^2\right)$ and $O\left(r_2 D\right)$, respectively. This brings us to the total cost per iteration\footnote{For the original minimization problem~\eqref{eq:inverse_problem} the TROM  alone still allows one to reduce the cost per iteration to $O\left(
D r_{\mathrm{max}}^3
+
N_{\mathrm{obs}}r_2
+
D^2r_2
\right)$ through assembling of $J^\mathsf{T}J$ and $J^\mathsf{T}\eta$ directly using representation~\eqref{eqn:Jac} avoiding the explicit formation of the Jacobian matrix.}
\begin{equation}\label{eqn:cost}
O\left(D r_{\mathrm{max}}^3 + r_2 D^2\right).
\end{equation}

This can be contrasted with the standard finite-difference (FD) approximation of the Jacobian in the full observation space. For example, in the interior of the parameter domain, the $d$-th Jacobian column can be approximated by
\[
\frac{\partial \widehat{f}(\alpha)}{\partial \alpha_d}
\approx
\frac{\widehat{f}(\alpha + h_d e_d) - \widehat{f}(\alpha - h_d e_d)}{2 h_d},
\qquad d = 1, \dots, D.
\]

\noindent Thus, finite differences require additional model evaluations for each parameter direction. By the same reasoning as above, one evaluation of $\widehat{f}$ has cost $O(D r_{\mathrm{max}}^2 + N_{\mathrm{obs}} r_{2})$. Therefore, the standard FD approach to Jacobian computation incurs a higher computational cost $O(D^2 r_{\mathrm{max}}^2 + D N_{\mathrm{obs}} r_2)$.

Computing the Gauss–Newton quantities $J^\mathsf{T}J$ and $J^\mathsf{T}\eta$ would cost additional $O(D^2 N_{\mathrm{obs}})$ and $O(D N_{\mathrm{obs}})$ operations. Adding it to the cost of forming the Jacobian  we find that the total cost of forming the GN system would be
\begin{equation}\label{eq:cost1}
O(D^2 (r_{\mathrm{max}}^2 +N_{\mathrm{obs}}) + D N_{\mathrm{obs}} r_2).
\end{equation}

\noindent Depending on the TT ranks, parametric and objective spaces dimensions this may become significantly higher than~\eqref{eqn:cost}.

\subsection{TROM optimization and initialization}
\label{sec:tt_initial_guess}

The TROM representation can be used either to minimize the objective functional directly over the discrete parameter domain or to construct a data-informed initial guess for a subsequent iterative inversion method. The preferable use is problem-dependent: the tensor-based minimizer may already provide an accurate reconstruction, or it may serve only to initialize a local method such as GN.

At this stage, we do not search over the full admissible parameter space $\mathcal A$. Instead, we minimize the TROM objective over the discrete parameter domain $\hcA$ defined in~\eqref{eqn:grid}. Specifically, for a given observation $y_{\mathrm{obs}}\in\mathbb R^{N_{\mathrm{obs}}}$, we define
\begin{equation}\label{initialGuess}
\alpha_\ast
= \argmin_{\halpha\in\hcA}\,\left\{\tfrac{1}{2}\,\|\widehat{f}(\halpha)-y_{\mathrm{obs}}\|_2^2 +  \tfrac{\lambda}{2}\|\halpha-\alpha_{\mathrm{ref}}\|_2^2\right\}.
\end{equation}

\noindent The parameter $\alpha_\ast$ can then be used either as the TROM-based reconstruction itself or as the initial guess $\alpha^{(0)}=\alpha_\ast$ for a subsequent GN iteration.

Naively, one could evaluate the objective in~\eqref{initialGuess} at every parameter value in $\hcA$. Instead, we exploit that the TROM provides a low-rank representation of all $f(\alpha)$, $\alpha\in\hcA$, and that the Tikhonov term is a sum of univariate (quadratic) functions. To describe the minimization procedure, consider the set of multi-indices
$\mathcal I
=
\{1,\ldots,N_1\}\times\cdots\times\{1,\ldots,N_D\}$.
For each $\mathbf n=(n_1,\ldots,n_D)\in\mathcal I$, we denote the corresponding parameter vector by
\[
\alpha_{\mathbf n}
=
\bigl(
\halpha_1^{\,n_1},\ldots,\halpha_D^{\,n_D}
\bigr)^{\mathsf{T}}
\in \widehat{\mathcal A}.
\]

\noindent Next, define the residual and regularization tensors,
\[
\mathcal R(\mathbf n)
= \tfrac{1}{2}\|\widehat{f}(\alpha_{\mathbf n})-y_{\mathrm{obs}}\|_2^2,
\quad
\mathcal L(\mathbf n)
= \tfrac{\lambda}{2}\|\alpha_{\mathbf n}-\alpha_{\mathrm{ref}}\|_2^2,
\quad
\mathcal R, \mathcal L\in\mathbb{R}^{N_1\times\cdots\times N_D}.
\]

\noindent Solving~\eqref{initialGuess} is therefore equivalent to finding the minimum entry of $\mathcal J = \mathcal R  +\mathcal L$. To efficiently solve this minimization problem, we first represent $\mathcal J$ in TT format using elementary multilinear algebra as explained below.

We start with $\mathcal R$. Consider the rank-one tensor
\[
\mathcal S
=
y_{\mathrm{obs}}\otimes \mathbf{1}_{N_1}\otimes\cdots\otimes \mathbf{1}_{N_D},
\]
\noindent where $\mathbf{1}_{N}=(1,\dots,1)^\mathsf{T}\in\mathbb{R}^{N}$,
and define
\[
\mathcal{E} \defeq \widehat{\mathcal T}-\mathcal S.
\]
The residual tensor can be written as
\begin{equation}\label{eq:R}
	\mathcal R=\tfrac12(\mathcal E\odot\mathcal E)\times_1\mathbf{1}_{N_{\mathrm{obs}}},
\end{equation}
where $\odot$ is the Hadamard product.

Since we are interested in a TT representation of $\mathcal R$, we first obtain the TT decomposition of $\mathcal{E}$ using the standard TT addition construction based on block-diagonal augmentation of the TT cores; see, e.g., \cite{oseledets2011tensor}.
Let
\[
\mathcal{E} = H_1\circ H_2\circ\cdots\circ H_{D+1},
\]
where
$H_k\in\mathbb{R}^{\bar r_k\times N_{k-1}\times\bar r_{k+1}}$ are the TT cores, with ranks $\bar r_k\le r_k+1$.

The TT cores of the Hadamard product in~\eqref{eq:R} are obtained by taking, for each physical index, the Kronecker product of the corresponding TT-core slice with itself. To make this more precise, let us use the following indexing convention: for $H_k\in\mathbb{R}^{\bar r_k\times N_{k-1}\times\bar r_{k+1}}$, write $a\in\{1,\ldots,\bar r_k\}$ for its \emph{left} rank index and $b\in\{1,\ldots,\bar r_{k+1}\}$ for its \emph{right} rank index. Squaring duplicates each rank index into a pair: $a$ becomes $(a,a')$ with both $a,a'\in\{1,\ldots,\bar r_k\}$, and similarly $b$ becomes $(b,b')$. We treat the pair $(a,a')$ as a \emph{single composite index} taking $\bar r_k^{2}$ values, and likewise $(b,b')$ for the right rank. Then, for $k=2,\ldots,D+1$, define
\[
\widetilde{H}_k\bigl((a,a'),n_{k-1},(b,b')\bigr)
= H_k(a,n_{k-1},b)H_k(a',n_{k-1},b'),
\]

\noindent so $\widetilde{H}_k\in\mathbb{R}^{\bar r_k^{2}\times N_{k-1}\times\bar r_{k+1}^{2}}$.

In \eqref{eq:R}, one contracts the Hadamard product along mode-$1$ with $\mathbf{1}_{N_{0}}$. To this end, treating $H_1$ as the $N_{\mathrm{obs}}\times \bar r_2$ matrix, define
\[
W = H_1^{\mathsf{T}}H_1, \qquad  W\in \mathbb{R}^{\bar r_2\times\bar r_2}.
\]
Stacking $W$ over the composite pair index and including the factor $\tfrac{1}{2}$ from \eqref{eq:R} gives a vector of length $\bar r_2^{2}$:
\[
\mathbf{w}=\tfrac{1}{2}\operatorname{vec} \bigl(H_1^{\mathsf{T}}H_1\bigr) \in\mathbb{R}^{\bar r_2^{2}}.
\]

\noindent The purpose of $\operatorname{vec}$ is to flatten the matrix indices $(a,a')$ into the same composite rank index that $\widetilde{H}_2$ carries on its left. Next, we absorb $\mathbf{w}$ into $\widetilde{H}_2$ by contracting along that composite index:
\[
\widehat{H}_2 = \widetilde{H}_2\times_1 \mathbf{w}.
\]

\noindent The result is a TT representation of the order-$D$ tensor $\mathcal{R}$:
\begin{equation}\label{eq:R2}
	\mathcal{R} =\widehat{H}_2\circ\widetilde {H}_3\circ\cdots\circ\widetilde{H}_{D+1},
\end{equation}

\noindent with unrounded TT ranks $\{1,\bar r_3^{2},\bar r_4^{2},\ldots,\bar r_{D+1}^{2},1\}$.
Since the squared ranks may be a loose upper bound, we perform a TT-rounding sweep on the decomposition~\eqref{eq:R2} (see \cite{oseledets2011tensor}) to compress the ranks before executing a minimal entry search.

Since $\mathcal L$ is given by a sum of univariate grid functions, it admits an exact tensor-train representation with TT-ranks
$
(1,2,\dots,2,1).
$
To see this, define the  grid function
\begin{equation*}
	g_i(n_i)  =  \left(\halpha_i^{n_i}-(\alpha_{\mathrm{ref}})_i\right)^2,
	\qquad n_i=1,
	\dots,N_i,
\end{equation*}

\noindent for each mode \(i=1,\dots,D\). The first and last cores are given by
\begin{equation*}
	L_1(n_1)= \frac{\lambda}{2}\begin{bmatrix}
		g_1(n_1) & 1
	\end{bmatrix} \in \mathbb R^{1\times 2},\qquad
	L_D(n_D)= \begin{bmatrix}
		1\\
		g_D(n_D)
	\end{bmatrix} \in \mathbb R^{2\times 1},
\end{equation*}

\noindent and the intermediate cores are
\begin{equation*}
	L_i(n_i)= \begin{bmatrix}
		1 & 0\\
		g_i(n_i) & 1
	\end{bmatrix} \in \mathbb R^{2\times 2}, \quad 2\le i\le D-1.
\end{equation*}

\noindent Straightforward calculations show
\begin{equation*}
	\mathcal L (\mathbf{n}) =\mathcal L(n_1,\dots,n_D) =L_1(n_1)\cdots L_D(n_D)
	= \frac{\lambda}{2} \sum_{i=1}^{D} g_i(n_i),
\end{equation*}

\noindent and thus we have TT representation of the regularization tensor $\mathcal L = L_1\circ L_2 \circ \cdots \circ L_D$.

The TT decomposition of $\mathcal{R} +\mathcal{L}$ is obtained using the standard TT addition procedure one more time.

Once $\mathcal J$ is in TT format, we use the optimization algorithm of~\cite{chertkov2022optimization} to locate its minimum entry; see also the related maximum-volume TT optimizer of~\cite{sozykin2022ttopt}, which searches a \emph{black-box} scalar objective on a quantized TT grid and is derivative free.  The essential difference here is that the tensor being searched is \emph{constructed} from the TT cores of $\widehat{\mathcal T}$ rather than sampled: evaluating its entries requires no further FOM solves, and the same cores also supply $\widehat f$ and its Jacobian off the grid for the subsequent GN refinement. Since the algorithm is formulated as a maximum-entry search, we apply it to an equivalent shifted tensor as described below.
The algorithm searches for a maximum entry by right-orthogonalizing the TT cores and performing a left-to-right sweep. At each step, only several candidate index paths with the largest partial-contraction norms are retained.

The minimal entry of $\mathcal J$ is  obtained by executing the algorithm twice. First, one computes
\[
\mathbf{n}_{\mathrm{max}}
= \argmax_{\mathbf{n}\in\mathcal{I}} \mathcal J (\mathbf{n}),
\]

\noindent and then solves
\[
\mathbf n_\ast =
\argmax_{\mathbf{n}\in\mathcal{I}}
\left\{ \mathcal J(\mathbf n_{\mathrm{max}}) - \mathcal J(\mathbf n) \right\}.
\]
The corresponding parameter value $\alpha_{\mathbf n_\ast}$ is the TROM-based discrete minimizer. Depending on the inverse problem, it can be used either as a stand-alone approximate minimizer of the objective functional or as an informed initial guess $\alpha^{(0)}=\alpha_{\mathbf n_\ast}$ for a subsequent GN solve.

\begin{remark}
\label[remark]{rem:map-vs-objective}
		Unlike TT-based methods that construct and search a TT approximation of a data-specific scalar objective~\cite{sozykin2022ttopt,krivorotko2020global,zheltkova2018tensor}, our  TROM approximates the complete parameter-to-observation map once and reuses it for multiple observation vectors.  Indeed, the objective tensor $\mathcal J$ is assembled in closed form from the TT cores of $\widehat{\mathcal T}$ by the construction above, so a new data set $y_{\mathrm{obs}}$ costs no FOM evaluations beyond those already spent on the TROM.
More precisely, the dependence on $y_{\mathrm{obs}}$ is confined to a single core.  In the TT addition defining $\mathcal E$, only the first core $H_1=[\,G_1\;\;{-}y_{\mathrm{obs}}\,]\in\mathbb R^{N_{\mathrm{obs}}\times(r_2+1)}$ involves the data; the cores $H_2,\ldots,H_{D+1}$, and with them the squared cores $\widetilde H_2,\ldots,\widetilde H_{D+1}$ in~\eqref{eq:R2}, are data independent and are formed once.  Since the columns of $G_1$ are orthonormal,
\[
W=H_1^{\mathsf T}H_1
=\begin{pmatrix}
I_{r_2} & -y_{\mathrm{red}}\\[0.2em]
-y_{\mathrm{red}}^{\mathsf T} & \|y_{\mathrm{obs}}\|_2^2
\end{pmatrix},
\]
so a new data set enters only through the $r_2+1$ numbers $y_{\mathrm{red}}=(G_1)^{\mathsf T}y_{\mathrm{obs}}$ of~\eqref{eq:y_red} and $\|y_{\mathrm{obs}}\|_2^2$, and only $\widehat H_2$ has to be reassembled.  What recurs for every data set is therefore the formation of $\widehat H_2$, one TT-rounding sweep of~\eqref{eq:R2}, and the minimal-entry search; no forward solve is involved.  Reuse is of course confined to the parameter box $\mathcal A$ and the grid $\hcA$ on which $\widehat{\mathcal T}$ was built. Building a data-specific objective tensor directly by TT-Cross instead requires one FOM solve per sampled entry, i.e., on the order of $D N_{\mathrm{max}} \bar r_{\mathrm{max}}^2$ FOM evaluations for every data set, where $N_{\mathrm{max}}=\max_i N_i$ and $\bar r_{\mathrm{max}}$ bounds the TT ranks of $\mathcal J$.
\end{remark}

\section{Error-to-inversion analysis}
\label{sec:error-to-inversion}

Error control for ROM-based optimization and inversion has been developed
primarily for certified reduced-basis methods~\cite{hesthaven2016certified}.  Certified reduced-basis approximations of parametrized optimal control problems, together with a posteriori bounds for the associated quantities of interest, were introduced in~\cite{negri2013reduced,karcher2014certified}.   The authors of~\cite{dihlmann2015certified} and~\cite{manzoni2019certified} derive a posteriori estimates for reduced
objective functions, gradients, and optimal parameters in PDE-constrained
parameter optimization, while related work certifies reduced four-dimensional variational (4D-Var)
approximations~\cite{karcher2018reduced} or embeds reduced models in
error-aware trust-region schemes for PDE-constrained optimization and
parameter identification~\cite{qian2017certified,kartmann2024adaptive}.  
More recently, finite-element surrogates based on extreme learning networks have been analyzed together with the resulting parameter-reconstruction error~\cite{burman2026solving}. The analysis below is complementary: it directly propagates errors in a TROM approximation of the parameter-to-observation map and its Jacobian to the recovered parameter, and then relates those errors to tensor compression and interpolation.


We next quantify how the TROM error affects the parameter recovered by the
continuous optimization problem. Denote the FOM
objective in~\eqref{eq:inverse_problem} by $\Phi_\lambda$ and its TROM
counterpart, obtained by replacing $f$ with the surrogate $\widehat f$
from~\eqref{eqn:f_TT}, by $\widehat\Phi_\lambda$.  The derivation of
\eqref{eq:inverse_problem_TROM} shows that the reduced TROM objective and
$\widehat\Phi_\lambda$ differ only by a constant independent of $\alpha$;
hence they have the same minimizers.

Since $\mathcal A$ is compact and the objectives are continuous, the two
problems have global minimizers
\begin{equation}
\alpha_\lambda\in\argmin_{\alpha\in\mathcal A}\Phi_\lambda(\alpha),
\qquad
\widehat\alpha_\lambda
\in\argmin_{\alpha\in\mathcal A}\widehat\Phi_\lambda(\alpha).
\label{eq:fom-trom-global-minimizers}
\end{equation}
These minimizers need not be unique; in what follows we fix one global
minimizer of each problem.

Let $\mathcal K\subset\mathcal A$ be a compact convex set containing the selected minimizers in~\eqref{eq:fom-trom-global-minimizers}. Since $\widehat\alpha_\lambda$ is not known a priori, $\mathcal K$ is treated as an assumed localization set rather than constructed explicitly. The choice $\mathcal K=\mathcal A$ is always admissible and is covered by \Cref{rem:large-lambda-strong-convexity}.
Assume that
$f\in C^1(\mathcal K)^{N_{\mathrm{obs}}}$ and let
$J(\alpha)=\partial f/\partial\alpha$. We also assume $\widehat f$ continuous, $\widehat f\in
	W^{1,\infty}(\mathcal K)^{N_{\mathrm{obs}}}$ and let
$\widehat J(\alpha)=\partial\widehat f/\partial\alpha\in L^\infty(\mathcal K)^{N_{\mathrm{obs}}\times D}$. Introduce the
local approximation errors
\begin{equation}
\epsilon_0^{\mathrm{rom}}
=\sup_{\alpha\in\mathcal K}
 \|\widehat f(\alpha)-f(\alpha)\|_2,
\qquad
\epsilon_1^{\mathrm{rom}}
=\mbox{ess}\sup_{\alpha\in\mathcal K}
 \|\widehat J(\alpha)-J(\alpha)\|_2,
\label{eq:C0-C1-trom-errors}
\end{equation}
and the local FOM bounds
\begin{equation}
R_{\mathcal K}
=\sup_{\alpha\in\mathcal K}
 \|f(\alpha)-y_{\mathrm{obs}}\|_2,
\qquad
L_{\mathcal K}
=\sup_{\alpha\in\mathcal K}\|J(\alpha)\|_2.
\label{eq:local-fom-bounds}
\end{equation}

The relaxed regularity assumption on $\widehat f$ permits continuous piecewise-polynomial
interpolation, including \eqref{eqn:lagrange} with
$p=2$. 

\begin{lemma}[A generic ROM inversion error]
\label{thm:local-trom-reconstruction-error}
Suppose that the regularized FOM objective is locally strongly convex on
$\mathcal K$: there exists $\mu>0$ such that
\begin{equation}
\big\langle
\nabla\Phi_\lambda(\alpha)-\nabla\Phi_\lambda(\beta),
\alpha-\beta
\big\rangle
\geq
\mu\|\alpha-\beta\|_2^2
\qquad
\text{for all }\alpha,\beta\in\mathcal K.
\label{eq:local-strong-monotonicity}
\end{equation}
If $\Phi_\lambda$ is twice continuously differentiable, it is sufficient that
$\nabla^2\Phi_\lambda(\alpha)\succeq\mu I$ on $\mathcal K$.
Then the global minimizers in~\eqref{eq:fom-trom-global-minimizers} satisfy
\begin{equation}
\|\widehat\alpha_\lambda-\alpha_\lambda\|_2
\leq
\frac{
R_{\mathcal K}\epsilon_1^{\mathrm{rom}}
+(L_{\mathcal K}+\epsilon_1^{\mathrm{rom}})\epsilon_0^{\mathrm{rom}}
}{\mu}.
\label{eq:local-parameter-error-bound}
\end{equation}
\end{lemma}

\begin{proof}
Let $r(\alpha)=f(\alpha)-y_{\mathrm{obs}}$ and
$e(\alpha)=\widehat f(\alpha)-f(\alpha)$.  For
a.e. $\alpha\in\mathcal K$,
\[
\nabla \widehat\Phi_\lambda (\alpha) -\nabla\Phi_\lambda (\alpha)
=(\widehat J(\alpha)-J(\alpha))^\mathsf{T}r(\alpha)
+\widehat J(\alpha)^\mathsf{T}e(\alpha).
\]
Set
	$q=\widehat\Phi_\lambda-\Phi_\lambda$. Since $\|\widehat J(\alpha)\|_2\leq
L_{\mathcal K}+\epsilon_1^{\mathrm{rom}}$ for a.e. $\alpha\in\mathcal K$,
\eqref{eq:C0-C1-trom-errors}--\eqref{eq:local-fom-bounds} give
\begin{equation}
\|\nabla q(\alpha)\|_2
\leq B:= R_{\mathcal K}\epsilon_1^{\mathrm{rom}}
+(L_{\mathcal K}+\epsilon_1^{\mathrm{rom}})\epsilon_0^{\mathrm{rom}}
\quad\text{for a.e. }\alpha\in\mathcal K.
\label{eq:gradient-perturbation-bound}
\end{equation}

 Since $\mathcal K$ is convex, $q$ is $B$-Lipschitz on
$\mathcal K$. After extending $q$ to a $B$-Lipschitz function and retaining
the same notation, every Clarke subgradient
$\zeta\in\partial_C q(\alpha)$ therefore satisfies
$\|\zeta\|_2\leq B$.

Let $d=\widehat\alpha_\lambda-\alpha_\lambda$.  We can assume
$d\neq0$.  The smooth and Clarke first-order necessary conditions therefore
yield, for some
$\zeta\in\partial_C q(\widehat\alpha_\lambda)$,
\[
\big\langle\nabla\Phi_\lambda(\alpha_\lambda),d\big\rangle\geq0,
\qquad
\big\langle
\nabla\Phi_\lambda(\widehat\alpha_\lambda)+\zeta,d
\big\rangle\leq0.
\]
Consequently, \eqref{eq:local-strong-monotonicity} and
\eqref{eq:gradient-perturbation-bound} imply
\begin{align*}
\mu\|d\|_2^2
&\leq
\big\langle
\nabla\Phi_\lambda(\widehat\alpha_\lambda)
-\nabla\Phi_\lambda(\alpha_\lambda),d
\big\rangle\\
&\leq{\big\langle
\nabla\Phi_\lambda(\widehat\alpha_\lambda),d
\big\rangle}\\
&\leq-\langle\zeta,d\rangle
\leq\|\zeta\|_2\|d\|_2
\leq
\bigl(R_{\mathcal K}\epsilon_1^{\mathrm{rom}}
 +(L_{\mathcal K}+\epsilon_1^{\mathrm{rom}})\epsilon_0^{\mathrm{rom}}\bigr)\|d\|_2.
\end{align*}
Here the second and third steps use the two first-order conditions above.
Dividing by $\mu\|d\|_2$ proves~\eqref{eq:local-parameter-error-bound}.
\end{proof}

We note that for the true parameter $\alpha_{\mathrm{true}}\in\mathcal A$, the assumptions
of \Cref{thm:local-trom-reconstruction-error} and the triangle inequality imply
\begin{equation}
\|\widehat\alpha_\lambda-\alpha_{\mathrm{true}}\|_2
\leq
\|\alpha_\lambda-\alpha_{\mathrm{true}}\|_2
+\frac{
R_{\mathcal K}\epsilon_1^{\mathrm{rom}}
+(L_{\mathcal K}+\epsilon_1^{\mathrm{rom}})\epsilon_0^{\mathrm{rom}}
}{\mu}.
\label{eq:total-parameter-error-bound}
\end{equation}

The first term on the right-hand side of
\eqref{eq:total-parameter-error-bound} is the reconstruction error of the FOM
inverse problem and contains the effects of measurement noise,
regularization bias, and FOM discretization.  The second term is the additional
error introduced by replacing the FOM by a ROM. We now consider the critical quantities $\epsilon_0^{\mathrm{rom}}$ and $\epsilon_1^{\mathrm{rom}}$  for the TROM inversion.

\begin{remark}
\label[remark]{rem:large-lambda-strong-convexity}
 Define
$
B_{\mathcal A}\defeq
\sup\limits_{\alpha\in\mathcal A}
\left\|
 \langle f(\alpha)-y_{\mathrm{obs}},
 \nabla^2 f(\alpha)\rangle
\right\|_2
$  for $f\in C^2(\mathcal A;\mathbb R^{N_{\mathrm{obs}}})$.
Since $\mathcal A$ is compact, it holds $B_{\mathcal A}<\infty$.
Denote by $\sigma_0\defeq
\min_{\alpha\in\mathcal A}\sigma_{\min}(J(\alpha))\geq0 $ the minimum singular value of the FOM Jacobian over $\mathcal A$.
Then
\[
\nabla^2\Phi_\lambda(\alpha)
=J(\alpha)^\mathsf{T}J(\alpha)
+\langle f(\alpha)-y_{\mathrm{obs}},
\nabla^2 f(\alpha)\rangle
+\lambda I
\succeq
\bigl(\sigma_0^2+\lambda-B_{\mathcal A}\bigr)I
\]
uniformly for $\alpha\in\mathcal A$.  In this case,
\eqref{eq:local-strong-monotonicity} holds on $\mathcal K=\mathcal A$
with $\mu=\sigma_0^2+\lambda-B_{\mathcal A}>0$ whenever
$\lambda>B_{\mathcal A}-\sigma_0^2$, and hence for all sufficiently large
$\lambda$.
\end{remark}

\subsection{Estimates for \texorpdfstring{$\epsilon_0^{\mathrm{rom}}$}{epsilon0} and
\texorpdfstring{$\epsilon_1^{\mathrm{rom}}$}{epsilon1}}
\label{sec:epsilon-estimates}

We adapt the interpolation argument of
\cite[Sec.~4]{mamonov2022interpolatory} and \cite{mamonov2025apriori} to the parameter-to-observation map
and extend it to its Jacobian. Moreover, to avoid a tensor error bound that grows under
refinement of the parameter grid, we introduce to this argument a max-fiber
tensor norm.

For $i=1,\ldots,D$, let
\[
I_i=[\alpha_i^{\mathrm{min}},\alpha_i^{\mathrm{max}}],
\qquad
\delta_i=\max_{1\leq j<N_i}
|\halpha_i^{j+1}-\halpha_i^j|
\]
denote the $i$th parameter interval and its maximum grid spacing.  For a
vector function $v$ on $\mathcal A$, let $\mathcal I_i$ denote the
factor of the tensor-product interpolant~\eqref{Interp} acting in the $i$th
parameter direction:
\begin{equation}
(\mathcal I_i v)(\alpha)
=\sum_{j=1}^{N_i}(\chi_i(\alpha_i))_j
v(\alpha_1,\ldots,\alpha_{i-1},\halpha_i^j,
\alpha_{i+1},\ldots,\alpha_D).
\label{eq:coordinate-interpolation-operator}
\end{equation}
For a measurable vector function $v$ on $\mathcal A$, we define
$
\|v\|_{\infty,2}
=\operatorname*{ess\,sup}_{\alpha\in\mathcal A}\|v(\alpha)\|_2.
$
For simplicity of presentation, we use the single mesh parameter
\begin{equation}
\delta=\max_{1\leq i\leq D}\delta_i
\label{eq:maximum-parameter-grid-spacing}
\end{equation}
and state the estimates under the mesh-comparability condition below.  This
convention is used only to streamline the presentation; retaining the
individual $\delta_i$ gives sharper anisotropic estimates and removes this
condition.

For any tensor
$\mathcal X\in\mathbb R^{N_{\mathrm{obs}}\times N_1\times\cdots\times N_D}$,
define the max-fiber norm
\begin{equation}
\|\mathcal X\|_\ast
\defeq
\max_{\substack{1\leq n_i\leq N_i\\i=1,\ldots,D}}
\|\mathcal X(:,n_1,\ldots,n_D)\|_2.
\label{eq:max-fiber-tensor-norm}
\end{equation}
It follows from~\eqref{eqn:snapmulti} that
\begin{equation}
\|\mathcal T\|_\ast
=\max_{\halpha\in\hcA}\|f(\halpha)\|_2
\leq\sup_{\alpha\in\mathcal A}\|f(\alpha)\|_2.
\label{eq:observation-tensor-max-fiber-bound}
\end{equation}
Thus $\|\mathcal T\|_\ast$ is bounded independently of the number of samples
in the parameter grid whenever $f$ is uniformly bounded on $\mathcal A$.

\begin{assumption}[Tensor compression in the max-fiber norm]
\label[assumption]{ass:max-fiber-tensor-compression}
For some $\varepsilon_{\mathrm{TT},\ast}>0$, the compressed tensor satisfies
\begin{equation}
\|\mathcal T-\widehat{\mathcal T}\|_\ast
\leq
\varepsilon_{\mathrm{TT},\ast}\|\mathcal T\|_\ast.
\label{eq:max-fiber-tensor-compression}
\end{equation}
\end{assumption}

\begin{remark}[Relation to the Frobenius TT tolerance]
\label[remark]{rem:frobenius-max-fiber-relation}
We retain the Frobenius-norm condition~\eqref{TTapprox}.  Since
$\|\mathcal X\|_\ast\leq\|\mathcal X\|_F$, it implies
\begin{equation}
\|\mathcal T-\widehat{\mathcal T}\|_\ast
\leq\varepsilon_{\mathrm{TT}}\|\mathcal T\|_F.
\label{eq:frobenius-to-max-fiber-error}
\end{equation}
Consequently, for $\mathcal T\ne0$, a prescribed max-fiber relative tolerance
$\varepsilon_{\mathrm{TT},\ast}$ is guaranteed by choosing the Frobenius
relative tolerance in~\eqref{TTapprox} so that
$
\varepsilon_{\mathrm{TT}} \leq
\varepsilon_{\mathrm{TT},\ast}
\|\mathcal T\|_\ast \|\mathcal T\|_F^{-1}.
$
A fixed relative tolerance in~\eqref{TTapprox} alone does not yield a grid-independent relative
tolerance in~\eqref{eq:max-fiber-tensor-compression} in general.  When TT-Cross is used, \eqref{eq:max-fiber-tensor-compression} may instead be imposed or validated independently.
\end{remark}

\begin{assumption}[Interpolation and parametric regularity]
\label[assumption]{ass:interpolation-regularity}
For some integer $p\geq2$ and constants
$C_a,C_{a,1},C_e,C_{e,1}>0$ independent of the grid, the coordinate
meshes are uniformly comparable: $\delta_i\geq c_\delta\delta$ for some
$c_\delta>0$ independent of grid refinement.  
The coordinate interpolants satisfy the standard approximation properties:
\begin{align}
\|v-\mathcal I_i v\|_{\infty,2}
&\leq C_a\delta^p\|\partial_i^p v\|_{\infty,2},
\label{eq:univariate-vector-interpolation}\\
\|\partial_i(v-\mathcal I_i v)\|_{\infty,2}
&\leq C_{a,1}\delta^{p-1}
\|\partial_i^p v\|_{\infty,2}
\label{eq:univariate-derivative-interpolation}
\end{align}
for every sufficiently smooth vector-valued function $v$.  Moreover, the
weights are continuous, weakly differentiable and satisfy
\begin{equation}
\sup_{a\in I_i}\|\chi_i(a)\|_{\ell_1}\leq C_e,
\qquad
\operatorname*{ess\,sup}_{a\in I_i}\|\chi_i'(a)\|_{\ell_1}
\leq C_{e,1}\delta^{-1}
\label{eq:interpolation-stability-error}
\end{equation}
for $i=1,\ldots,D$.  Finally, we assume
$f\in C^{p+1}(\mathcal A;\mathbb R^{N_{\mathrm{obs}}})$, and we set
\begin{align}
M_{p,\mathcal A}
\defeq
\max_{1\leq i\leq D}
\|\partial_i^p f\|_{\infty,2},
\qquad
M_{p+1,\mathcal A}
\defeq
\max_{1\leq i,q\leq D}
\|\partial_q\partial_i^p f\|_{\infty,2}.
\label{eq:mixed-parametric-smoothness-constant}
\end{align}
\end{assumption}
We note that the interpolation coefficients stability bounds as in  \eqref{eq:interpolation-stability-error}  are standard for local polynomial interpolation on a uniform or
quasi-uniform grid; the factor $\delta^{-1}$ records the scaling introduced
by differentiating the interpolation weights.
The stability constants enter the estimates below through the factor $C_e^D$, which is exponential in the parameter dimension. For the local interpolation rule~\eqref{eqn:lagrange}, $C_e$ is close to one and $C_e=1$ for $p=2$, so this factor remains moderate for the parameter dimensions considered here. Global high-degree interpolation may require different arguments for the resulting estimates to remain practical.

We need the elementary estimate given by the following lemma.

\begin{lemma}[Max-fiber contraction]
\label{lem:max-fiber-contraction}
For every
$\mathcal X\in\mathbb R^{N_{\mathrm{obs}}\times N_1\times\cdots\times N_D}$
and vectors $v_i\in\mathbb R^{N_i}$, $i=1,\ldots,D$, with the contraction understood as in~\eqref{eqn:extractbt}, it holds
\begin{equation}
\|\mathcal X\times_2v_1\times_3\cdots\times_{D+1}v_D\|_2
\leq
\|\mathcal X\|_\ast\prod_{i=1}^D\|v_i\|_{\ell_1}.
\label{eq:max-fiber-contraction}
\end{equation}
\end{lemma}

\begin{proof}
Expanding the contractions and applying the triangle inequality gives
\begin{align*}
&\|\mathcal X\times_2v_1\times_3\cdots\times_{D+1}v_D\|_2\\
&\quad=
\left\|
\sum_{n_1=1}^{N_1}\cdots\sum_{n_D=1}^{N_D}
\left(\prod_{i=1}^D(v_i)_{n_i}\right)
\mathcal X(:,n_1,\ldots,n_D)
\right\|_2\\
&\quad\leq
\|\mathcal X\|_\ast
\sum_{n_1=1}^{N_1}\cdots\sum_{n_D=1}^{N_D}
\prod_{i=1}^D|(v_i)_{n_i}|
=\|\mathcal X\|_\ast\prod_{i=1}^D\|v_i\|_{\ell_1}.
\end{align*}
\end{proof}

Define the interpolant obtained from the uncompressed observation tensor by replacing $\widehat{\mathcal T}$ with $\mathcal T$ in~\eqref{eqn:extractbt},
\begin{equation}
f_{\mathcal T}(\alpha)\defeq\mathcal T\times_2\chi_1(\alpha_1)\times_3\cdots\times_{D+1}\chi_D(\alpha_D)
\label{eq:full-tensor-interpolant}
\end{equation}

For brevity, set
\begin{equation}\label{ETT}
E_{\mathrm{TT},\ast}
=\varepsilon_{\mathrm{TT},\ast}\|\mathcal T\|_\ast.
\end{equation}

\begin{proposition}[TROM map error]
	\label{prop:epsilon-zero-estimate}
		Under \Cref{ass:max-fiber-tensor-compression,ass:interpolation-regularity}, it holds
	\begin{equation}
		\epsilon_0^{\mathrm{rom}}
		\leq
		C_e^D E_{\mathrm{TT},\ast}
		+DC_aC_e^{D-1}M_{p,\mathcal A}\delta^p.
		\label{eq:epsilon-zero-estimate-relative}
	\end{equation}
with constants $C_a$, $C_e$, and $M_{p,\mathcal A}$ defined in \eqref{eq:univariate-vector-interpolation},  \eqref{eq:interpolation-stability-error} , and \eqref{eq:mixed-parametric-smoothness-constant}, respectively.
\end{proposition}

\begin{proof}
	The triangle inequality gives
	\begin{equation}
		\|f(\alpha)-\widehat f(\alpha)\|_2
		\leq
		\|f(\alpha)-f_{\mathcal T}(\alpha)\|_2
		+\|f_{\mathcal T}(\alpha)-\widehat f(\alpha)\|_2.
		\label{eq:epsilon-zero-triangle}
	\end{equation}
		For the second term, \Cref{lem:max-fiber-contraction} and
		\eqref{eq:interpolation-stability-error} give by repeating calculations from \cite[(4.10)]{mamonov2022interpolatory}
	\begin{equation}
		\|f_{\mathcal T}(\alpha)-\widehat f(\alpha)\|_2
		\leq
		\|\mathcal T-\widehat{\mathcal T}\|_\ast
		\prod_{i=1}^D\|\chi_i(\alpha_i)\|_{\ell_1}
        \leq C_e^D\|\mathcal T-\widehat{\mathcal T}\|_\ast.
		\label{eq:compressed-tensor-contraction-error}
	\end{equation}

	By~\eqref{eq:coordinate-interpolation-operator}, we can write
	$f_{\mathcal T}=\mathcal I_1\cdots\mathcal I_Df$.  Now the identity
	\begin{equation*}
		I-\mathcal I_1\cdots\mathcal I_D
		=\sum_{i=1}^D
		\mathcal I_1\cdots\mathcal I_{i-1}(I-\mathcal I_i).
		\end{equation*}
	and the stability of $\mathcal I_i$ allow us to apply
	\eqref{eq:univariate-vector-interpolation} successively, which yields
	\begin{align}
		\sup_{\alpha\in\mathcal A}
		\|f(\alpha)-f_{\mathcal T}(\alpha)\|_2
		&\leq
		DC_aC_e^{D-1}M_{p,\mathcal A}\delta^p.
		\label{eq:tensor-product-interpolation-error}
	\end{align}
	 Combining
	\eqref{eq:epsilon-zero-triangle},
	\eqref{eq:compressed-tensor-contraction-error}, and
	\eqref{eq:tensor-product-interpolation-error}, using
		$\mathcal K\subset\mathcal A$, and invoking
		\eqref{eq:max-fiber-tensor-compression} proves
	\eqref{eq:epsilon-zero-estimate-relative}.
\end{proof}

\begin{proposition}[TROM Jacobian error]
\label{prop:epsilon-one-estimate}
Under \Cref{ass:max-fiber-tensor-compression,ass:interpolation-regularity},
\begin{align}
\epsilon_1^{\mathrm{rom}}
&\leq
\sqrt D\,C_{e,1}C_e^{D-1}E_{\mathrm{TT},\ast}\delta^{-1}
\notag\\
&\quad+
\sqrt D\,C_e^{D-1}
\left(
C_{a,1}M_{p,\mathcal A}\delta^{p-1}
+(D-1)C_aM_{p+1,\mathcal A}\delta^p
\right)
\label{eq:epsilon-one-estimate-relative}
\end{align}
with constants $C_a$, $C_{a,1}$, $C_e,\,C_{e,1}$, and $M_{p,\mathcal A}$, $M_{p+1,\mathcal A}$ defined in \eqref{eq:univariate-vector-interpolation}, \eqref{eq:univariate-derivative-interpolation}, \eqref{eq:interpolation-stability-error}, and \eqref{eq:mixed-parametric-smoothness-constant}, respectively.
\end{proposition}

\begin{proof}
Let $J_{\mathcal T}=\partial f_{\mathcal T}/\partial\alpha$.
Differentiating the contractions that define $f_{\mathcal T}$ and
$\widehat f$ gives, for $q=1,\ldots,D$, and a.e.
$\alpha\in\mathcal A$,
\begin{align}
\left\|
\partial_q(f_{\mathcal T}-\widehat f)(\alpha)
\right\|_2
&\leq
\|\mathcal T-\widehat{\mathcal T}\|_\ast
\|\chi_q'(\alpha_q)\|_{\ell_1}
\prod_{i\ne q}\|\chi_i(\alpha_i)\|_{\ell_1}
\notag\\
&\leq
C_{e,1}C_e^{D-1}\delta^{-1}
\|\mathcal T-\widehat{\mathcal T}\|_\ast.
\label{eq:differentiated-tensor-contraction-error}
\end{align}
Since these vectors are the columns of
$J_{\mathcal T}-\widehat J$, the spectral norm is bounded by the
Frobenius norm, and hence
\begin{equation}
\operatorname*{ess\,sup}_{\alpha\in\mathcal A}
\|J_{\mathcal T}(\alpha)-\widehat J(\alpha)\|_2
\leq
\sqrt D\,C_{e,1}C_e^{D-1}\delta^{-1}
\|\mathcal T-\widehat{\mathcal T}\|_\ast.
\label{eq:compressed-tensor-jacobian-error}
\end{equation}

It remains to estimate the Jacobian error of the uncompressed interpolant.
For fixed $q$, set
$\mathcal I_{-q}=\prod_{i\ne q}\mathcal I_i$.  Since interpolation in the
other parameter directions commutes with $\partial_q$, from the definition of $f_{\mathcal T}$ it follows  that
\[
\partial_q f_{\mathcal T}
=\mathcal I_{-q}\partial_q(\mathcal I_q f).
\]
Consequently,
\begin{equation}
\partial_q(f-f_{\mathcal T})
=(I-\mathcal I_{-q})\partial_q f
+\mathcal I_{-q}\partial_q(f-\mathcal I_q f)
\qquad\text{a.e. in }\mathcal A.
\label{eq:jacobian-interpolation-decomposition}
\end{equation}
Now, similar to \eqref{eq:tensor-product-interpolation-error},  the telescoping expansion of $I-\mathcal I_{-q}$, followed by
\eqref{eq:univariate-vector-interpolation} applied to $\partial_q f$, gives
\begin{equation}
\|(I-\mathcal I_{-q})\partial_q f\|_{\infty,2}
\leq
(D-1)C_aC_e^{D-1}M_{p+1,\mathcal A}\delta^p.
\label{eq:cross-direction-jacobian-interpolation-error}
\end{equation}
Here the mixed derivative
$\partial_i^p\partial_q f$ required the additional order of parametric
regularity from \Cref{ass:interpolation-regularity}.  For the second
term in~\eqref{eq:jacobian-interpolation-decomposition}, the stability of
$\mathcal I_{-q}$ and
\eqref{eq:univariate-derivative-interpolation} yield
\begin{equation}
\|\mathcal I_{-q}\partial_q(f-\mathcal I_q f)\|_{\infty,2}
\leq
C_{a,1}C_e^{D-1}M_{p,\mathcal A}\delta^{p-1}.
\label{eq:same-direction-jacobian-interpolation-error}
\end{equation}
Combining~\eqref{eq:jacobian-interpolation-decomposition}--
\eqref{eq:same-direction-jacobian-interpolation-error} columnwise and again
using the Frobenius norm gives
\begin{align}
\operatorname*{ess\,sup}_{\alpha\in\mathcal A}
\|J(\alpha)-J_{\mathcal T}(\alpha)\|_2
&\leq
\sqrt D\,C_e^{D-1}
\left(
C_{a,1}M_{p,\mathcal A}\delta^{p-1}
+(D-1)C_aM_{p+1,\mathcal A}\delta^p
\right).
\label{eq:full-tensor-jacobian-interpolation-error}
\end{align}
Finally, the triangle inequality for
$\widehat J-J=(\widehat J-J_{\mathcal T})+(J_{\mathcal T}-J)$,
the tensor bound~\eqref{eq:max-fiber-tensor-compression}, and
$\mathcal K\subset\mathcal A$ prove
\eqref{eq:epsilon-one-estimate-relative}.
\end{proof}

Assuming $\delta\lesssim1$, the two estimates can be
written in a more compact way,
\begin{align}
\epsilon_0^{\mathrm{rom}}
&\leq C_0\bigl(E_{\mathrm{TT},\ast}+\delta^p\bigr),
\label{eq:epsilon-zero-uniform-grid}\\
\epsilon_1^{\mathrm{rom}}
&\leq C_1\bigl(\delta^{-1}E_{\mathrm{TT},\ast}
+\delta^{p-1}\bigr),
\label{eq:epsilon-one-uniform-grid}
\end{align}
where $C_0$ and $C_1$ are independent of $\delta$ and
$\varepsilon_{\mathrm{TT},\ast}$.  The loss of one power of the mesh size parameter in the derivative error estimate, compared with the error estimate for the map itself, is typical in interpolation error analysis.

We finally combine the surrogate estimates with the inversion stability
result to show the main TROM error estimate.  Set
\begin{equation}
\theta_\delta
\defeq
\delta^{-1}E_{\mathrm{TT},\ast}+\delta^{p-1}.
\label{eq:combined-surrogate-error-rate}
\end{equation}

\begin{theorem}[Compression-to-inversion estimate]
\label{thm:compression-to-inversion-estimate}
Suppose that \eqref{eq:local-strong-monotonicity} and
\Cref{ass:max-fiber-tensor-compression,ass:interpolation-regularity} hold and that $\delta\leq1$.  Then
\begin{equation}
\|\widehat\alpha_\lambda-\alpha_\lambda\|_2
\leq
\frac{\theta_\delta}{\mu}
\left[
R_{\mathcal K}C_1
+C_0\delta\bigl(L_{\mathcal K}+C_1\theta_\delta\bigr)
\right].
\label{eq:compression-to-inversion-estimate}
\end{equation}
In particular, if
$E_{\mathrm{TT},\ast}=O(\delta^p)$ as $\delta\to0$, then
\begin{equation}
\|\widehat\alpha_\lambda-\alpha_\lambda\|_2
=O(\delta^{p-1}).
\label{eq:asymptotic-parameter-error-rate}
\end{equation}
\end{theorem}

\begin{proof}
By~\eqref{eq:epsilon-zero-uniform-grid}--\eqref{eq:combined-surrogate-error-rate},  we have
$
\epsilon_0^{\mathrm{rom}}\leq C_0\delta\theta_\delta$,
$\epsilon_1^{\mathrm{rom}}\leq C_1\theta_\delta$.
Substitution into~\eqref{eq:local-parameter-error-bound} proves
\eqref{eq:compression-to-inversion-estimate}.  If
$E_{\mathrm{TT},\ast}=O(\delta^p)$, then
$\theta_\delta=O(\delta^{p-1})$, which gives
\eqref{eq:asymptotic-parameter-error-rate}.
\end{proof}

Together with~\eqref{eq:total-parameter-error-bound},
\eqref{eq:compression-to-inversion-estimate} also bounds the TROM contribution
to the total error with respect to $\alpha_{\mathrm{true}}$.

\section{Forward and inverse problems}
\label{s:methods}

We consider two parameter estimation problems used to assess the proposed TROM framework: an inverse heat-transfer problem with a PDE forward model and a parameter-estimation problem governed by the FitzHugh--Nagumo ODE system.

\subsection{Inverse heat transfer in a heterogeneous medium}\label{sec:heat}

As the first example, we consider heat transfer in a heterogeneous medium governed by
\begin{equation}\label{eq:heat}
u_t - \nabla\cdot \bigl(A(x,y;\alpha)\nabla u\bigr) = 0,
\qquad (x,y)\in\Omega,\quad t\in(0,t_{\mathrm{max}}],
\end{equation}

\noindent where $\Omega=[0,2]\times[0,1]$ and $t_{\mathrm{max}}=0.5$. The conductivity field is piecewise constant,
\begin{equation}\label{eq:A}
A(x,y;\alpha)=
\begin{cases}
A_k, & (x,y)\in \Omega_k,\quad k=1,\dots,K,\\[0.2em]
1, & \text{otherwise},
\end{cases}
\end{equation}

\noindent with $A_k=1\mathrm{e}\mathord{-3}$ in all experiments. The inclusions $\Omega_k$ are circular and are parameterized by their centers and radii,
\[
\Omega_k=\{(x,y)\in\mathbb{R}^2 : (x-l_k)^2+(y-m_k)^2\le r_k^2\},
\quad
\alpha = [l_1,m_1,r_1,\dots,l_K,m_K,r_K]^\mathsf{T}.
\]

We consider $K\in\{1,2,3\}$ inclusions, so that $D=3K$. The admissible set $\mathcal A$ is a box determined by bounds on the inclusion centers and radii; we denote the range of the $i$th parameter component by $[\alpha_i^{\mathrm{min}},\alpha_i^{\mathrm{max}}]$.

Let $\Omega_{\mathrm{incl}}=\Omega_1\cup\dots\cup\Omega_K$ and $\Omega_{\mathrm{bulk}}=\Omega\setminus\Omega_{\mathrm{incl}}$. We assume that the inclusions are immersed in $\Omega$, i.e., $\partial\Omega_{\mathrm{incl}}\subset\Omega$. The temperature is prescribed on the left boundary $\Gamma_1=\{(x,y)\in\partial\Omega\,:\,x=0\}$, while homogeneous Neumann conditions are imposed on the remaining boundary:
\begin{equation}\label{eq:bc}
u=1\quad\text{on}\,\,\Gamma_1,
\qquad
\frac{\partial u}{\partial n}=0\quad\text{on}\,\,\partial\Omega\setminus\Gamma_1,
\qquad t\in[0,t_{\mathrm{max}}],
\end{equation}

\noindent with initial condition $u|_{t=0}=0$. Across the inclusion interfaces, we impose continuity of temperature and normal heat flux.

The observation data consist of temperature measurements at $N_s=9$ sensor locations on the Neumann boundary over the time interval $[0,t_{\mathrm{max}}]$:
\[
  \{p_i\}_{i=1,\dots,9}=
\big\{(\tfrac{1}{4},1),\,(\tfrac{3}{4},1),\,(\tfrac{5}{4},1),\,(\tfrac{7}{4},1),\,(\tfrac{1}{8},0),(\tfrac{9}{16},0),\,(\tfrac{17}{16},0),\,(\tfrac{25}{16},0),\, (2,\tfrac{1}{2})\big\}
\]

\noindent Thus, the inverse problem is to recover the inclusion locations and radii from boundary temperature measurements.

\begin{figure}
\centering
\begin{tikzpicture}[scale=3.5, every node/.style={font=\normalem}]
\fill[gray!10] (0,0) rectangle (2,1);
\draw[thick] (0,0) rectangle (2,1);

\draw[very thick,red] (0,0) -- (0,1);
\node[red,left] at (0,0.5) {$u=1$};



\draw[fill=blue!60] (0.3,0.30) circle (0.22);
\fill (0.3,0.30) circle (0.01);

\node[white] at (0.3,0.38) {$(l_1,m_1)$};
\draw (0.3,0.30) -- (0.3,0.08);
\node[white,right] at (0.3,0.18) {$r_1$};
\node at (0.17,0.25) {$\Omega_1$};

\draw[fill=blue!60] (0.60,0.60) circle (0.18);
\node at (0.58,0.65) {$\Omega_2$};

\draw[fill=blue!60] (0.90,0.50) circle (0.20);
\node at (0.90,0.45) {$\Omega_3$};
\node at (1.6,0.85) {$A=1$};
\node at (1.6,0.75) { in the bulk};

\draw[->, very thick, gray!70]
    (1.05,0.55) to[out=20,in=180] (1.3,0.45);

\node[right] at (1.3,0.45) {$A_3=1\mathrm{e}\mathord{-3}$};
\node[right] at (1.3,0.37) {(inclusion)};

\node at (1.35,0.18) {$u(x,y,0)=0$};

\foreach \x/\lab in {0.1250/p_5,0.5625/p_6,1.0625/p_7,1.5625/p_8}
{
    \fill[black] (\x,0) circle (0.015);
    \node[below=2pt] at (\x,0) {$\lab$};
}

\foreach \x/\lab in {0.25/p_1,0.75/p_2,1.25/p_3,1.75/p_4}
{
    \fill[black] (\x,1) circle (0.015);
    \node[above=2pt] at (\x,1) {$\lab$};
}

\fill[black] (2,0.5) circle (0.015);
\node[right=2pt] at (2,0.5) {$p_9$};


\node[blue!70!black,above=8pt] at (1,1) {$\partial_n u=0$};
\node[blue!70!black,below=10pt] at (1,0) {$\partial_n u=0$};
\node[blue!70!black,right=1pt] at (2,0.4) {$\partial_n u=0$};

\end{tikzpicture}
\caption{Forward-problem setup with three inclusions $\Omega_k$, $k=1,2,3$, and sensor locations $p_i$, $i=1,\ldots,9$ for a computational domain $\Omega=[0,2]\times[0,1]$. Dirichlet data are imposed on the left boundary $\Gamma_1$, homogeneous Neumann data on $\partial\Omega\setminus\Gamma_1$, and the inclusions have conductivity $A_k=1\mathrm{e}\mathord{-3}$, while the bulk conductivity is 1.}
\label{fig:domain}
\end{figure}
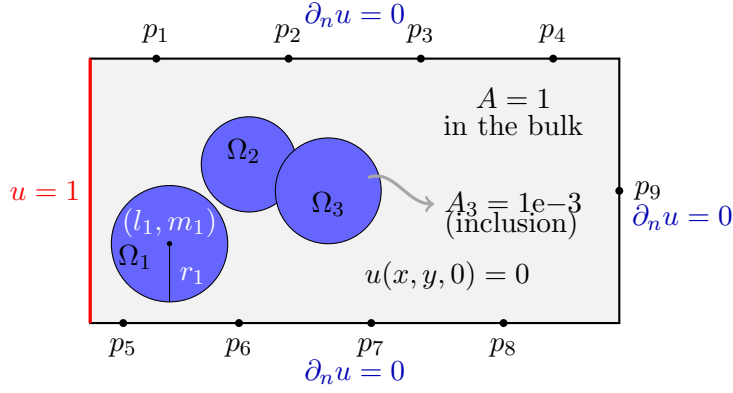

We also state the continuous mapping framework. Since the interface $\partial\Omega_{\mathrm{incl}}$ is piecewise smooth\footnote{We assume here for simplicity that $\partial\Omega_{\mathrm{incl}}$ is Lipschitz, i.e., inclusions can overlap but cannot ``touch'' each other.} and $\Omega$ is a convex polygon, maximum regularity theory for interface parabolic problems~\cite{meyries2012maximal,pruss2016moving} implies
\[
u \in L^{2}_{\omega}(0,t_{\mathrm{max}}; H^2(\Omega_{\mathrm{incl}}\cup\Omega_{\mathrm{bulk}})),
\]

\noindent where $H^2(\Omega_{\mathrm{incl}}\cup\Omega_{\mathrm{bulk}})$ is a broken Sobolev space and $L^{2}_{\omega}$ is the weighted $L^2$ space with weight $\omega=t^{2(1-\mu)}$, $\mu\in(\tfrac{1}{2},\tfrac{3}{4})$. The restriction $\mu<\tfrac{3}{4}$ is due to the incompatibility between the zero initial condition and the nonhomogeneous Dirichlet data; see, e.g., \cite[Example~3.7(i)]{lindemulder2019maximal}. By Sobolev embedding, the trace of $u$ at each sensor point belongs to $L^{2}_{\omega}(0,t_{\mathrm{max}})$, so the problem fits the abstract setting with
\[
\mathcal{U} = L^{2}_{\omega}(0,t_{\mathrm{max}}; H^2(\Omega_{\mathrm{incl}}\cup\Omega_{\mathrm{bulk}})),
\qquad
\mathcal{Y} = L^{2}_{\omega}(0,t_{\mathrm{max}})^{N_s},
\]

\noindent where $g(\alpha)$ is the solution operator for~\eqref{eq:heat}--\eqref{eq:bc} and $q(u)$ is the trace operator.

\begin{figure}
\centering
\includegraphics[width=0.9\textwidth]{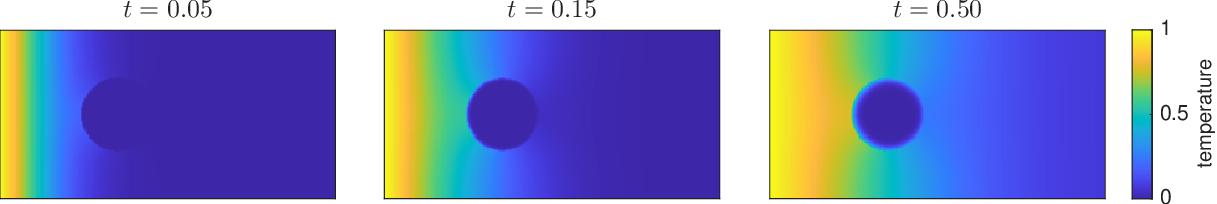}
\caption{Temperature field $u(x,y,t)$ at selected times $t \in \{0.05,0.15,0.50\}$. Heat propagates from the left boundary $\Gamma_1$, while the low-conductivity inclusion locally impedes diffusion and distorts the temperature distribution.}
\label{fig:snapshots}
\end{figure}

\begin{figure}
\centering
\includegraphics[width=0.7\textwidth]{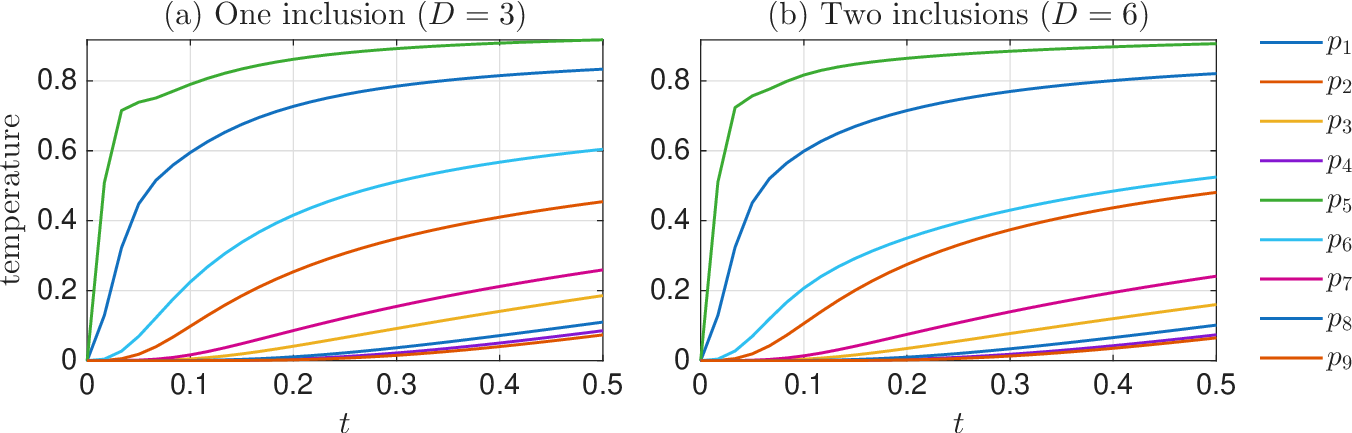}
\caption{Sensor measurements as functions of time at the locations $p_i$, $i=1,\dots,9$, for one inclusion (left; panel (a)) and two non-overlapping inclusions (right; panel (b)).}
\label{fig:sensor-data}
\end{figure}

For numerical experiments, \eqref{eq:heat}--\eqref{eq:bc} is discretized by a second-order FD scheme on a uniform Cartesian grid.  The conductivity is smoothed across a thin layer ($\eta = h$) near the inclusion interface. Time integration uses the second-order backward differentiation formula (BDF2) with a Crank--Nicolson first step. After discretization, the parameter-to-observation map is
$f:\mathcal A\to\mathbb R^{N_{\mathrm{obs}}}$,
$N_{\mathrm{obs}}=N_sN_t$,
and the observation vector satisfies $y_{\mathrm{obs}}\in\mathbb R^{N_{\mathrm{obs}}}$. This discretized map is the full-order model used in the inverse problem. The resulting temperature evolution is shown in \Cref{fig:snapshots}, and representative sensor data are shown in \Cref{fig:sensor-data}.

For the numerical experiments we consider the following parameter domains:
\[
\begin{aligned}
\mathcal A
  &=[\tfrac{3}{5},1]\times[\tfrac{3}{10},\tfrac{7}{10}]\times[\tfrac{1}{10},\tfrac{1}{4}],\\
\mathcal A
&=[\tfrac{3}{10},\tfrac{1}{2}]\times[\tfrac{3}{10},\tfrac{1}{2}]\times[\tfrac{1}{10},\tfrac{1}{4}]\times [\tfrac{7}{10},\tfrac{9}{10}]\times[\tfrac{1}{2},\tfrac{7}{10}]\times[\tfrac{1}{10},\tfrac{1}{4}],\\
\mathcal A
&=[\tfrac{1}{4},\tfrac{2}{5}]\times[\tfrac{1}{4},\tfrac{2}{5}]\times[\tfrac{1}{10},\tfrac{1}{4}]\times [\tfrac{11}{20},\tfrac{7}{10}]\times[\tfrac{9}{20},\tfrac{3}{5}]\times[\tfrac{1}{10},\tfrac{1}{4}] \times [\tfrac{4}{5},\tfrac{19}{20}]\times[\tfrac{2}{5},\tfrac{3}{5}]\times[\tfrac{1}{10},\tfrac{1}{4}],
\end{aligned}
\]

\noindent for $D=3,6,9$, respectively.

\begin{remark}[{Identifiability of multiple inclusions}]
\label[remark]{rem:non-identifiability}
{The quantity $\sigma_0$ in \Cref{rem:large-lambda-strong-convexity} degenerates for structural reasons in the multi-inclusion case.  The conductivity field~\eqref{eq:A} takes the same value $A_k=1\mathrm{e}\mathord{-3}$ in every inclusion, so it depends on $\alpha$ only through the union $\Omega_1\cup\dots\cup\Omega_K$, and overlapping inclusions are admissible. Consequently, for $K\geq2$ any configuration in which one inclusion is contained in the union of the others is non-identifiable: perturbing the center and radius of that inclusion leaves the union, and hence $f$ unchanged throughout a neighborhood; $J$ has a nontrivial null space, and since $\sigma_0$ is a minimum over all of $\mathcal A$, a single such configuration in the admissible box forces $\sigma_0=0$, so that the threshold in \Cref{rem:large-lambda-strong-convexity} reduces to $\lambda>B_{\mathcal A}$.  This is a property of the parameterization rather than of the TROM, and it affects FOM and TROM inversion in the same way.}
\end{remark}

\subsection{Parameter estimation for the FitzHugh--Nagumo model}\label{sec:FHN}

As a second example, we consider an inverse problem governed by the FitzHugh--Nagumo model, a nonlinear two-variable ODE system commonly used as a simplified model of excitable neural dynamics~\cite{fitzhugh1961impulses,nagumo1962active}. The state consists of the membrane potential $u(t)$ and a recovery variable $v(t)$, and the model is written in the form
\begin{equation}
\label{eq:fhn}
\dot u(t)
=
\alpha_2\biggl(u(t)-\frac{u(t)^3}{3}+v(t)+I\biggr),
\quad
\dot v(t)
=
-\frac{1}{\alpha_2}\bigl(u(t)-\alpha_0+\alpha_1v(t)\bigr).
\end{equation}

\noindent $t\in(0,t_{\mathrm{max}}]$, with prescribed initial conditions. Here $I$ is a fixed stimulus, while the unknown parameter vector $\alpha=(\alpha_0,\alpha_1,\alpha_2)^\mathsf{T}$ controls the excitability threshold, recovery response, and time-scale separation, respectively. We take $t_{\mathrm{max}}=100$, $I=-0.4$, and fix the scale-separation parameter at $\alpha_2=3$; hence, in this problem, $\mathcal A$ is a box in $\mathbb R^2$. Following~\cite{rudi2022parameter,villalobos2026neural} we let
$\mathcal A =[-0.2,1]\times [-0.4,1.2]$.

The observation data are time-series measurements of the membrane potential,
\[
y_{\mathrm{obs}} =
\bigl(u(t_1;\alpha),\ldots,u(t_{N_t};\alpha)\bigr)^\mathsf{T}
+\varepsilon_{\mathrm{meas}},
\]

\noindent so that the discretized parameter-to-observation map is
\[
f:\mathcal A\to\mathbb R^{N_t},
\qquad
f(\alpha)=
\bigl(u(t_1;\alpha),\ldots,u(t_{N_t};\alpha)\bigr)^\mathsf{T}.
\]

\noindent The inverse problem is to recover $\alpha$ from the observed voltage trace. The FOM is given by a 4th order Runge–Kutta numerical integrator with $\Delta t = t_{\mathrm{max}}/(N_t-1)$.

\begin{figure}
\centering
\includegraphics[width=0.36\textwidth]{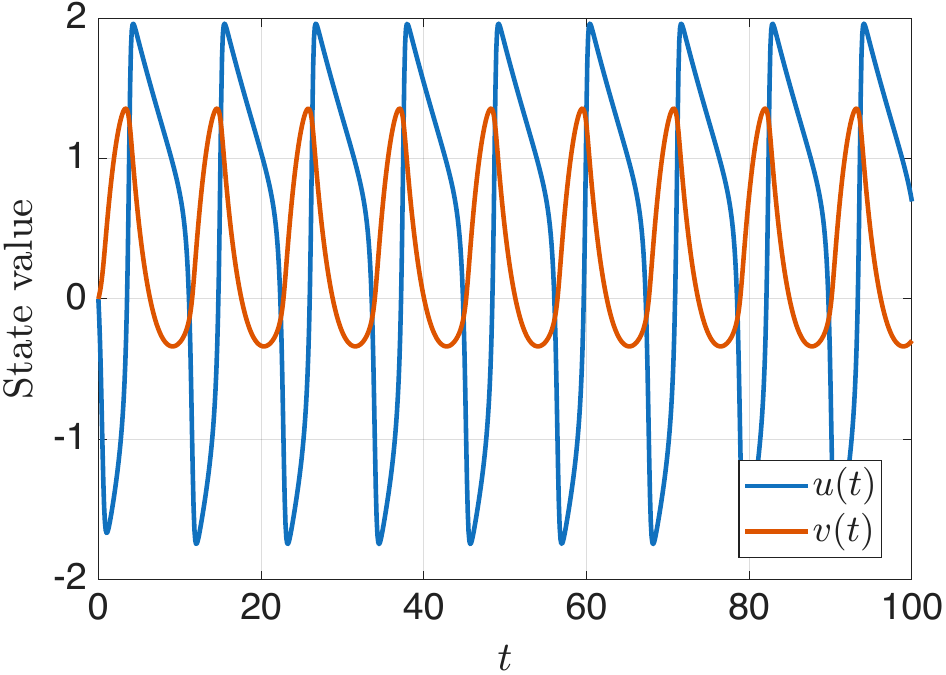}
\quad\includegraphics[width=0.4\textwidth]{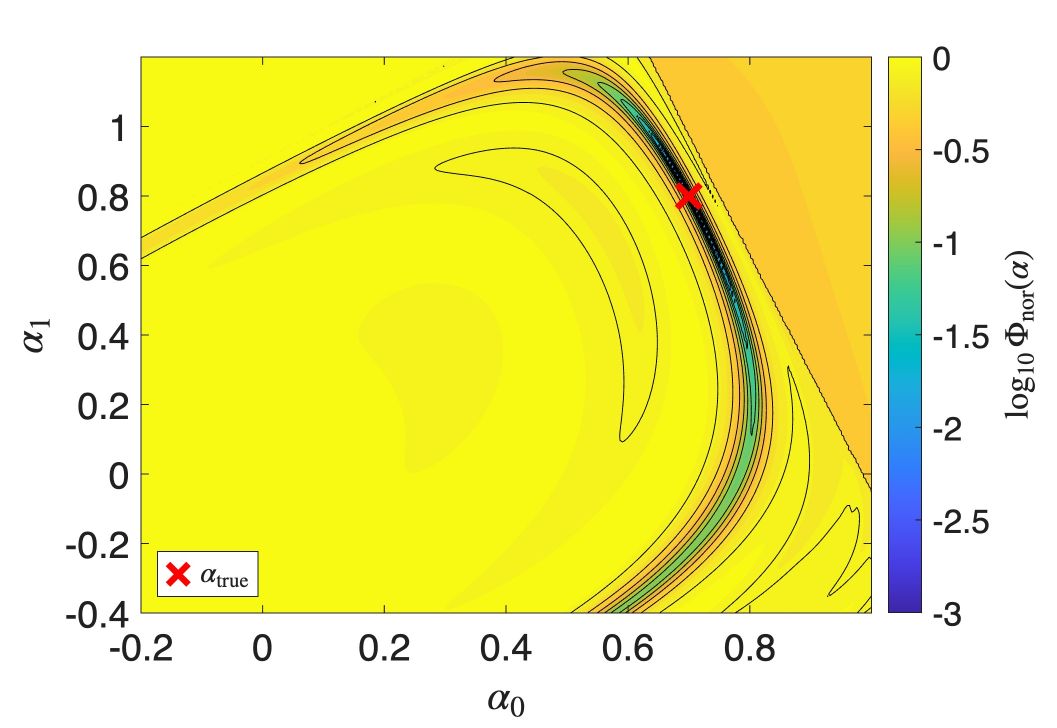}
\caption{FitzHugh--Nagumo inverse problem. Left: Representative solution trajectory. Right: Normalized least-squares objective landscape,
$
\Phi_{\mathrm{nor}}(\alpha)
=
\|f(\alpha)-y_{\mathrm{obs}}\|_2^2/\|y_{\mathrm{obs}}\|_2^2,
$
  evaluated over the displayed parameter plane and plotted in log-scale. Here, $y_{\mathrm{obs}}=f(\alpha_{\mathrm{true}})+\varepsilon_{\mathrm{meas}}$, $\alpha_{\mathrm{true}}=(0.7,0.8)^\mathsf{T}$.
\label{fig:fhn_problem}
}
\end{figure}

Although the forward model is low-dimensional, this inverse problem is challenging for classical optimization-based inversion. The parameter-to-observation map is strongly nonlinear, the parameters enter the dynamics in a coupled way, and small parameter changes can produce sharp changes in the spiking response. As a result, the least-squares objective can be highly nonconvex, with narrow valleys, steep gradients, and multiple local minima; see~\Cref{fig:fhn_problem}. These features motivate alternative approaches based on global, surrogate-assisted, or learning-based parameter identification~\cite{rudi2022parameter,villalobos2026neural}.

Thus, this problem provides a useful test case for assessing whether a deeper integration of TROM into the inversion procedure can mitigate some of the difficulties encountered by classical least-squares minimization.

\section{Numerical results}
\label{s:results}

We present numerical experiments assessing TROM accuracy, inversion robustness, online speedup, and scalability. The tests compare TT-SVD and TT-Cross surrogates, study the effect of TROM error, measurement noise, initialization, and regularization, and demonstrate the additional speedup obtained through the deeper TROM integration into the minimization process.

\subsection{Hardware, Software, and Setup}

All timings are reported in seconds. The computations were performed in MATLAB R2025b using 12 parallel workers on a MacBook Pro with an Apple M4 Pro processor and 24 GB of RAM. For each fixed parameter value, the forward solver computes one LU factorization at the first time step and reuses it at subsequent time steps.

Throughout this section we use the same number of nodes in every parameter direction and write $N\defeq N_1=\cdots=N_D$ for this common value; $h$ denotes the spatial mesh size of the FOM. For the heat-transfer problem the observations consist of $N_s=9$ sensors sampled at $N_t=31$ times, so that $N_{\mathrm{obs}}=N_sN_t=279$ unless stated otherwise.

\subsection{Performance of TROM surrogates}\label{sec:TROMper}
\begin{table}
\tabadjust
\caption{Heat-transfer problem: Offline-stage results for the
single-inclusion case ($D=3$) using TT-SVD and TT-Cross, with
$\varepsilon_{\mathrm{TT}}{}=\varepsilon_{\mathrm{TTC}}=10^{-3}$. The table reports TT ranks,
the sample-based relative TROM error $\mathrm{err}_{\mathrm{trom}}$, and the
offline time in seconds as functions of $N$ and the mesh size $h$.}
\label{t:single_inclusion_offline_table}
\begin{tabular}{ccccr|ccr}
\toprule
& & \multicolumn{3}{c|}{\textbf{TT-SVD}}
  & \multicolumn{3}{c}{\textbf{TT-Cross}} \\
\cmidrule(lr){3-5} \cmidrule(lr){6-8}
$N$ & $h$ & \bf ranks & \bf $\mathrm{err}_{\mathrm{trom}}$ & \bf time
          & \bf ranks & \bf $\mathrm{err}_{\mathrm{trom}}$ & \bf time\\
\midrule
4  & 1/16  & [1\,10\,5\,3\,1] & \tnum{2.790e-03} & 0.16
           & [1\,14\,8\,3\,1] & \tnum{2.914e-03} & 1 \\
   & 1/32  & [1\,11\,6\,3\,1] & \tnum{3.222e-03} & 0.71
           & [1\,14\,7\,4\,1] & \tnum{3.356e-03} & 2 \\
   & 1/64  & [1\,11\,6\,3\,1] & \tnum{3.350e-03} & 6
           & [1\,16\,9\,3\,1] & \tnum{3.307e-03} & 17 \\
   & 1/128 & [1\,11\,6\,3\,1] & \tnum{3.555e-03} & 76
           & [1\,15\,9\,3\,1] & \tnum{3.533e-03} & 225 \\
\midrule
8  & 1/16  & [1\,9\,5\,2\,1]  & \tnum{1.272e-03} & 1
           & [1\,15\,9\,3\,1] & \tnum{9.659e-04} & 2 \\
   & 1/32  & [1\,10\,6\,3\,1] & \tnum{1.154e-03} & 3
           & [1\,13\,8\,3\,1] & \tnum{1.239e-03} & 12 \\
   & 1/64  & [1\,10\,6\,3\,1] & \tnum{1.434e-03} & 31
           & [1\,13\,9\,3\,1] & \tnum{1.201e-03} & 121 \\
   & 1/128 & [1\,11\,6\,3\,1] & \tnum{1.213e-03} & 418
           & [1\,15\,9\,3\,1] & \tnum{1.205e-03} & 2136 \\
\midrule
12 & 1/16  & [1\,9\,5\,2\,1]  & \tnum{1.228e-03} & 2
           & [1\,15\,9\,3\,1] & \tnum{8.572e-04} & 5 \\
   & 1/32  & [1\,10\,6\,2\,1] & \tnum{1.330e-03} & 6
           & [1\,14\,9\,4\,1] & \tnum{9.354e-04} & 34 \\
   & 1/64  & [1\,10\,6\,2\,1] & \tnum{1.311e-03} & 61
           & [1\,15\,9\,3\,1] & \tnum{7.755e-04} & 363 \\
   & 1/128 & [1\,10\,6\,2\,1] & \tnum{1.343e-03} & 816
           & [1\,15\,9\,4\,1] & \tnum{7.165e-04} & 4637 \\
\bottomrule
\end{tabular}
\end{table}

\begin{table}
\tabadjust
\caption{Heat-transfer problem: Offline-stage results for the
two-inclusion case ($D=6$) using TT-SVD and TT-Cross, with
$\varepsilon_{\mathrm{TT}}{}=\varepsilon_{\mathrm{TTC}}=10^{-3}$. The table reports TT ranks,
the sample-based relative TROM error, and the offline time in seconds
as a function of $N$ and the mesh size $h$.}
\label{t:double_inclusion_offline_table}
\begin{tabular}{@{}ccccr|ccr@{}}
\toprule
& & \multicolumn{3}{c|}{\textbf{TT-SVD}}
  & \multicolumn{3}{c}{\textbf{TT-Cross}} \\
\cmidrule(lr){3-5} \cmidrule(l){6-8}
$N$ & $h$
& \bf ranks & \bf $\mathrm{err}_{\mathrm{trom}}$ & \bf time
& \bf ranks & \bf $\mathrm{err}_{\mathrm{trom}}$ & \bf time \\
\midrule
4 & 1/16
& $[1\,14\,17\,16\,10\,6\,3\,1]$ & \tnum{1.554e-03} & 3
& $[1\,21\,29\,30\,19\,11\,4\,1]$ & \tnum{1.286e-03} & 11 \\

  & 1/32
& $[1\,15\,21\,21\,13\,7\,3\,1]$ & \tnum{2.982e-03} & 11
& $[1\,16\,16\,16\,16\,14\,4\,1]$ & \tnum{7.442e-03} & 54 \\

  & 1/64
& $[1\,15\,23\,26\,18\,9\,3\,1]$ & \tnum{3.026e-03} & 116
& $[1\,12\,12\,12\,12\,12\,4\,1]$ & \tnum{9.732e-03} & 484 \\

  & 1/128
& $[1\,16\,25\,31\,21\,10\,3\,1]$ & \tnum{3.974e-03} & 1592
& $[1\,18\,22\,22\,22\,16\,4\,1]$ & \tnum{8.279e-03} & 9121 \\
\midrule
8 & 1/16
& $[1\,13\,15\,14\,9\,5\,3\,1]$ & \tnum{1.267e-03} & 166
& $[1\,17\,18\,18\,17\,10\,3\,1]$ & \tnum{2.439e-03} & 52 \\

  & 1/32
& $[1\,13\,18\,18\,12\,6\,3\,1]$ & \tnum{1.770e-03} & 597
& $[1\,16\,16\,16\,16\,15\,5\,1]$ & \tnum{7.049e-03} & 330 \\

  & 1/64
& $[1\,14\,22\,23\,15\,8\,3\,1]$ & \tnum{1.656e-03} & 6633
& $[1\,20\,22\,22\,22\,22\,6\,1]$ & \tnum{5.008e-03} & 6413 \\

  & 1/128
& -- & -- & --
& $[1\,17\,18\,18\,18\,18\,8\,1]$ & \tnum{8.313e-03} & 60637 \\
\midrule
12 & 1/16
& $[1\,13\,14\,13\,8\,5\,3\,1]$ & \tnum{1.503e-03} & 995
& $[1\,19\,29\,30\,20\,10\,3\,1]$ & \tnum{1.251e-03} & 223 \\

   & 1/32
& $[1\,13\,18\,17\,11\,6\,3\,1]$ & \tnum{1.877e-03} & 7288
& $[1\,18\,18\,18\,18\,16\,5\,1]$ & \tnum{3.807e-03} & 902 \\

   & 1/64
& $[1\,14\,22\,22\,15\,8\,3\,1]$ & \tnum{1.864e-03} & 76737
& $[1\,15\,16\,16\,16\,16\,6\,1]$ & \tnum{7.601e-03} & 7764 \\
\bottomrule
\end{tabular}
\end{table}

\begin{table}
\tabadjust
\caption{Heat-transfer problem: Offline-stage results for the
three-inclusion case ($D=9$) using the TT-Cross surrogate, with
$\varepsilon_{\mathrm{TTC}}=10^{-3}$. The table reports TT ranks,
the number of FOM calls required to build the TT-Cross surrogate,
the sample-based relative TROM error, and the offline time in seconds
as a function of $N$ and the mesh size $h$.}
\label{t:triple_inclusion_offline_table}
\begin{tabular}{cccrcr}
\toprule
$N$ & $h$ & \bf ranks & \bf FOM calls
& \bf $\mathrm{err}_{\mathrm{trom}}$ & \bf offline time \\
\midrule
\multirow{3}{*}{4}
& 1/16 & $[1\,18\,18\,18\,18\,18\,18\,18\,12\,4\,1]$
& \inum{19217}
& \tnum{5.550e-03}
& 54 \\

& 1/32 & $[1\,20\,20\,20\,20\,20\,20\,20\,16\,4\,1]$
& \inum{22749}
& \tnum{1.449e-02}
& 438 \\

& 1/64 & $[1\,19\,20\,20\,20\,20\,20\,20\,16\,4\,1]$
& \inum{22525}
& \tnum{1.321e-02}
& 4450 \\
\midrule
\multirow{3}{*}{6}
& 1/16 & $[1\,18\,18\,18\,18\,18\,18\,18\,11\,4\,1]$
& \inum{41851}
& \tnum{5.535e-03}
& 112 \\

& 1/32 & $[1\,22\,26\,26\,26\,26\,26\,26\,20\,5\,1]$
& \inum{87900}
& \tnum{8.419e-03}
& 1656 \\

& 1/64 & $[1\,14\,14\,14\,14\,14\,14\,14\,14\,6\,1]$
& \inum{28405}
& \tnum{2.071e-02}
& 5503 \\
\bottomrule
\end{tabular}
\end{table}

\begin{table}
\tabadjust
\caption{FitzHugh--Nagumo model: Offline-stage results using TT-SVD.
The sample-based relative TROM error is computed using 100 randomly
chosen off-grid parameter samples. The left block reports results for
varying $N$ with fixed $\varepsilon_{\mathrm{TT}}=10^{-3}$; the right
block reports results for varying $\varepsilon_{\mathrm{TT}}$ with
fixed $N=800$.}
\label{t:fhn_rom_accuracy}
\begin{tabular}{@{}cccr|cccr@{}}
\toprule
\multicolumn{4}{c|}{\textbf{varying $N$, $\varepsilon_{\mathrm{TT}}=10^{-3}$}}
&
\multicolumn{4}{c}{\textbf{varying $\varepsilon_{\mathrm{TT}}$, $N=800$}} \\
\cmidrule(lr){1-4} \cmidrule(l){5-8}
$N$ & \bf ranks & \bf $\mathrm{err}_{\mathrm{trom}}$ & \bf time
& $\varepsilon_{\mathrm{TT}}$ & \bf ranks
& \bf $\mathrm{err}_{\mathrm{trom}}$ & \bf time \\
\midrule
12
& $[1\,123\,12\,1]$  & \tnum{1.053e+00} & 0.58
& $\text{1e--2}$ & $[1\,265\,499\,1]$ & \tnum{2.492e-02} & 133.23 \\

25
& $[1\,382\,25\,1]$  & \tnum{9.593e-01} & 1.65
& $\text{1e--3}$ & $[1\,489\,679\,1]$ & \tnum{2.395e-02} & 147.69 \\

50
& $[1\,474\,50\,1]$  & \tnum{5.504e-01} & 1.60
& $\text{1e--4}$ & $[1\,721\,718\,1]$ & \tnum{2.396e-02} & 197.69 \\

100
& $[1\,486\,100\,1]$ & \tnum{5.965e-01} & 2.94
& $\text{1e--5}$ & $[1\,933\,751\,1]$ & \tnum{2.396e-02} & 239.71 \\

200
& $[1\,488\,195\,1]$ & \tnum{1.373e-01} & 7.30
& --- & --- & --- & --- \\

400
& $[1\,489\,365\,1]$ & \tnum{7.764e-02} & 31.67
& --- & --- & --- & --- \\

800
& $[1\,489\,679\,1]$ & \tnum{2.395e-02} & 147.69
& --- & --- & --- & --- \\

\bottomrule
\end{tabular}
\end{table}

We first assess the offline construction of the TROM surrogates. For the heat-transfer problem, \Cref{t:single_inclusion_offline_table,t:double_inclusion_offline_table,t:triple_inclusion_offline_table} report results for $D=3,6,9$, respectively, using $\varepsilon_{\mathrm{TT}}{}=\varepsilon_{\mathrm{TTC}}=10^{-3}$.
We are interested in the resulting TT ranks, the offline construction time, and the maximum relative TROM error over 100 off-grid test parameters defined as
\[
\mathrm{err}_{\mathrm{trom}}
=
\frac{
\displaystyle\max_{1\leq m\leq 100}
\|\widehat f(\alpha^{(m)})-f(\alpha^{(m)})\|_2
}{
\displaystyle\max_{1\leq m\leq 100}
\|f(\alpha^{(m)})\|_2}
\]
where $\{\alpha^{(m)}\}_{m=1}^{100}$ denotes the set of off-grid test parameters.

For $D=3$, both TT-SVD and TT-Cross are inexpensive and achieve comparable accuracy. TT-SVD is usually faster because the full tensor grid is still affordable, whereas TT-Cross carries adaptive-sampling overhead. For $D=6$, the ranks and offline times increase substantially. TT-SVD often gives smaller TROM errors, but requires the full tensor grid of size $N^D$; TT-Cross avoids this cost by using adaptively selected FOM evaluations. For example, when $N=12$ and $D=6$, TT-SVD requires \inum{2985984} FOM solves, whereas TT-Cross uses \inum{82956}, \inum{46994}, and \inum{39950} FOM evaluations for $h=1/16$, $1/32$, and $1/64$, respectively.

For $D=9$, full tensor construction is no longer practical, so only TT-Cross is reported. With $N=6$, the full tensor grid would contain $6^9 \approx 10^7$ parameter samples, whereas TT-Cross uses between $\tnum{2.8e4}$ and $\tnum{8.79e4}$ FOM calls in the reported tests. Thus, although the offline cost remains substantial, TT-Cross extends the approach to parameter dimensions where TT-SVD is infeasible.

For the FitzHugh--Nagumo model, \Cref{t:fhn_rom_accuracy} shows that the dominant accuracy improvement comes from refining the parameter grid. Decreasing $\varepsilon_{\mathrm{TT}}$ below $10^{-3}$ gives only a minor additional gain while increasing the TT ranks and construction time. Therefore, we use $\varepsilon_{\mathrm{TT}}=10^{-3}$ in the remaining FitzHugh--Nagumo experiments.

Up to a scaling factor, the reported error $\mathrm{err}_{\mathrm{trom}}$ is an (approximation of) the uniform absolute TROM map error $\epsilon_0^{\mathrm{rom}}$ entering the inversion
bound~\eqref{eq:local-parameter-error-bound}. The estimate~\eqref{eq:epsilon-zero-uniform-grid} separates $\epsilon_0^{\mathrm{rom}}$ into tensor-compression and interpolation contributions. The results in \Crefrange{t:single_inclusion_offline_table}{t:fhn_rom_accuracy} are consistent with this analysis. For the heat-transfer problem, the smoothness of the parameter-to-observation map and the weak dependence
of the error on the parameter-grid resolution suggest that the compression term $E_{\mathrm{TT},\ast}$ dominates over the tested grids. Accordingly, $\mathrm{err}_{\mathrm{trom}}$ remains comparable in order of
magnitude to the prescribed compression tolerance $\varepsilon_{\mathrm{TT}}=10^{-3}$, except for $D=9$, $N=6$, where this target was not achieved. For the FitzHugh--Nagumo model, the sharp variation of the map makes the interpolation contribution significant. Refining the parameter grid therefore decreases
$\mathrm{err}_{\mathrm{trom}}$ at a rate approaching quadratic for $p=2$, whereas tightening the compression tolerance $\varepsilon_{\mathrm{TT}}$ has only a minor effect.

\subsection{TROM-based inversion with an informed prior}
\label{exp_1}

We first consider the idealized informed-prior case $\alpha_{\mathrm{ref}}=\alpha_{\mathrm{true}}$. This setting isolates the effects of TROM error, spatial discretization, parameter dimension, and initialization. The initial guess is sampled uniformly from $\mathcal A$, with candidates too close to $\alpha_{\mathrm{true}}$ rejected according to
\[
\left\lVert
\frac{\alpha^{(0)}-\alpha_{\mathrm{true}}}{\alpha^{\mathrm{max}}-\alpha^{\mathrm{min}}}
\right\rVert_2
\ge 0.25,
\]

\noindent where the division is componentwise. For a recovered parameter $\alpha_\lambda$, we define the relative parameter error by
\begin{equation}
\label{eq:rel_param_error}
e_{\mathrm{rel}}(\alpha_\lambda)
=
\frac{
\big(\sum_{j=1}^D|\alpha_{\lambda,j}-\alpha_{\mathrm{true},j}|^2\big)^{1/2}
}{
\big(\sum_{j=1}^D|\alpha_{\mathrm{true},j}|^2\big)^{1/2}
}.
\end{equation}

\noindent Multiple random initializations are used to assess robustness with respect to the starting point.

\subsubsection{Consistency test}
As a solver validation, we first generate the observations with the same TROM surrogate used in the inversion,
$y_{\mathrm{obs}}=\widehat f(\alpha_{\mathrm{true}})$.
For the single-inclusion case with $D=3$, $h=1/64$, $N=8$, and the TT-Cross surrogate, all random initializations converge to $\alpha_{\mathrm{true}}$ up to machine precision, with mean componentwise relative errors of order $10^{-12}$ and four GN iterations. This confirms the consistency of the surrogate evaluation, Jacobian construction, and optimization algorithm.

\subsubsection{Numerical verification of the inversion error estimates}
\label{sec:inversion-error-estimate-verification}
We next verify the compression-to-inversion estimate
\eqref{eq:compression-to-inversion-estimate} and the asymptotic rate
\eqref{eq:asymptotic-parameter-error-rate}.  We use the single-inclusion
heat-transfer problem ($D=3$) with $h=1/64$, $N_t=31$, noise-free data,
$\lambda=0.1$, and $\alpha_{\mathrm{ref}}=(0.8,0.5,0.175)^\mathsf{T}$.
Two-point linear ($p=2$) and three-point quadratic ($p=3$) interpolation
are considered.  For each of 100  off-grid truth parameters, define
\[
e_{\mathrm{inv}}^{(m)}
=
\left\|
\alpha_{\lambda}^{\mathrm{FOM},(m)}
-\alpha_{\lambda}^{\mathrm{TROM},(m)}
\right\|_2,
\qquad
e_{\mathrm{inv}}^{\mathrm{max}}
=
\max_{1\leq m\leq100}e_{\mathrm{inv}}^{(m)}.
\]
For a grid with $N$ nodes in each parameter direction, let
\[
\delta_N
=
\max_{1\leq i\leq D}\max_{1\leq j<N}
\left|\halpha_i^{j+1}-\halpha_i^j\right|.
\]
With $\varepsilon_{\mathrm{TT},\ast}$, the best constant from \eqref{eq:max-fiber-tensor-compression}, and $E_{\mathrm{TT},\ast}$  defined in \eqref{ETT}, the quantity governing~\eqref{eq:compression-to-inversion-estimate} is
$\theta_{\delta_N}=\delta_N^{-1}E_{\mathrm{TT},\ast}
+\delta_N^{p-1}$.

\begin{figure}
\centering
\begin{minipage}[t]{0.45\textwidth}
  \centering
  \includegraphics[width=\linewidth]{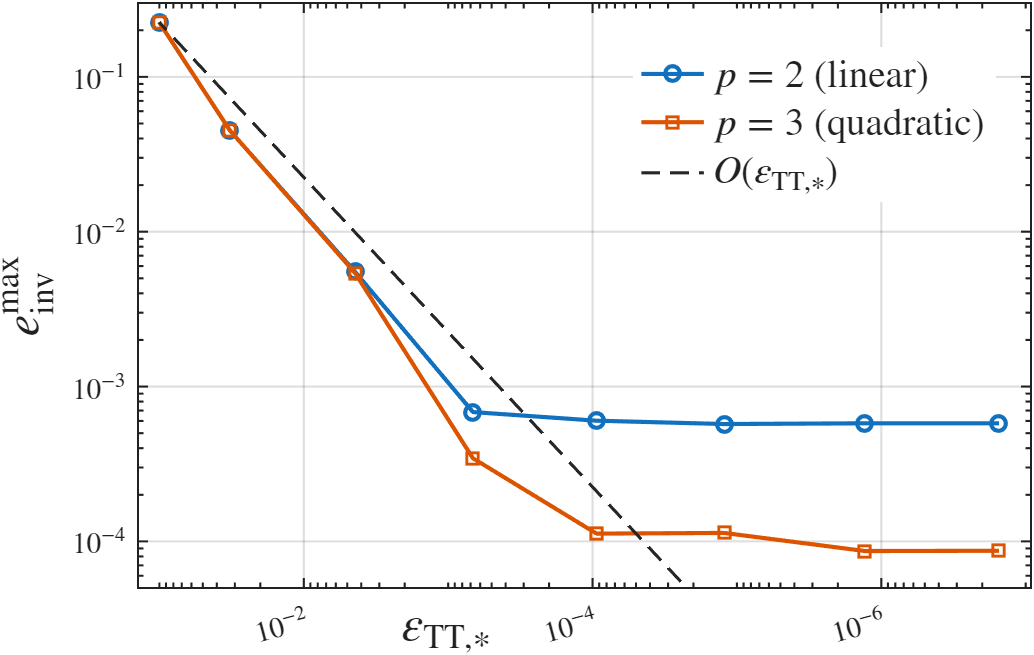}\\[-0.3em]
  {\small (a)}
\end{minipage}\hfill
\begin{minipage}[t]{0.45\textwidth}
  \centering
  \includegraphics[width=\linewidth]{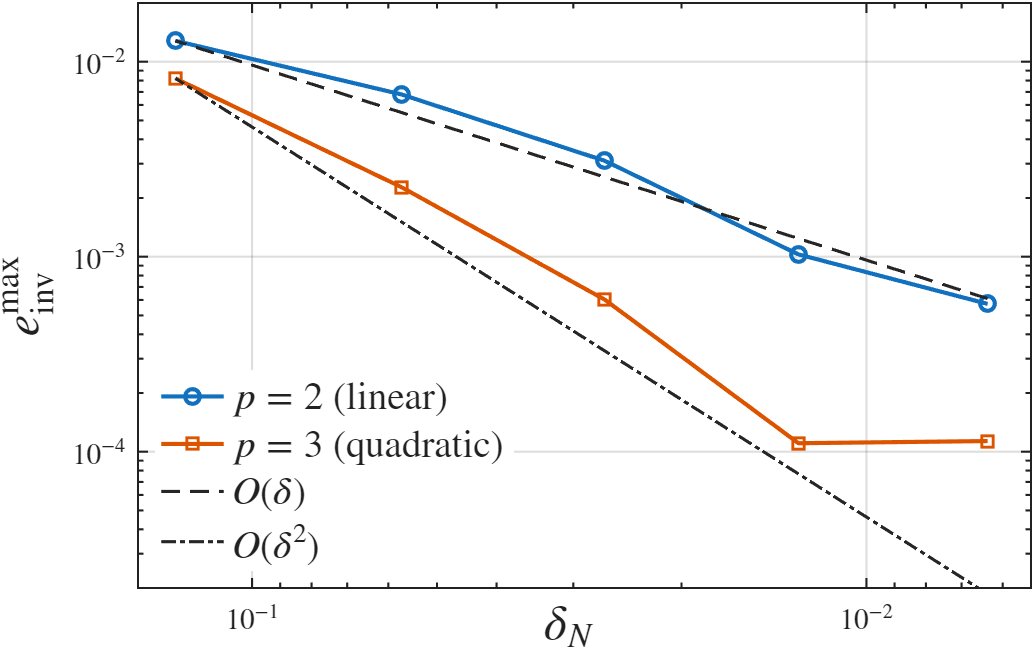}\\[-0.3em]
  {\small (b)}
\end{minipage}
\begin{minipage}[t]{0.45\textwidth}
  \centering
  \includegraphics[width=\linewidth]{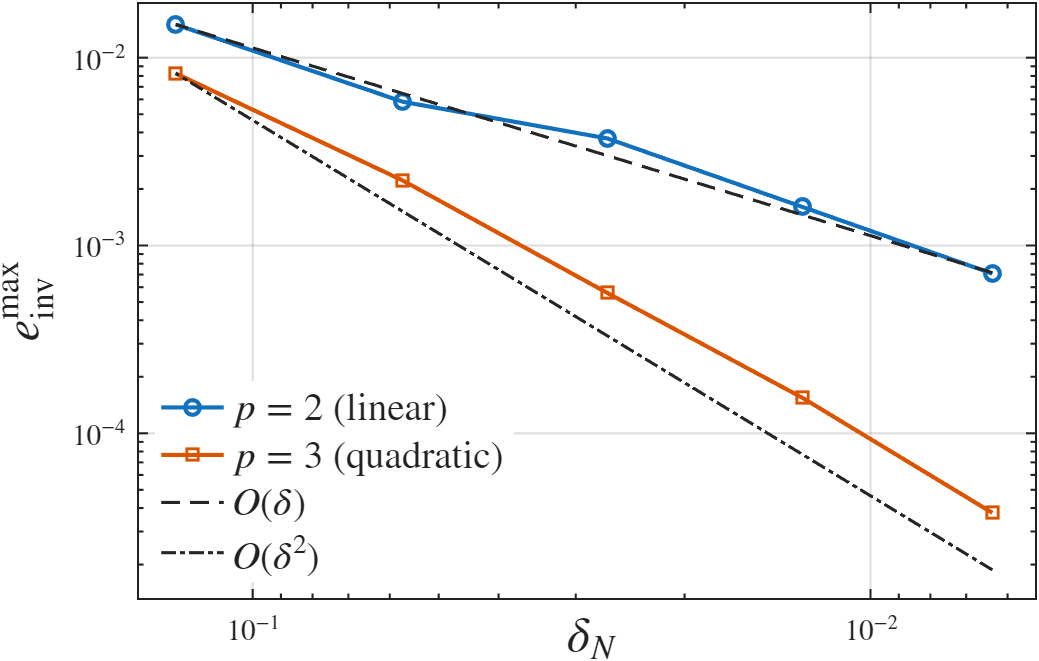}\\[-0.3em]
  {\small (c)}
\end{minipage}
\caption{Numerical verification of the inversion error estimates.
(a) Maximum inversion error versus the relative
compression error $\varepsilon_{\mathrm{TT},\ast}$ for $N=64$. 
(b) Maximum inversion error versus $\delta_N$ for
$\varepsilon_{\mathrm{TT}}=10^{-6}$. 
(c) Maximum inversion error under coupled grid and compression refinement,
with $\varepsilon_{\mathrm{TT},\ast}\approx(0.1/N)^p$; the reference lines
show the rates predicted by~\eqref{eq:asymptotic-parameter-error-rate}.}
\label{fig:inversion-error-estimate-verification}
\end{figure}

In \Cref{fig:inversion-error-estimate-verification}(a), the fine parameter
grid makes the interpolation term small.  For loose compression, the two
interpolation schemes give nearly identical errors and
$e_{\mathrm{inv}}^{\mathrm{max}}$ decreases approximately linearly with
$\varepsilon_{\mathrm{TT},\ast}$.  Since $\delta_N$ and
$\|\mathcal T\|_\ast$ are fixed, this verifies the contribution of
$\delta_N^{-1}E_{\mathrm{TT},\ast}$ to $\theta_{\delta_N}$.  At tighter
compression, the curves reach distinct interpolation floors.  Conversely,
panel~(b) uses a tight TT-SVD tolerance and shows the
$\delta_N^{p-1}$ contribution: the observed rates are $O(\delta_N)$ for
$p=2$ and approximately $O(\delta_N^2)$ for $p=3$ until the finest grid,
where compression is no longer negligible.  Finally, in panel~(c), the
coupled choice $\varepsilon_{\mathrm{TT},\ast}=O(\delta_N^p)$ and the
uniform boundedness of $\|\mathcal T\|_\ast$ give
$E_{\mathrm{TT},\ast}=O(\delta_N^p)$, so both terms in
$\theta_{\delta_N}$ are $O(\delta_N^{p-1})$.  The inversion errors follow the
predicted $O(\delta_N)$ and $O(\delta_N^2)$ rates over the resolved range,
confirming~\eqref{eq:asymptotic-parameter-error-rate}.

\subsubsection{ROM inversion using FOM-generated observations}

Now, we consider a setting in which possibly noisy observations are generated by the FOM,
$y_{\mathrm{obs}}
= f(\alpha_{\mathrm{true}})+\varepsilon_{\mathrm{meas}}$,
whereas inversion is performed using the TROM surrogate. This introduces a model mismatch, characterized by the ROM error level $\delta_{\mathrm{rom}} > 0$.

We consider the single- ($D=3$) and two-inclusion ($D=6$) cases. The true parameters, chosen outside the training grid, are
$\alpha_{\mathrm{true}}
= [0.80, 0.50, 0.22]^\mathsf{T}
$ for $D=3$ and
$\alpha_{\mathrm{true}}
= [0.43,0.41,0.21,0.81,0.60,0.21]^\mathsf{T}$
for $D=6$.
For $D=3$, we use $h=1/64$ and $N=8$; for $D=6$, we use $N=8$ and $h=1/16,1/32,1/64$. Both noise-free observations and noisy observations with $\delta_{\mathrm{meas}}=10^{-2}$ are considered. Since $\alpha_{\mathrm{ref}}=\alpha_{\mathrm{true}}$, we set $\lambda=1$. The results are reported in \Cref{tab:fom_rom_summary}.

\begin{table}
\tabadjust
  \caption{Summary of TROM-based inversion with FOM-generated observations for heat-transfer problem with noise-free (noise=0) and noisy data (noise=1). Reported quantities are averaged over 10 random initializations and, for noisy tests, also over five noise realizations. We report the mean final residual, mean relative parameter error, standard deviation of the parameter error, mean number of iterations, and mean inference time in seconds.}
\label{tab:fom_rom_summary}
\begin{tabular}{ccccccccc}\toprule
\bf noise & \bf surrogate & $h$ & $\rho$ & \bf residual & \bf error & \bf error (std) & \bf iter & \bf time \\
\midrule
\rowcolor{gray!30}\multicolumn{9}{c}{$D=3$}\\
\midrule
0 & TT-SVD   & $1/64$ & \tnum{3.923e-03} & \tnum{3.800e-03} & \tnum{5.9321e-04} & \tnum{3.1882e-09} & 4.00 & \tnum{3.60e-04} \\
1 &          &        & \tnum{6.844e-02} & \tnum{6.811e-02} & \tnum{4.3699e-03} & \tnum{1.1870e-03} & 4.00 & \tnum{1.60e-04} \\
0 & TT-Cross &        & \tnum{6.186e-03} & \tnum{6.031e-03} & \tnum{7.6850e-04} & \tnum{6.0867e-09} & 4.00 & \tnum{6.70e-04} \\
1 &          &        & \tnum{6.860e-02} & \tnum{6.806e-02} & \tnum{4.3297e-03} & \tnum{1.5278e-03} & 4.00 & \tnum{2.80e-04} \\
\midrule
\rowcolor{gray!30}
\multicolumn{9}{c}{$D=6$}\\
\midrule
0 &  TT-SVD  & $1/64$ & \tnum{5.150e-03} & \tnum{3.507e-03} & \tnum{1.2017e-03} & \tnum{8.9851e-08} & 4.20 & \tnum{8.00e-04} \\
1 &          &        & \tnum{6.718e-02} & \tnum{6.636e-02} & \tnum{3.2563e-03} & \tnum{1.2620e-03} & 4.22 & \tnum{3.90e-04} \\
0 & TT-Cross & $1/64$ & \tnum{2.718e-02} & \tnum{1.630e-02} & \tnum{4.5308e-03} & \tnum{5.8769e-07} & 5.50 & \tnum{1.05e-03} \\
1 &          &        & \tnum{7.229e-02} & \tnum{6.812e-02} & \tnum{6.4930e-03} & \tnum{9.9790e-04} & 5.56 & \tnum{4.50e-04} \\
\midrule
1 &  TT-Cross & $1/16$ & \tnum{6.682e-02} & \tnum{6.559e-02} & \tnum{3.9553e-03} & \tnum{9.2791e-04} & 4.04 & \tnum{3.80e-04} \\
1 &           & $1/32$ & \tnum{6.951e-02} & \tnum{6.742e-02} & \tnum{5.3239e-03} & \tnum{8.6163e-04} & 4.66 & \tnum{4.10e-04} \\
1 &           & $1/64$ & \tnum{7.229e-02} & \tnum{6.812e-02} & \tnum{6.4930e-03} & \tnum{9.9790e-04} & 5.56 & \tnum{4.50e-04} \\
\midrule
\rowcolor{gray!30}
\multicolumn{9}{c}{$D=9$}\\
\midrule
1 & TT-Cross & $1/16$ & \tnum{6.846e-02} & \tnum{6.691e-02} & \tnum{5.2682e-03} & \tnum{9.9088e-04} & 4.44 & \tnum{1.56e-03} \\
1 &          & $1/32$ & \tnum{6.849e-02} & \tnum{6.546e-02} & \tnum{5.3591e-03} & \tnum{5.7634e-04} & 4.84 & \tnum{7.60e-04} \\
1 &          & $1/64$ & \tnum{1.482e-01} & \tnum{7.510e-02} & \tnum{2.4641e-02} & \tnum{8.5996e-04} & 7.38 & \tnum{8.20e-04} \\
\bottomrule
\end{tabular}
\end{table}

For $D=3$, both TT-SVD and TT-Cross give stable reconstructions in the noise-free and noisy cases. In the noise-free case, the recovered parameters are close to $\alpha_{\mathrm{true}}$; TT-SVD gives a slightly smaller error, consistent with its smaller TROM approximation error at the truth. The final residuals are close to the prescribed discrepancy levels $\rho$ from~\eqref{e:discrepancy2} (see \Cref{sec:discrepancy}).

The same qualitative behavior is observed for $D=6$, although the effect of surrogate error becomes more visible. In the noise-free case, TT-SVD gives smaller parameter error and residual than TT-Cross, while both methods remain stable and require only a few iterations. For noisy data, the residuals again track the prescribed discrepancy levels. The TT-Cross tests with varying $h$ show only mild changes in runtime and iteration count, while the parameter error reflects the combined influence of surrogate error, noise, and spatial discretization. We omit the corresponding TT-SVD and noise-free results for varying $h$, since they exhibit the same trends, with somewhat smaller errors and error standard deviations.

The three-inclusion case ($D=9$) is more challenging; see \Cref{tab:fom_rom_summary}. For $h=1/16$ and $h=1/32$, TT-Cross gives errors comparable to the $D=6$ case and residuals close to the prescribed discrepancy levels. For $h=1/64$, the discrepancy level and parameter error increase noticeably, indicating greater sensitivity to surrogate accuracy, spatial discretization, and regularization. Even then, the online TROM inversion remains inexpensive, with runtimes on the order of $10^{-3}$ seconds.

\subsubsection{FitzHugh--Nagumo model: the effect of TROM optimization} \label{sec:FHN1}

\begin{table}
\tabadjust
  \caption{FHN model inversion  with $\varepsilon_{\mathrm{TT}}=10^{-3}$, $\alpha_{\mathrm{ref}}=\alpha_{\mathrm{true}}=(0.7,0.8)^\mathsf{T}$, $\lambda=0.1$, and $\delta_{\mathrm{meas}}=10^{-2}$. The first block reports statistics over $100$ random initial guesses for $N=800$. The second block reports results obtained from the single run of the optimization solver with the TROM initial guess $\alpha_{\ast}$ for different values of $N$. We denote the estimated parameter vector by $\alpha_{\mathrm{est}}^\mathsf{T}$. }
\label{tab:fhn_initialization}
\begin{tabular}{@{}ccclllr@{}}
\toprule
$N$ & $\alpha_\ast^\mathsf{T}$ & $\alpha_{\mathrm{est}}^\mathsf{T}$ & \multicolumn{2}{c}{\bf residual} &  \bf error (std) & \bf iter\\
& & & \bf initial & \bf final  \\
\midrule
\multicolumn{7}{@{}l}{\textbf{random initialization}} \\
\midrule
  800 & --- & --- & \tnum{5.725e1} & \tnum{1.979e1} & \tnum{1.293e-1} (\tnum{2.283e-1}) & $24.79$  \\
\midrule
\multicolumn{7}{@{}l}{\textbf{TROM initialization}} \\
\midrule
  12 & $(0.7818,0.4727) $ & $(0.7818,0.4601)$ & \tnum{1.163e1}  & \tnum{1.082e1}  & \tnum{3.289e-1} (---) &  7 \\
  24 & $(0.6870,0.8522) $ & $(0.6870,0.8454)$ & \tnum{9.264e0}  & \tnum{7.457e0}  & \tnum{4.446e-2} (---) & 10 \\
  50 & $(0.7061,0.7755) $ & $(0.7061,0.7803)$ & \tnum{3.517e0}  & \tnum{2.246e0}  & \tnum{1.943e-2} (---) &  4 \\
 100 & $(0.6970,0.8121) $ & $(0.6969,0.8104)$ & \tnum{2.013e0}  & \tnum{1.061e0}  & \tnum{1.022e-2} (---) &  6 \\
 200 & $(0.6985,0.8060) $ & $(0.6985,0.8050)$ & \tnum{1.048e0}  & \tnum{5.092e-1} & \tnum{4.944e-3} (---) &  8 \\
 400 & $(0.6992,0.8030) $ & $(0.6993,0.8024)$ & \tnum{6.322e-1} & \tnum{4.319e-1} & \tnum{2.370e-3} (---) &  4 \\
 800 & $(0.6996,0.8015) $ & $(0.6998,0.8006)$ & \tnum{4.788e-1} & \tnum{4.179e-1} & \tnum{5.824e-4} (---) &  6 \\
\bottomrule
\end{tabular}
\end{table}

We now consider FitzHugh--Nagumo model and test the sensitivity of the inverse solver to initialization.

\begin{figure}
\centering
\includegraphics[width=0.5\textwidth]{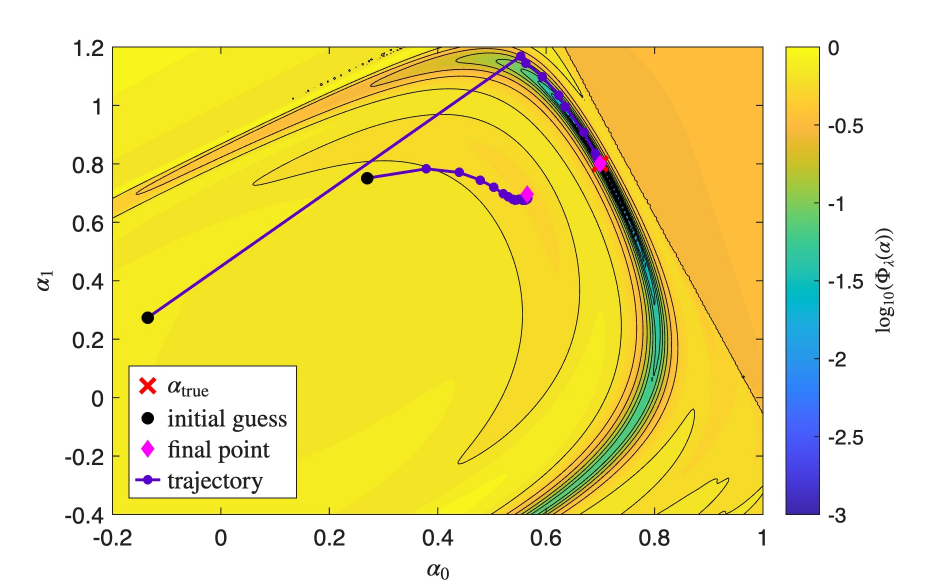}
  \caption{Two representative optimization trajectories of the standard GN optimization method for the random choice of initial guess. We overlay these trajectories on the optimization  landscape (in particular, on $\log_{10}\Phi_\lambda(\alpha)$).}
\label{fig:fhn_random_trajectories}
\end{figure}

For $N=800$, we sample $100$ random initial guesses uniformly from the parameter domain, rejecting values too close to $\alpha_{\mathrm{true}}$ as in the heat-transfer problem above, and solve the regularized inverse problem from each starting point. We compare these runs with the TROM-based initial guess $\alpha_\ast$ obtained for different values of $N$. The corresponding results are reported in \Cref{tab:fhn_initialization}.

As expected, the GN minimization is highly sensitive to the initial guess. About $40\%$ of the random initial guesses converge to local minima away from $\alpha_{\mathrm{true}}$, leading to a relatively large mean error and standard deviation. Representative optimization trajectories on the regularized objective landscape are shown in \Cref{fig:fhn_random_trajectories}. The trajectories illustrate that different initial guesses can lead the GN iteration to different local minima.

In contrast, the purely low-rank tensor-based optimization procedure described in \Cref{sec:tt_initial_guess} finds a minimizer of the objective functional on the parameter grid. This point $\alpha_\ast$ is close to the global minimizer, and the subsequent optional GN iteration initialized at $\alpha_\ast$ improves the minimizer only slightly. The final error in the recovered parameter is consistent with the modeling error of the TROM surrogate; see \Cref{t:fhn_rom_accuracy}.

\subsection{Randomized true parameters with a reference prior}
\label{exp_2}

Next, we study the case in which the true parameter $\alpha_{\mathrm{true}}$ differs from the reference $\alpha_{\mathrm{ref}}$ in the Tikhonov term. We fix $\alpha_{\mathrm{ref}}$ and sample several true parameters in its neighborhood. For each $\alpha_{\mathrm{true}}$, synthetic observations $y_{\mathrm{obs}} = f(\alpha_{\mathrm{true}}) + \varepsilon_{\mathrm{meas}}$ are generated by the FOM, while inversion is performed using a TROM surrogate. We use either
$\alpha^{(0)}=\alpha_{\mathrm{ref}}$
or
$\alpha^{(0)}=\alpha_\ast$,
where $\alpha_\ast$ is the TROM initial guess defined in~\eqref{initialGuess}.

The true parameters are sampled from a Gaussian distribution centered at $\alpha_{\mathrm{ref}}$. More precisely, we let
\begin{equation}\label{eq:AlphaDraw}
(\alpha_{\mathrm{true}})_j
\sim \mathcal{N}\left((\alpha_{\mathrm{ref}})_j,\sigma_j^2\right),
\qquad
\sigma_j=0.60(\alpha_j^{\mathrm{max}}-\alpha_j^{\mathrm{min}}),
\end{equation}

\noindent with samples outside $\mathcal A$ rejected and redrawn. We generate 10 realizations of $\alpha_{\mathrm{true}}$. For each realization $i$, the noiseless observation is
$
y_{\mathrm{true}}^{(i)} = f(\alpha_{\mathrm{true}}^{(i)}).
$
Since the TROM error
$
\varepsilon_{\mathrm{rom}}^{(i)}
=
\widehat f(\alpha_{\mathrm{true}}^{(i)})
-
f(\alpha_{\mathrm{true}}^{(i)})
$
and the total perturbation level
$
\delta_{\mathrm{tot}}^{(i)}
=
(\delta_{\mathrm{meas}}^2+(\delta_{\mathrm{rom}}^{(i)})^2)^{1/2}
$
vary across realizations, we use the prescribed discrepancy level
\[
\rho_\ast = \max_{1\le i\le 10}\rho^{(i)},
\qquad
\rho^{(i)}
=
\delta_{\mathrm{tot}}^{(i)}
\bigl\|y_{\mathrm{true}}^{(i)}\bigr\|_2.
\]

For each realization, we solve the inverse problem over a logarithmic grid of candidate values of $\lambda$ and choose the value for which the final residual best matches $\rho_\ast$. We then rerun the inversion with this selected regularization parameter and report the quantities from the final solve: the relative parameter error~\eqref{eq:rel_param_error}, selected $\lambda$, initial and final residuals, residual reduction
\[
\frac{\|f(\alpha_{\mathrm{ref}})-y_{\mathrm{obs}}\|_2}{\|f(\alpha_\lambda)-y_{\mathrm{obs}}\|_2},
\]

\noindent number of iterations, and runtime.

For the single-inclusion case, we use $D=3$, $h=1/64$, $N=8$, and
\[
\alpha_{\mathrm{ref}}=[0.80,\,0.50,\,0.22]^\mathsf{T}.
\]

\noindent For the two-inclusion case, we use $D=6$, $h=1/64$, $N=8$, and
\[
\alpha_{\mathrm{ref}} = [0.43,\,0.41,\,0.21,\,0.81,\,0.60,\,0.21]^\mathsf{T}.
\]

\noindent For $D=6$, we report only the noisy setting, which is the more practical case.

\begin{table}
\tabadjust
\caption{Summary of randomized-true-parameter inversion results for the single- and two-inclusion cases, all averaged over ten realizations of $\alpha_{\mathrm{true}}$. Results are shown for noise-free (\(\mathrm{noise}=0\)) and noisy (\(\mathrm{noise}=1\)) data using TT-SVD and TT-Cross surrogates. Reported quantities include the mean relative parameter error, selected regularization parameter, initial and final residuals, residual reduction, iteration count, and runtime.}
\label{tab:exp2_random_truth_summary_max_fixed}
\begin{tabular}{cccccccrrc}
\toprule
\bf noise & \bf surrogate & $\rho_\ast$ & \bf error & $\lambda$ & \multicolumn{3}{c}{\bf residual} & \bf iter & \bf time \\
\multicolumn{5}{c}{} & \bf initial & \bf final & \bf drop & & \\
\midrule
\rowcolor{gray!30}\multicolumn{10}{c}{$D=3$}\\
\midrule
  0 & TT-SVD   & \tnum{4.607e-3} & \tnum{9.453e-3} & \tnum{1.992e-2} & \tnum{1.556e-1} & \tnum{4.604e-3} & \fnum{33.81} & \fnum{4.20} & \fnum{0.044}\\
  1 &          & \tnum{6.894e-2} & \tnum{3.833e-2} & \tnum{1.136e-1} & \tnum{1.703e-1} & \tnum{6.894e-2} & \fnum{2.47}  & \fnum{4.50} & \fnum{0.042}\\
  0 & TT-Cross & \tnum{6.547e-3} & \tnum{1.476e-2} & \tnum{3.336e-2} & \tnum{1.553e-1} & \tnum{6.536e-3} & \fnum{23.77} & \fnum{4.60} & \fnum{0.045}\\
  1 &          & \tnum{6.910e-2} & \tnum{4.063e-2} & \tnum{1.348e-1} & \tnum{1.700e-1} & \tnum{6.910e-2} & \fnum{2.46}  & \fnum{4.20} & \fnum{0.043}\\
\midrule
\rowcolor{gray!30}\multicolumn{10}{c}{$D=6$}\\
\midrule
  1 & TT-SVD   & \tnum{6.788e-2} & \tnum{2.939e-2} & \tnum{1.662e-1} & \tnum{2.771e-1} & \tnum{6.787e-2} & \fnum{4.08} & \fnum{5.00} & \fnum{0.123}\\
  1 & TT-Cross & \tnum{7.044e-2} & \tnum{3.866e-2} & \tnum{2.570e-1} & \tnum{2.788e-1} & \tnum{7.040e-2} & \fnum{3.96} & \fnum{8.10} & \fnum{0.180}\\
\bottomrule
\end{tabular}
\end{table}

\Cref{tab:exp2_random_truth_summary_max_fixed} collects results for $D=3$ and $D=6$. For $D=3$ both surrogates achieve final residuals very close to $\rho_\ast$. In the noise-free setting, the residual reduction is large, especially for TT-SVD, while in the noisy setting the reduction is naturally smaller because $\rho_\ast$ is dominated by the measurement-noise level. The TT-SVD surrogate gives smaller parameter errors in the noise-free case, whereas the two surrogates produce comparable errors in the noisy case.

For $D=6$, the problem is more challenging. The selected regularization parameters are larger on average, and the residual reduction is smaller than in the single-inclusion noise-free experiments. Both surrogates still reach their prescribed discrepancy levels, but TT-Cross requires substantially more iterations and gives larger parameter errors than TT-SVD. This indicates that the increased dimension and stronger parameter coupling make the optimization more sensitive to the TROM approximation error.

\subsubsection{Three-inclusion case ($D=9$)}

For the three-inclusion case, we use $D=9$, $h=1/64$, $N=6$, and
\[
\alpha_{\mathrm{ref}} = [0.31,\,0.30,\,0.21,\,0.61,\,0.51,\,0.18,\,0.85,\,0.49,\,0.18]^\mathsf{T}.
\]

\noindent Because of the high parameter dimension, TT-SVD construction is not feasible in this setting, and we restrict attention to the TT-Cross surrogate. We consider the noisy case with fixed prescribed discrepancy level $\rho_\ast = \tnum{1.468e-1}$. The regularization parameter is selected from the same logarithmic grid using the discrepancy principle, as above. In \Cref{tab:exp2_TT-Cross_D9_noisy_summary} we show both average statistics and the results for each realization of $\alpha_{\mathrm{true}}$.

\begin{table}
\tabadjust
\caption{TT-Cross inversion results for the three-inclusion case ($D=9$) with relative noise level $\tnum{1.0e-2}$ and fixed discrepancy level $\rho_\ast=\tnum{1.468e-1}$. We report results for ten individual realizations and their average, including the relative parameter error, selected regularization parameter, initial and final residuals, residual reduction, iteration count, and runtime.}
\label{tab:exp2_TT-Cross_D9_noisy_summary}
\begin{tabular}{rcccccrc}
\toprule
\bf draw & \bf error & $\lambda$ & \multicolumn{3}{c}{\bf residual} & \bf iter & \bf time \\
\multicolumn{3}{c}{} & \bf initial & \bf final & \bf drop & \multicolumn{2}{c}{}\\
\midrule
1  & \tnum{7.800e-2} & \tnum{7.431e+0}  & \tnum{3.273e-1} & \tnum{1.464e-1} & \fnum{2.24} & 8  & \fnum{0.475}\\
2  & \tnum{8.546e-2} & \tnum{4.622e+0}  & \tnum{4.242e-1} & \tnum{1.468e-1} & \fnum{2.89} & 12 & \fnum{0.400}\\
3  & \tnum{1.087e-1} & \tnum{5.328e+0}  & \tnum{5.072e-1} & \tnum{1.468e-1} & \fnum{3.46} & 9  & \fnum{0.493}\\
4  & \tnum{5.484e-2} & \tnum{3.258e+0}  & \tnum{7.476e-1} & \tnum{1.468e-1} & \fnum{5.09} & 15 & \fnum{0.487}\\
5  & \tnum{7.980e-2} & \tnum{1.008e+1}  & \tnum{4.795e-1} & \tnum{1.468e-1} & \fnum{3.27} & 7  & \fnum{0.440}\\
6  & \tnum{7.758e-2} & \tnum{1.512e+1}  & \tnum{1.813e-1} & \tnum{1.468e-1} & \fnum{1.23} & 5  & \fnum{0.763}\\
7  & \tnum{7.776e-2} & \tnum{2.754e+1}  & \tnum{4.326e-1} & \tnum{1.468e-1} & \fnum{2.95} & 6  & \fnum{0.432}\\
8  & \tnum{6.831e-2} & \tnum{7.388e+1}  & \tnum{1.822e-1} & \tnum{1.468e-1} & \fnum{1.24} & 4  & \fnum{0.413}\\
9  & \tnum{7.886e-2} & \tnum{3.888e+0}  & \tnum{6.433e-1} & \tnum{1.468e-1} & \fnum{4.38} & 12 & \fnum{0.490}\\
10 & \tnum{6.329e-2} & \tnum{1.000e+4}  & \tnum{1.406e-1} & \tnum{1.405e-1} & \fnum{1.00} & 3  & \fnum{0.413}\\
\midrule
\rowcolor{gray!30}
mean & \tnum{7.726e-2} & \tnum{1.015e+3} & \tnum{4.066e-1} & \tnum{1.461e-1} & \fnum{2.77} & \fnum{8.10} & \fnum{0.481}\\
\bottomrule
\end{tabular}
\end{table}

The $D=9$ case confirms the increasing difficulty of the inverse problem in higher dimension. The selected values of $\lambda$ are larger and more variable, although the iteration counts are comparable  with $D=6$. The final residuals remain close to $\rho_\ast$ on average with more variability across realizations, see \Cref{tab:exp2_TT-Cross_D9_noisy_summary}.

\begin{table}
\tabadjust
\caption{Effect of TROM initialization in the noisy two- and three-inclusion tests. We compare initialization from $\alpha_{\mathrm{ref}}$ with initialization from the TROM-based guess $\alpha_\ast$. For each choice, we report averages over ten realizations of the relative parameter error, final residual, and iteration count. }
\label{tab:ttgrid_init_noisy_mean}
\begin{tabular}{c ccc | ccc}
\toprule
& \multicolumn{3}{c|}{$\alpha^{(0)}=\alpha_{\mathrm{ref}}$}
& \multicolumn{3}{c}{$\alpha^{(0)}=\alpha_\ast$} \\
\cmidrule(lr){2-4}\cmidrule(lr){5-7}
$D$ &
\bf error & \bf residual & \bf iter & \bf error & \bf residual & \bf iter \\
\midrule
  $6$ & \tnum{3.866e-02} & \tnum{7.040e-02} & 8.10 & \tnum{3.867e-02} & \tnum{7.041e-02} & \fnum{5.60} \\
  $9$ & \tnum{7.726e-02} & \tnum{1.461e-01} & 8.10 & \tnum{7.865e-02}& \tnum{1.473e-01}& \fnum{7.00} \\
\bottomrule
\end{tabular}
\end{table}

The total online runtime increases with the parameter dimension primarily because higher-dimensional cases require more optimization iterations. In contrast, the runtime per iteration grows only moderately with $D$, since the dominant online computations use the TROM surrogate rather than the full-order model.

We also repeated the noisy $D=6$ and $D=9$ TT-Cross tests using the TROM-provided initial guess from \Cref{sec:tt_initial_guess}. All other settings were left unchanged. For $D=6$, the mean iteration count drops from $8.10$ to $5.60$, while the mean parameter error and final residual remain almost the same. For $D=9$, the improvement is more modest, with the mean iteration count decreasing from $8.10$ to $7.00$; see \Cref{tab:ttgrid_init_noisy_mean}.

\subsubsection{FitzHugh--Nagumo model: the effect of TROM optimization}

\begin{table}
\tabadjust
\caption{FHN model inversion with $\varepsilon_{\mathrm{TT}}=10^{-3}$, $\alpha_{\mathrm{ref}}=(0.4,0.4)^\mathsf{T}$, and $\delta_{\mathrm{meas}}=10^{-2}$. The first block reports statistics obtained when $\alpha_{\mathrm{ref}}$ is used as the initial guess in the GN minimization procedure. The second block reports statistics obtained using the TROM optimization step for surrogates with different values of $N$. The residuals, errors, and iteration counts are averaged over $10$ runs with randomly sampled $\alpha_{\mathrm{true}}$.}
\label{tab:fhn_exp3_initialization}
\begin{tabular}{@{}rlllrrr@{}}
\toprule
  $N$ & $\rho_\ast$ & \multicolumn{2}{c}{\bf residual} & \bf error & \bf error std & \bf iter \\
 &  & \bf initial & \bf final \\
\midrule
\multicolumn{7}{@{}l}{\textbf{reference initialization}} \\
\midrule
 $800$ & $\tnum{5.204e-1}$ & $\tnum{4.897e1}$ & $\tnum{2.000e1}$ & $\tnum{2.057e-1}$ & $\tnum{2.624e-1}$ & $24.20$ \\
\midrule
\multicolumn{7}{@{}l}{\textbf{TROM initialization}} \\
\midrule
$12$  & $\tnum{4.452e1}$  & $\tnum{1.164e1}$  & $\tnum{1.028e1}$  & $\tnum{3.494e-1}$ & $\tnum{1.972e-1}$ & $7.70$ \\
$24$  & $\tnum{3.048e1}$  & $\tnum{3.235e0}$  & $\tnum{2.327e0}$  & $\tnum{2.771e-1}$ & $\tnum{4.024e-1}$ & $5.10$ \\
$50$  & $\tnum{9.273e0}$  & $\tnum{2.303e0}$  & $\tnum{1.573e0}$  & $\tnum{2.925e-1}$ & $\tnum{5.152e-1}$ & $5.70$ \\
$100$ & $\tnum{3.783e0}$  & $\tnum{1.816e0}$  & $\tnum{1.438e0}$  & $\tnum{2.087e-1}$ & $\tnum{3.036e-1}$ & $7.00$ \\
$200$ & $\tnum{2.317e0}$  & $\tnum{1.254e0}$  & $\tnum{1.296e0}$  & $\tnum{1.897e-1}$ & $\tnum{2.752e-1}$ & $4.40$ \\
$400$ & $\tnum{8.475e-1}$ & $\tnum{7.985e-1}$ & $\tnum{7.135e-1}$ & $\tnum{8.616e-2}$ & $\tnum{1.438e-1}$ & $6.30$ \\
$800$ & $\tnum{5.204e-1}$ & $\tnum{5.753e-1}$ & $\tnum{5.066e-1}$ & $\tnum{3.257e-2}$ & $\tnum{4.493e-2}$ & $3.80$ \\
\bottomrule
\end{tabular}
\end{table}

Here, we consider FitzHugh--Nagumo inversion in the case $\alpha_{\mathrm{true}}\neq \alpha_{\mathrm{ref}}$. The reference parameter is fixed at $\alpha_{\mathrm{ref}}=(0.4,0.4)^\mathsf{T}$, and ten $\alpha_{\mathrm{true}}$ parameters are sampled in its neighborhood according to~\eqref{eq:AlphaDraw}. The prescribed discrepancy level $\rho_\ast$ is fixed as the maximum target residual over the $10$ random realizations.

\Cref{tab:fhn_exp3_initialization} shows that initializing the GN iteration at $\alpha_{\mathrm{ref}}$ leads, on average, to poor parameter recovery. In fact, in $4$ out of $10$ runs with randomly sampled $\alpha_{\mathrm{true}}$, the solver converges to a false local minimum. The accuracy of the TROM-based inversion improves as $N$ increases. For the same surrogate complexity, $N=800$, the tensor-based optimization procedure provides significantly more accurate parameter estimates with much smaller variation. The subsequent optional GN iteration initialized at $\alpha_\ast$ further improves the minimizer slightly.

\subsection{TROM vs FOM inversion}

We next compare FOM inversion with TT-SVD and TT-Cross ROM inversions in the one-inclusion setting. The goal is to assess how TROM approximation error affects the inverse solution and how much online speedup is obtained relative to the full-order inversion. We consider two noise-free experiments. In the first experiment, the reference parameter coincides with the truth,
\[
\alpha_{\mathrm{true}}=\alpha_{\mathrm{ref}}=[0.80,\,0.50,\,0.22]^\mathsf{T};
\qquad
\lambda=1,
\]

\noindent the same ten random initial guesses are used for all inversion models. In the second, $\alpha_{\mathrm{ref}}$ is fixed but $\alpha_{\mathrm{true}}$ is sampled from the Gaussian distribution as described above; all runs are initialized with $\alpha^{(0)}=\alpha_{\mathrm{ref}}$. In both experiments, observations are generated by the FOM, $y_{\mathrm{obs}} = f(\alpha_{\mathrm{true}})$.

The results are reported in \Cref{tab:rom_fom_comparison}. For the FOM inversion, we compare two ways of forming the GN Jacobian. The first, denoted FOM-FD-GN, uses FDs to assemble the Jacobian. The second, denoted FOM-Sens-GN, assembles the Jacobian from forward sensitivity equations~\cite{hinze2009optimization} and therefore requires fewer additional FOM evaluations per iteration. This avoids FD perturbations and reuses the factorizations from the forward solve. We also include the average number of ``evaluations'', which denotes the number of FOM or TROM evaluations used in the final inversion solve with the preselected $\lambda$.

\begin{table}
\caption{Comparison of FOM and TROM inversions for the noise-free single-inclusion case. Results are averaged over ten runs for two choices of reference parameter, $\alpha_{\mathrm{ref}}=\alpha_{\mathrm{true}}$ and $\alpha_{\mathrm{ref}}\ne\alpha_{\mathrm{true}}$. The FOM is solved using either FD GN derivatives (FOM-FD-GN) or forward-sensitivity derivatives (FOM-Sens-GN), while the ROM inversions use TT-SVD and TT-Cross surrogates. Reported quantities include the relative parameter error, final residual, iteration count, number of function evaluations, time-to-solution, and speedup.}
\label{tab:rom_fom_comparison}
\tabadjust
\begin{tabular}{llcccrcc}
\toprule
$\alpha_{\mathrm{ref}}$ & \bf model & \bf error & \bf residual & \bf iter & \bf evals & \bf time & \bf speedup \\
\midrule
$\alpha_{\mathrm{true}}$
& FOM-FD-GN    & \(1.81\mathrm{e}{-12}\) & \(4.12\mathrm{e}{-13}\) & 4.00  & 29.00 & 5.54794 & -- \\
& FOM-Sens-GN  & \(1.80\mathrm{e}{-12}\) & \(4.11\mathrm{e}{-13}\) & 4.00  & 9.00  & 2.85906 & \(1.94{\times}\) \\
& TT-SVD       & \(5.93\mathrm{e}{-04}\) & \(5.56\mathrm{e}{-04}\) & 4.00  & 5.00  & 0.00299 & \(1856{\times}\) \\
& TT-Cross     & \(7.69\mathrm{e}{-04}\) & \(8.83\mathrm{e}{-04}\) & 4.00  & 5.00  & 0.00087 & \(6377{\times}\) \\
\midrule
$\neq\alpha_{\mathrm{true}}$
& FOM-FD-GN    & \(9.86\mathrm{e}{-02}\) & \(8.24\mathrm{e}{-03}\) & 4.40  & 31.80 & 6.09507 & -- \\
& FOM-Sens-GN  & \(9.86\mathrm{e}{-02}\) & \(8.24\mathrm{e}{-03}\) & 4.40  & 9.80  & 3.17046 & \(1.92{\times}\) \\
& TT-SVD       & \(9.85\mathrm{e}{-02}\) & \(8.25\mathrm{e}{-03}\) & 4.40  & 5.40  & 0.00415 & \(1469{\times}\) \\
& TT-Cross     & \(9.86\mathrm{e}{-02}\) & \(8.24\mathrm{e}{-03}\) & 4.40  & 5.40  & 0.00144 & \(4233{\times}\) \\
\bottomrule
\end{tabular}
\end{table}

When $\alpha_{\mathrm{ref}}=\alpha_{\mathrm{true}}$, the FOM inversion recovers the parameter to nearly machine precision. This is expected because the data are noise-free and the true parameter is also used as the reference parameter in the Tikhonov term. Both TROM inversions recover parameters close to the truth at dramatically lower online cost, with errors limited by the surrogate approximation.

The forward-sensitivity FOM implementation gives essentially the same reconstruction as the FD FOM implementation, while reducing the FOM runtime by about a factor of two. Nevertheless, the TROM inversions remain orders of magnitude faster than both FOM variants.

When $\alpha_{\mathrm{ref}}\ne\alpha_{\mathrm{true}}$, the reconstruction error is dominated by the regularization bias toward $\alpha_{\mathrm{ref}}$. In this setting, both TROM solvers closely reproduce the FOM inversion behavior while reducing the runtime by several orders of magnitude.

\subsection{Speedup from deeper TROM integration}

\begin{table}
\tabadjust
\caption{Runtime comparison of a plain and integrated TROM inversion. Here ``Full/FD'' uses the observation-space objective with a FD Jacobian, whereas ``Reduced/TROM'' uses~\eqref{eq:inverse_problem_TROM} and the reduced TROM GN quantities. The TT-Cross surrogate is used in all TROM runs, and the FOM inversion time is included for reference. Reported quantities include $D$, $N$, the ranks $r_2$ and $r_{\mathrm{max}}$, the mean runtime over 50 executions, and the speedup; $N_{\mathrm{obs}}=279$.}
\label{tab:full_vs_reduced_runtime}
\begin{tabular}{lccccccc}
\toprule
\bf ROM & \bf objective/jacobian & $D$ & $N$ & $r_2$ & $r_{\mathrm{max}}$ & \bf time & \bf speedup \\
\midrule
  & Full\,/\,FD      & 3 &  8 & 13 & 13 & $\tnum{4.62e-4}$ &  \\
  & Reduced\,/\,TROM & 3 &  8 & 13 & 13 & $\tnum{1.03e-4}$ & $4.49\times$ \\
\cmidrule(lr){2-8}
  & Full\,/\,FD      & 6 &  8 & 20 & 22 & $\tnum{1.76e-3}$ &  \\
  & Reduced\,/\,TROM & 6 &  8 & 20 & 22 & $\tnum{2.87e-4}$ & $6.13\times$ \\
\cmidrule(lr){2-8}
  & Full\,/\,FD      & 9 &  6 & 14 & 14 & $\tnum{7.31e-3}$ &  \\
  & Reduced\,/\,TROM & 9 &  6 & 14 & 14 & $\tnum{1.78e-3}$ & $4.11\times$ \\
\midrule
  FOM & Full\,/\,FD  & 3 & -- & -- & -- & $\tnum{7.7436}$ & -- \\
\bottomrule
\end{tabular}
\end{table}

\begin{table}
\tabadjust
\caption{Runtime comparison of a plain and integrated TROM inversion for different TT truncation tolerances. The TROM surrogate is constructed using TT-SVD. We report the ranks $r_2$ and $r_{\mathrm{max}}$, the mean runtime over 50 executions, and the speedup. Here, $D=6$, $N=8$, and $N_{\mathrm{obs}}=279$.}
\label{tab:ttsvd_runtime_vs_eps}
\begin{tabular}{cccccc}
\toprule
$\varepsilon_{\mathrm{TT}}$ & \bf objective / jacobian & $r_2$ & $r_{\mathrm{max}}$ & \bf time & \bf speedup \\
\midrule
1e--2 & Full / FD      & 6  & 6   & $\tnum{1.50e-3}$ & -- \\
      & Reduced / TROM & 6  & 6   & $\tnum{1.89e-4}$ & $7.94\times$ \\
\midrule
1e--3 & Full / FD      & 14 & 23  & $\tnum{1.52e-3}$ & -- \\
      & Reduced / TROM & 14 & 23  & $\tnum{2.74e-4}$ & $5.55\times$ \\
\midrule
1e--4 & Full / FD      & 25 & 78  & $\tnum{3.42e-3}$ & -- \\
      & Reduced / TROM & 25 & 78  & $\tnum{6.09e-4}$ & $5.62\times$ \\
\midrule
1e--5 & Full / FD      & 38 & 205 & $\tnum{1.03e-2}$ & -- \\
      & Reduced / TROM & 38 & 205 & $\tnum{1.91e-3}$ & $5.39\times$ \\
\bottomrule
\end{tabular}
\end{table}

In the previous section, we observed that using the TROM surrogate for forward solves gives a dramatic speedup compared with using the FOM, as expected. Here, we study to what extent this efficiency is due to the deeper integration of TROM into the inversion process, beyond its use as an approximate forward map.

\Cref{tab:full_vs_reduced_runtime,tab:ttsvd_runtime_vs_eps} compare plain TROM inversion, where the surrogate is used only for forward evaluations and the Jacobian is computed by finite differences, with the integrated TROM formulation based on~\eqref{eq:inverse_problem_TROM} and the reduced GN quantities. Each solve is repeated 50 times to reduce timing noise.

The integrated formulation consistently reduces the online runtime by a factor of about $4$ to $8$ across the tested cases. The speedup is due mainly to the TROM-based GN quantities, which avoid FD Jacobian evaluations and replace observation-space linear algebra by reduced-coordinate computations.

\subsection{Effect of measurement noise on TROM inversion}
\label{rom_inversion_exp1_noise}

Next, we study the effect of measurement noise on the TROM inversion in the one-inclusion case with
$
\alpha_{\mathrm{ref}}=[0.80,\,0.50,\,0.22]^\mathsf{T}.
$
The true parameter is sampled around $\alpha_{\mathrm{ref}}$, all runs are initialized with $\alpha^{(0)}=\alpha_{\mathrm{ref}}$, and the noiseless observations are generated by the FOM,
$
y_{\mathrm{true}}=f(\alpha_{\mathrm{true}}).
$
Noisy observations are defined by \eqref{eq:noisyObs} (see \Cref{sec:discrepancy}) with various $\delta_{\mathrm{meas}}$. The regularization parameter is chosen using the  discrepancy principle with $\delta_{\mathrm{tot}}=\sqrt{\delta_{\mathrm{meas}}^2+\delta_{\mathrm{rom}}^2},$ $\tau=1$. Here $\delta_{\mathrm{rom}}$ is the fixed relative TROM error estimate obtained from the preceding noise-free random-truth experiment, which was found to be $\delta_{\mathrm{rom}}=\tnum{5.016e-04}$ for TT-SVD, $\delta_{\mathrm{rom}}=\tnum{6.228e-04}$ for TT-Cross.

\begin{table}
\caption{TROM inversion with increasing noise level. All quantities are averaged over 10 runs. Reported quantities include the mean relative parameter error, final residual, iteration count, runtime, and selected regularization parameter $\lambda$.}
\label{tab:rom_noise_practical_merged}
\tabadjust
\begin{tabular}{ccccccl}\toprule
\bf surrogate & $\delta_{\mathrm{meas}}$ & \bf error & \bf residual & \bf iter & \bf time & $\lambda$ \\
\midrule
TT-SVD
& $0.00$ & \tnum{5.5505e-03} & \tnum{5.3579e-04} & \fnum{4.30} & $0.017$ & \tnum{4.8514e-03} \\
& $0.05$ & \tnum{7.0042e-02} & \tnum{5.0001e-02} & \fnum{4.90} & $0.019$ & \tnum{1.5682e+00} \\
& $0.10$ & \tnum{8.2654e-02} & \tnum{1.0000e-01} & \fnum{5.50} & $0.025$ & \tnum{1.3196e+00} \\
& $0.20$ & \tnum{8.9520e-02} & \tnum{1.9999e-01} & \fnum{5.50} & $0.033$ & \tnum{1.0122e+03} \\
& $0.25$ & \tnum{1.1006e-01} & \tnum{2.4991e-01} & \fnum{5.70} & $0.037$ & \tnum{2.0018e+03} \\
& $0.50$ & \tnum{1.1938e-01} & \tnum{4.9971e-01} & \fnum{6.90} & $0.037$ & \tnum{2.0049e+03} \\
\midrule
TT-Cross
& $0.00$ & \tnum{7.0826e-03} & \tnum{6.7040e-04} & \fnum{4.40} & $0.017$ & \tnum{5.6062e-03} \\
& $0.05$ & \tnum{7.3084e-02} & \tnum{5.0002e-02} & \fnum{4.80} & $0.020$ & \tnum{1.6884e+00} \\
& $0.10$ & \tnum{8.2515e-02} & \tnum{9.9999e-02} & \fnum{5.70} & $0.029$ & \tnum{1.2514e+00} \\
& $0.20$ & \tnum{8.7505e-02} & \tnum{1.9997e-01} & \fnum{5.40} & $0.034$ & \tnum{1.0113e+03} \\
& $0.25$ & \tnum{1.0954e-01} & \tnum{2.4992e-01} & \fnum{6.40} & $0.041$ & \tnum{1.0401e+03} \\
& $0.50$ & \tnum{1.1813e-01} & \tnum{4.9972e-01} & \fnum{7.90} & $0.041$ & \tnum{2.0050e+03} \\
\bottomrule
\end{tabular}
\end{table}

\Cref{tab:rom_noise_practical_merged} shows consistent behavior for both TT-SVD and TT-Cross as the noise level increases. The final residual closely matches the prescribed discrepancy level in all tested regimes. For small and moderate noise, the parameter recovery remains stable. For larger noise, the selected regularization parameter increases, which stabilizes the solve but also leads to larger reconstruction errors. The online runtimes remain small and vary only mildly with the noise level.

\subsection{Scalability of TROM inversion}

Finally, we assess how the online TROM inversion scales with the parameter dimension. We consider one-, two-, and three-inclusion cases and use the TT-Cross ROM in all cases with $N=8$. The true parameter is sampled around $\alpha_{\mathrm{ref}}$, all runs are initialized with $\alpha^{(0)}=\alpha_{\mathrm{ref}}$, and the measurement-noise level is fixed at $\delta_{\mathrm{meas}}=0.10$. The regularization parameter is chosen by the same practical discrepancy principle
with
$
\delta_{\mathrm{rom}}=
\tnum{6.228e-04},\,
\tnum{2.457e-03},\,
\tnum{7.951e-03}
$
for $D=3,6,9$, respectively.

\begin{table}
\caption{TT-Cross scalability results for noisy TROM inversion with fixed measurement-noise level $\delta_{\mathrm{meas}}=0.10$. All quantities are averaged over 10 runs. Reported quantities include the mean relative parameter error, final residual, iteration count, online runtime, selected regularization parameter $\lambda$, and offline construction time.}
\tabadjust
\label{tab:scalability_TT-Cross_merged}
\begin{tabular}{cccrrccr}
\toprule
$D$ & \bf error & \bf residual & \bf iter & \bf evals & \bf online time & $\lambda$ & \bf offline time \\
\midrule
$3$ & \tnum{9.5876e-02} & \tnum{9.9975e-02} & \fnum{4.70}  & \fnum{5.70}  & 0.026 & \tnum{1.0011e+03} &    121 \\
$6$ & \tnum{6.2962e-02} & \tnum{1.0003e-01} & \fnum{8.70}  & \fnum{10.80} & 0.198 & \tnum{1.9778e+00} &  6,413 \\
$9$ & \tnum{7.5559e-02} & \tnum{1.0039e-01} & \fnum{10.90} & \fnum{29.10} & 0.297 & \tnum{4.8378e+00} & 13,379 \\
\bottomrule
\end{tabular}
\end{table}

\Cref{tab:scalability_TT-Cross_merged} shows that TT-Cross inversion remains stable as the parameter dimension increases from $D=3$ to $D=9$. The final residual stays close to the prescribed discrepancy level in all cases. The main effect of increasing $D$ is computational: the average iteration count, number of surrogate evaluations, and online runtime all increase. Nevertheless, the online cost remains small compared with the offline TT-Cross construction, which is the dominant computational expense.

\section{Conclusions}
\label{s:conclusions}

We demonstrated that TROM surrogates provide fast, reliable approximations of the parameter-to-observation map, with an error that can be estimated and incorporated into the inversion procedure. In the Tikhonov framework, both TT-SVD and TT-Cross surrogates reproduce full-order inversion at a substantially lower online cost. The experiments also show that the TROM error should be treated as part of the effective data perturbation when selecting the regularization parameter, which leads to stable reconstructions for both noise-free and noisy observations.

The accompanying error-to-inversion analysis makes this control explicit: under local strong convexity of the regularized FOM objective and suitable parametric smoothness, it bounds the perturbation of the recovered parameter in terms of TROM map and Jacobian errors and estimates these errors from tensor compression and interpolation.

A central conclusion is that the low-rank multilinear structure is useful beyond fast forward evaluation. After orthogonalizing the first TT core, the TROM least-squares term can be minimized equivalently in reduced coordinates; the Gauss--Newton residuals and Jacobian products can be assembled from TT contractions without forming the full observation-space Jacobian; and the tensor representation can solve discrete objective minimization problems in TT format. Depending on the problem, this tensor optimization step serves either as a stand-alone approximate minimizer or as a data-informed initialization for a subsequent Gauss--Newton solve, and the numerical results illustrate both roles. In the heat-transfer problem, the method remains effective as the parameter dimension increases, and TT-Cross extends the approach to regimes where full tensor construction by TT-SVD is impractical. In the FitzHugh--Nagumo example, the objective landscape is strongly nonconvex and Gauss--Newton is sensitive to initialization; there, the TROM optimization step substantially improves robustness by finding a better basin of attraction.

The main tradeoff is the offline cost of constructing the TROM surrogate, especially in higher parameter dimension; once it is available, the online inverse solves are inexpensive.

\section*{Disclosure of AI-assisted tools}
During the preparation of this manuscript, the authors used the generative AI tools ChatGPT (OpenAI) and Claude (Anthropic) to assist with language and style editing, correction of typographical and grammatical errors, and the drafting of scripts. All mathematical developments, derivations, modeling, and analyses were carried out by the authors. The authors are responsible for the correctness of all code used in this work; they reviewed and tested this code and verified its functionality. The authors reviewed and edited all AI-assisted content and take full responsibility for the content of this publication.

\bibliographystyle{plainnat}
\bibliography{references}

@article{oseledets2011tensor,
  author      = {Oseledets, Ivan V.},
  title       = {Tensor-Train Decomposition},
  journal     = {SIAM Journal on Scientific Computing},
  volume      = {33},
  number      = {5},
  pages       = {2295--2317},
  year        = {2011},
  OPTdoi      = {10.1137/090752286},
}

@book{hansen2010discrete,
  author      = {Hansen, Per Christian},
  title       = {Discrete Inverse Problems: Insight and Algorithms},
  publisher   = {SIAM},
  address     = {Philadelphia},
  year        = {2010},
  OPTdoi      = {10.1137/1.9780898718836},
}

@book{kaipio2005statistical,
  author      = {Kaipio, Jari and Somersalo, Erkki},
  title       = {Statistical and Computational Inverse Problems},
  series      = {Applied Mathematical Sciences},
  volume      = {160},
  publisher   = {Springer},
  address     = {New York},
  year        = {2005},
  OPTdoi      = {10.1007/b138659},
}

@article{krivorotko2020global,
  author  = {Krivorotko, Olga and Kabanikhin, Sergey and Zhang, Shuhua and Kashtanova, Victoriya},
  title   = {Global and Local Optimization in Identification of Parabolic Systems},
  journal = {Journal of Inverse and Ill-Posed Problems},
  volume  = {28},
  number  = {6},
  pages   = {899--913},
  year    = {2020},
  OPTdoi  = {10.1515/jiip-2020-0083},
}

@book{nocedal2006numerical,
  author      = {Nocedal, Jorge and Wright, Stephen J.},
  title       = {Numerical Optimization},
  edition     = {2},
  publisher   = {Springer},
  address     = {New York, NY},
  year        = {2006},
  OPTdoi      = {10.1007/978-0-387-40065-5},
}

@article{mang2018pdeconstrained,
  author      = {Mang, Andreas and Gholami, Amir and Davatzikos, Christos and Biros, George},
  title       = {{PDE}-constrained optimization in medical image analysis},
  journal     = {Optimization and Engineering},
  volume      = {19},
  number      = {3},
  pages       = {765--812},
  year        = {2018},
  OPTdoi      = {10.1007/s11081-018-9390-9},
}

@article{islam2026tensorial,
  author      = {Islam, Asikul and Mizan, Md Rezwan Bin and Olshanskii, Maxim and Mang, Andreas},
  title       = {Tensorial Reduced-Order Models for Parametric Coupled Reaction-Diffusion Systems: Application to Brain Tumor Growth Modeling},
  journal     = {arXiv preprint arXiv:2603.14101},
  year        = {2026},
  doi         = {10.48550/arXiv.2603.14101},
}

@article{lieberman2010parameter,
  author      = {Lieberman, Chad and Willcox, Karen and Ghattas, Omar},
  title       = {Parameter and State Model Reduction for Large-Scale Statistical Inverse Problems},
  journal     = {SIAM Journal on Scientific Computing},
  volume      = {32},
  number      = {5},
  pages       = {2523--2542},
  year        = {2010},
  OPTdoi      = {10.1137/090775622},
}

@article{burman2026solving,
  author  = {Burman, Erik and Larson, Mats G. and Larsson, Karl and
             Vallin, Jonatan},
  title   = {Solving Inverse Parametrized Problems via Finite Elements
             and Extreme Learning Networks},
  journal = {Computer Methods in Applied Mechanics and Engineering},
  volume  = {460},
  pages   = {119077},
  year    = {2026},
  OPTdoi  = {10.1016/j.cma.2026.119077}
}

@article{ghattas2021learning,
  author      = {Ghattas, Omar and Willcox, Karen},
  title       = {Learning physics-based models from data: perspectives from inverse problems and model reduction},
  journal     = {Acta Numerica},
  volume      = {30},
  pages       = {445--554},
  year        = {2021},
  publisher   = {Cambridge University Press},
  OPTdoi      = {10.1017/S0962492921000064},
}

@article{dolgov2015polynomial,
  title={Polynomial chaos expansion of random coefficients and the solution of stochastic partial differential equations in the tensor train format},
  author={Dolgov, Sergey and Khoromskij, Boris N and Litvinenko, Alexander and Matthies, Hermann G},
  journal={SIAM/ASA Journal on Uncertainty Quantification},
  volume={3},
  number={1},
  pages={1109--1135},
  year={2015},
  publisher={SIAM}
}

@article{villalobos2026neural,
  author      = {Villalobos, German and Rudi, Johann and Mang, Andreas},
  title       = {Neural Networks for {Bayesian} Inverse Problems Governed by a Nonlinear {ODE}},
  journal     = {SIAM Journal on Scientific Computing},
  volume      = {48},
  number      = {3},
  pages       = {C415--C452},
  year        = {2026},
  publisher   = {SIAM},
  OPTdoi      = {10.1137/24M1711480},
}

@article{nagumo1962active,
  author      = {Nagumo, Jinichi and Arimoto, Suguru and Yoshizawa, Shuji},
  title       = {An active pulse transmission line simulating nerve axon},
  journal     = {Proceedings of the IRE},
  volume      = {50},
  number      = {10},
  pages       = {2061--2070},
  year        = {1962},
  publisher   = {IEEE},
  OPTdoi      = {10.1109/JRPROC.1962.288235},
}

@book{hinze2009optimization,
  author      = {Hinze, Michael and Pinnau, Rene and Ulbrich, Michael and Ulbrich, Stefan},
  title       = {Optimization with {PDE} Constraints},
  series      = {Mathematical Modelling: Theory and Applications},
  volume      = {23},
  publisher   = {Springer},
  address     = {Dordrecht},
  year        = {2009},
  OPTdoi      = {10.1007/978-1-4020-8839-1},
}

@article{fitzhugh1961impulses,
  author      = {FitzHugh, Richard},
  title       = {Impulses and physiological states in theoretical models of nerve membrane},
  journal     = {Biophysical Journal},
  volume      = {1},
  number      = {6},
  pages       = {445--466},
  year        = {1961},
  publisher   = {Elsevier},
  OPTdoi      = {10.1016/S0006-3495(61)86902-6},
}

@inproceedings{rudi2022parameter,
  author      = {Rudi, Johann and Bessac, Julie and Lenzi, Amanda},
  title       = {Parameter estimation with dense and convolutional neural networks applied to the {FitzHugh--Nagumo} {ODE}},
  booktitle   = {Mathematical and Scientific Machine Learning},
  series      = {Proceedings of Machine Learning Research},
  volume      = {145},
  pages       = {781--808},
  year        = {2022},
  organization = {PMLR},
}

@article{lindemulder2019maximal,
  author      = {Lindemulder, Nick},
  title       = {Maximal regularity with weights for parabolic problems with inhomogeneous boundary conditions},
  journal     = {Journal of Evolution Equations},
  volume      = {20},
  number      = {1},
  pages       = {59--108},
  year        = {2019},
  publisher   = {Springer},
  OPTdoi      = {10.1007/s00028-019-00515-7},
}

@article{oseledets2010tt,
  author      = {Oseledets, Ivan and Tyrtyshnikov, Eugene},
  title       = {{TT}-cross approximation for multidimensional arrays},
  journal     = {Linear Algebra and its Applications},
  volume      = {432},
  number      = {1},
  pages       = {70--88},
  year        = {2010},
  publisher   = {Elsevier},
  OPTdoi      = {10.1016/j.laa.2009.07.024},
}

@article{meyries2012maximal,
  author      = {Meyries, Martin and Schnaubelt, Roland},
  title       = {Maximal regularity with temporal weights for parabolic problems with inhomogeneous boundary conditions},
  journal     = {Mathematische Nachrichten},
  volume      = {285},
  number      = {8-9},
  pages       = {1032--1051},
  year        = {2012},
  publisher   = {Wiley Online Library},
  OPTdoi      = {10.1002/mana.201100057},
}

@article{mamonov2022interpolatory,
  author      = {Mamonov, Alexander V. and Olshanskii, Maxim A.},
  title       = {Interpolatory tensorial reduced order models for parametric dynamical systems},
  journal     = {Computer Methods in Applied Mechanics and Engineering},
  volume      = {397},
  pages       = {115122},
  year        = {2022},
  publisher   = {Elsevier},
  OPTdoi      = {10.1016/j.cma.2022.115122},
}

@article{mamonov2024tensorial,
  author      = {Mamonov, Alexander V. and Olshanskii, Maxim A.},
  title       = {Tensorial parametric model order reduction of nonlinear dynamical systems},
  journal     = {SIAM Journal on Scientific Computing},
  volume      = {46},
  number      = {3},
  pages       = {A1850--A1878},
  year        = {2024},
  publisher   = {SIAM},
  OPTdoi      = {10.1137/23M1553789},
}

@article{mamonov2025apriori,
  author      = {Mamonov, Alexander V. and Olshanskii, Maxim A.},
  title       = {A priori analysis of a tensor {ROM} for parameter dependent parabolic problems},
  journal     = {SIAM Journal on Numerical Analysis},
  volume      = {63},
  number      = {1},
  pages       = {239--261},
  year        = {2025},
  publisher   = {SIAM},
  OPTdoi      = {10.1137/23M1616844},
}

@book{morozov2012methods,
  author      = {Morozov, Vladimir Alekseevich},
  title       = {Methods for solving incorrectly posed problems},
  publisher   = {Springer Science \& Business Media},
  year        = {2012},
  OPTdoi      = {10.1007/978-1-4612-5280-1},
}

@book{pruss2016moving,
  author      = {Pr{\"u}ss, Jan and Simonett, Gieri},
  title       = {Moving interfaces and quasilinear parabolic evolution equations},
  volume      = {105},
  publisher   = {Springer},
  year        = {2016},
  OPTdoi      = {10.1007/978-3-319-27698-4},
}

@article{dolgov2014alternating,
  author      = {Dolgov, Sergey V. and Savostyanov, Dmitry V.},
  title       = {Alternating minimal energy methods for linear systems in higher dimensions},
  journal     = {SIAM Journal on Scientific Computing},
  volume      = {36},
  number      = {5},
  pages       = {A2248--A2271},
  year        = {2014},
  publisher   = {SIAM},
  OPTdoi      = {10.1137/140953289},
}

@incollection{frangos2010surrogate,
  author      = {Frangos, Michalis and Marzouk, Youssef and Willcox, Karen and van Bloemen Waanders, Bart},
  title       = {Surrogate and Reduced-Order Modeling: A Comparison of Approaches for Large-Scale Statistical Inverse Problems},
  booktitle   = {Large-Scale Inverse Problems and Quantification of Uncertainty},
  publisher   = {John Wiley \& Sons},
  pages       = {123--149},
  year        = {2010},
  OPTdoi      = {10.1002/9780470685853.ch7},
}

@article{benner2015survey,
  author      = {Benner, Peter and Gugercin, Serkan and Willcox, Karen},
  title       = {A Survey of Projection-Based Model Reduction Methods for Parametric Dynamical Systems},
  journal     = {SIAM Review},
  volume      = {57},
  number      = {4},
  pages       = {483--531},
  year        = {2015},
  OPTdoi      = {10.1137/130932715},
}

@incollection{grieves2017digital,
  author      = {Grieves, Michael and Vickers, John},
  title       = {Digital Twin: Mitigating Unpredictable, Undesirable Emergent Behavior in Complex Systems},
  booktitle   = {Transdisciplinary Perspectives on Complex Systems},
  pages       = {85--113},
  year        = {2017},
  publisher   = {Springer},
  OPTdoi      = {10.1007/978-3-319-38756-7_4},
}

@article{tao2018digital,
  author      = {Tao, Fei and Zhang, Meng and Liu, Ying and Nee, Andrew Y. C.},
  title       = {Digital Twin Driven Prognostics and Health Management for Complex Equipment},
  journal     = {CIRP Annals},
  volume      = {67},
  number      = {1},
  pages       = {169--172},
  year        = {2018},
  OPTdoi      = {10.1016/j.cirp.2018.04.055},
}

@book{quarteroni2016reduced,
  author      = {Quarteroni, Alfio and Manzoni, Andrea and Negri, Federico},
  title       = {Reduced Basis Methods for Partial Differential Equations: An Introduction},
  series      = {UNITEXT},
  volume      = {92},
  publisher   = {Springer},
  address     = {Cham},
  year        = {2016},
  OPTdoi      = {10.1007/978-3-319-15431-2},
}

@book{hackbusch2012tensor,
  author      = {Hackbusch, Wolfgang},
  title       = {Tensor Spaces and Numerical Tensor Calculus},
  series      = {Springer Series in Computational Mathematics},
  volume      = {42},
  publisher   = {Springer},
  address     = {Berlin, Heidelberg},
  year        = {2012},
  OPTdoi      = {10.1007/978-3-642-28027-6},
}

@article{grasedyck2013literature,
  author      = {Grasedyck, Lars and Kressner, Daniel and Tobler, Christine},
  title       = {A Literature Survey of Low-Rank Tensor Approximation Techniques},
  journal     = {GAMM-Mitteilungen},
  volume      = {36},
  number      = {1},
  pages       = {53--78},
  year        = {2013},
  OPTdoi      = {10.1002/gamm.201310004},
}

@article{galbally2010nonlinear,
  author      = {Galbally, Diego and Fidkowski, Krzysztof and Willcox, Karen and Ghattas, Omar},
  title       = {Nonlinear Model Reduction for Uncertainty Quantification in Large-Scale Inverse Problems},
  journal     = {International Journal for Numerical Methods in Engineering},
  volume      = {81},
  number      = {12},
  pages       = {1581--1608},
  year        = {2010},
  OPTdoi      = {10.1002/nme.2746},
}

@article{cui2015datadriven,
  author      = {Cui, Tiangang and Marzouk, Youssef M. and Willcox, Karen E.},
  title       = {Data-Driven Model Reduction for the {Bayesian} Solution of Inverse Problems},
  journal     = {International Journal for Numerical Methods in Engineering},
  volume      = {102},
  number      = {5},
  pages       = {966--990},
  year        = {2015},
  OPTdoi      = {10.1002/nme.4748},
}

@article{chertkov2022optimization,
  author      = {Chertkov, Andrei and Ryzhakov, Gleb and Novikov, Georgii and Oseledets, Ivan},
  title       = {Optimization of functions given in the tensor train format},
  journal     = {arXiv preprint arXiv:2209.14808},
  year        = {2022},
  doi         = {10.48550/arXiv.2209.14808},
}

@article{dolgov2020approximation,
  author      = {Dolgov, Sergey and Anaya-Izquierdo, Karim and Fox, Colin and Scheichl, Robert},
  title       = {Approximation and Sampling of Multivariate Probability Distributions in the Tensor Train Decomposition},
  journal     = {Statistics and Computing},
  volume      = {30},
  pages       = {603--625},
  year        = {2020},
  OPTdoi      = {10.1007/s11222-019-09910-z},
}

@article{mizan2026parametric,
  author      = {Mizan, Md Rezwan Bin and Olshanskii, Maxim and Timofeyev, Ilya},
  title       = {A parametric tensor {ROM} for the shallow water dam break problem},
  journal     = {Computers \& Fluids},
  pages       = {107005},
  year        = {2026},
  publisher   = {Elsevier},
  OPTdoi      = {10.1016/j.compfluid.2026.107005},
}

@article{vijaywargiya2026structure,
  author      = {Vijaywargiya, Arjun and Cyr, Eric C. and Gruber, Anthony},
  title       = {Structure-Aware Tensorial Model Reduction},
  journal     = {arXiv preprint arXiv:2604.26280},
  year        = {2026},
  doi         = {10.48550/arXiv.2604.26280},
}

@article{eigel2022lowrank,
  author      = {Eigel, Martin and Gruhlke, Robert and Marschall, Manuel},
  title       = {Low-Rank Tensor Reconstruction of Concentrated Densities with Application to {Bayesian} Inversion},
  journal     = {Statistics and Computing},
  volume      = {32},
  number      = {27},
  year        = {2022},
  OPTdoi      = {10.1007/s11222-022-10087-1},
}

@article{mueller2026tensor,
  author      = {Mueller, Nicholas and Zhao, Yiran and Badia, Santiago and Cui, Tiangang},
  title       = {A Tensor-Train Reduced Basis Solver for Parameterized Partial Differential Equations on {Cartesian} Grids},
  journal     = {Journal of Computational and Applied Mathematics},
  volume      = {472},
  pages       = {116790},
  year        = {2026},
  OPTdoi      = {10.1016/j.cam.2025.116790},
}

@article{lee2020model,
  author      = {Lee, Kookjin and Carlberg, Kevin T.},
  title       = {Model reduction of dynamical systems on nonlinear manifolds using deep convolutional autoencoders},
  journal     = {Journal of Computational Physics},
  volume      = {404},
  pages       = {108973},
  year        = {2020},
  publisher   = {Elsevier},
  OPTdoi      = {10.1016/j.jcp.2019.108973},
}

@article{fresca2021comprehensive,
  author      = {Fresca, Stefania and Dede', Luca and Manzoni, Andrea},
  title       = {A comprehensive deep learning-based approach to reduced order modeling of nonlinear time-dependent parametrized {PDEs}},
  journal     = {Journal of Scientific Computing},
  volume      = {87},
  number      = {2},
  pages       = {61},
  year        = {2021},
  publisher   = {Springer},
  OPTdoi      = {10.1007/s10915-021-01462-7},
}

@article{ivagnes2023towards,
  author      = {Ivagnes, Anna and Demo, Nicola and Rozza, Gianluigi},
  title       = {Towards a machine learning pipeline in reduced order modelling for inverse problems: neural networks for boundary parametrization, dimensionality reduction and solution manifold approximation},
  journal     = {Journal of Scientific Computing},
  volume      = {95},
  number      = {1},
  pages       = {23},
  year        = {2023},
  publisher   = {Springer},
  OPTdoi      = {10.1007/s10915-023-02142-4},
}

@article{lieberman2013goal,
  author      = {Lieberman, Chad and Willcox, Karen},
  title       = {Goal-Oriented Inference: Approach, Linear Theory, and Application to Advection Diffusion},
  journal     = {SIAM Review},
  volume      = {55},
  number      = {3},
  pages       = {493--519},
  year        = {2013},
  OPTdoi      = {10.1137/130913110},
}

@inproceedings{sozykin2022ttopt,
  author      = {Sozykin, Konstantin and Chertkov, Andrei and Schutski, Roman and Phan, Anh-Huy and Cichocki, Andrzej and Oseledets, Ivan},
  title       = {{TTOpt}: A maximum volume quantized tensor train-based optimization and its application to reinforcement learning},
  booktitle   = {Advances in Neural Information Processing Systems},
  volume      = {35},
  pages       = {26052--26065},
  year        = {2022},
  OPTdoi      = {10.48550/arXiv.2205.00293},
}

@incollection{goreinov2010submatrix,
  author      = {Goreinov, Sergei A. and Oseledets, Ivan V. and Savostyanov, Dmitry V. and Tyrtyshnikov, Eugene E. and Zamarashkin, Nikolai L.},
  title       = {How to find a good submatrix},
  booktitle   = {Matrix Methods: Theory, Algorithms and Applications},
  publisher   = {World Scientific},
  pages       = {247--256},
  year        = {2010},
  OPTdoi      = {10.1142/9789812836021_0015},
}

@article{savostyanov2014quasioptimality,
  author      = {Savostyanov, Dmitry V.},
  title       = {Quasioptimality of maximum-volume cross interpolation of tensors},
  journal     = {Linear Algebra and its Applications},
  volume      = {458},
  pages       = {217--244},
  year        = {2014},
  OPTdoi      = {10.1016/j.laa.2014.06.006},
}

@article{kapteyn2021probabilistic,
  author      = {Kapteyn, Michael G. and Pretorius, Jacob V. R. and Willcox, Karen E.},
  title       = {A probabilistic graphical model foundation for enabling predictive digital twins at scale},
  journal     = {Nature Computational Science},
  volume      = {1},
  number      = {5},
  pages       = {337--347},
  year        = {2021},
  OPTdoi      = {10.1038/s43588-021-00069-0},
}

@book{nasem2024foundational,
  author      = {{National Academies of Sciences, Engineering, and Medicine}},
  title       = {Foundational Research Gaps and Future Directions for Digital Twins},
  publisher   = {The National Academies Press},
  address     = {Washington, DC},
  year        = {2024},
  OPTdoi      = {10.17226/26894},
}

@article{cui2022deep,
  author      = {Cui, Tiangang and Dolgov, Sergey},
  title       = {Deep composition of tensor-trains using squared inverse {Rosenblatt} transports},
  journal     = {Foundations of Computational Mathematics},
  volume      = {22},
  pages       = {1863--1922},
  year        = {2022},
  OPTdoi      = {10.1007/s10208-021-09537-5},
}

@article{cui2023scalable,
  author      = {Cui, Tiangang and Dolgov, Sergey and Zahm, Olivier},
  title       = {Scalable conditional deep inverse {Rosenblatt} transports using tensor trains and gradient-based dimension reduction},
  journal     = {Journal of Computational Physics},
  volume      = {485},
  pages       = {112103},
  year        = {2023},
  OPTdoi      = {10.1016/j.jcp.2023.112103},
}

@book{engl1996regularization,
  author      = {Engl, Heinz W. and Hanke, Martin and Neubauer, Andreas},
  title       = {Regularization of Inverse Problems},
  series      = {Mathematics and its Applications},
  volume      = {375},
  publisher   = {Kluwer Academic Publishers},
  address     = {Dordrecht},
  year        = {1996},
  OPTdoi      = {10.1007/978-94-009-1740-8},
}

@article{olearyroseberry2022derivative,
  author      = {O'Leary-Roseberry, Thomas and Villa, Umberto and Chen, Peng and Ghattas, Omar},
  title       = {Derivative-informed projected neural networks for high-dimensional parametric maps governed by {PDEs}},
  journal     = {Computer Methods in Applied Mechanics and Engineering},
  volume      = {388},
  pages       = {114199},
  year        = {2022},
  OPTdoi      = {10.1016/j.cma.2021.114199},
}

@article{stuart2010inverse,
  author      = {Stuart, Andrew M.},
  title       = {Inverse problems: A {Bayesian} perspective},
  journal     = {Acta Numerica},
  volume      = {19},
  pages       = {451--559},
  year        = {2010},
  OPTdoi      = {10.1017/S0962492910000061},
}

@book{isakov2017inverse,
  author      = {Isakov, Victor},
  title       = {Inverse Problems for Partial Differential Equations},
  series      = {Applied Mathematical Sciences},
  volume      = {127},
  edition     = {3rd},
  publisher   = {Springer},
  year        = {2017},
  OPTdoi      = {10.1007/978-3-319-51658-5},
}

@book{hesthaven2016certified,
  author      = {Hesthaven, Jan S. and Rozza, Gianluigi and Stamm, Benjamin},
  title       = {Certified Reduced Basis Methods for Parametrized Partial Differential Equations},
  series      = {SpringerBriefs in Mathematics},
  publisher   = {Springer},
  year        = {2016},
  OPTdoi      = {10.1007/978-3-319-22470-1},
}

@article{kolda2009tensor,
  author      = {Kolda, Tamara G. and Bader, Brett W.},
  title       = {Tensor Decompositions and Applications},
  journal     = {SIAM Review},
  volume      = {51},
  number      = {3},
  pages       = {455--500},
  year        = {2009},
  OPTdoi      = {10.1137/07070111X},
}

@article{greif2019decay,
  author      = {Greif, Constantin and Urban, Karsten},
  title       = {Decay of the {Kolmogorov} {N}-width for wave problems},
  journal     = {Applied Mathematics Letters},
  volume      = {96},
  pages       = {216--222},
  year        = {2019},
  OPTdoi      = {10.1016/j.aml.2019.05.013},
}

@article{dihlmann2015certified,
  author      = {Dihlmann, Markus A. and Haasdonk, Bernard},
  title       = {Certified {PDE}-Constrained Parameter Optimization Using Reduced Basis Surrogate Models for Evolution Problems},
  journal     = {Computational Optimization and Applications},
  volume      = {60},
  number      = {3},
  pages       = {753--787},
  year        = {2015},
  OPTdoi      = {10.1007/s10589-014-9697-1},
}

@article{manzoni2019certified,
  author      = {Manzoni, Andrea and Pagani, Stefano},
  title       = {A Certified {RB} Method for {PDE}-Constrained Parametric Optimization Problems},
  journal     = {Communications in Applied and Industrial Mathematics},
  volume      = {10},
  number      = {1},
  pages       = {123--152},
  year        = {2019},
  OPTdoi      = {10.2478/caim-2019-0017},
}

@article{karcher2018reduced,
  author      = {K{\"a}rcher, Mark and Boyaval, S{\'e}bastien and Grepl, Martin A. and Veroy, Karen},
  title       = {Reduced Basis Approximation and A Posteriori Error Bounds for {4D-Var} Data Assimilation},
  journal     = {Optimization and Engineering},
  volume      = {19},
  number      = {3},
  pages       = {663--695},
  year        = {2018},
  OPTdoi      = {10.1007/s11081-018-9389-2},
}

@article{qian2017certified,
  author      = {Qian, Elizabeth and Grepl, Martin A. and Veroy, Karen and Willcox, Karen},
  title       = {A Certified Trust Region Reduced Basis Approach to {PDE}-Constrained Optimization},
  journal     = {SIAM Journal on Scientific Computing},
  volume      = {39},
  number      = {5},
  pages       = {S434--S460},
  year        = {2017},
  OPTdoi      = {10.1137/16M1081981},
}

@article{kartmann2024adaptive,
  author      = {Kartmann, Michael and Keil, Tim and Ohlberger, Mario and Volkwein, Stefan and Kaltenbacher, Barbara},
  title       = {Adaptive Reduced Basis Trust Region Methods for Parameter Identification Problems},
  journal     = {Computational Science and Engineering},
  volume      = {1},
  pages       = {3},
  year        = {2024},
  OPTdoi      = {10.1007/s44207-024-00002-z},
}

@article{zheltkova2018tensor,
  author      = {Zheltkova, Valeriya V. and Zheltkov, Dmitry A. and Grossman, Zvi and Bocharov, Gennady A. and Tyrtyshnikov, Eugene E.},
  title       = {Tensor based approach to the numerical treatment of the parameter estimation problems in mathematical immunology},
  journal     = {Journal of Inverse and Ill-posed Problems},
  volume      = {26},
  number      = {1},
  pages       = {51--66},
  year        = {2018},
  OPTdoi      = {10.1515/jiip-2016-0083},
}

@article{jones1998efficient,
  author      = {Jones, Donald R. and Schonlau, Matthias and Welch, William J.},
  title       = {Efficient Global Optimization of Expensive Black-Box Functions},
  journal     = {Journal of Global Optimization},
  volume      = {13},
  number      = {4},
  pages       = {455--492},
  year        = {1998},
}

@article{booker1999rigorous,
  author      = {Booker, Andrew J. and Dennis, Jr., J. E. and Frank, Paul D. and Serafini, David B. and Torczon, Virginia and Trosset, Michael W.},
  title       = {A rigorous framework for optimization of expensive functions by surrogates},
  journal     = {Structural Optimization},
  volume      = {17},
  number      = {1},
  pages       = {1--13},
  year        = {1999},
}

@article{forrester2009recent,
  author      = {Forrester, Alexander I. J. and Keane, Andy J.},
  title       = {Recent advances in surrogate-based optimization},
  journal     = {Progress in Aerospace Sciences},
  volume      = {45},
  pages       = {50--79},
  year        = {2009},
}

@article{alexandrov1998trust,
  author      = {Alexandrov, Natalia M. and Dennis, Jr., J. E. and Lewis, Robert Michael and Torczon, Virginia},
  title       = {A trust-region framework for managing the use of approximation models in optimization},
  journal     = {Structural Optimization},
  volume      = {15},
  number      = {1},
  pages       = {16--23},
  year        = {1998},
}

@article{yue2013accelerating,
  author      = {Yue, Yao and Meerbergen, Karl},
  title       = {Accelerating Optimization of Parametric Linear Systems by Model Order Reduction},
  journal     = {SIAM Journal on Optimization},
  volume      = {23},
  number      = {2},
  pages       = {1344--1370},
  year        = {2013},
}

@article{zahr2015progressive,
  author      = {Zahr, Matthew J. and Farhat, Charbel},
  title       = {Progressive construction of a parametric reduced-order model for {PDE}-constrained optimization},
  journal     = {International Journal for Numerical Methods in Engineering},
  volume      = {102},
  number      = {5},
  pages       = {1111--1135},
  year        = {2015},
}

@article{jones2001taxonomy,
  author      = {Jones, Donald R.},
  title       = {A Taxonomy of Global Optimization Methods Based on Response Surfaces},
  journal     = {Journal of Global Optimization},
  volume      = {21},
  number      = {4},
  pages       = {345--383},
  year        = {2001},
}

@article{shahriari2016taking,
  author      = {Shahriari, Bobak and Swersky, Kevin and Wang, Ziyu and Adams, Ryan P. and de Freitas, Nando},
  title       = {Taking the Human Out of the Loop: A Review of {Bayesian} Optimization},
  journal     = {Proceedings of the IEEE},
  volume      = {104},
  number      = {1},
  pages       = {148--175},
  year        = {2016},
}

@article{karcher2014certified,
  author      = {K{\"a}rcher, Mark and Grepl, Martin A.},
  title       = {A certified reduced basis method for parametrized elliptic optimal control problems},
  journal     = {ESAIM: Control, Optimisation and Calculus of Variations},
  volume      = {20},
  number      = {2},
  pages       = {416--441},
  year        = {2014},
}

@article{negri2013reduced,
  author      = {Negri, Federico and Rozza, Gianluigi and Manzoni, Andrea and Quarteroni, Alfio},
  title       = {Reduced Basis Method for Parametrized Elliptic Optimal Control Problems},
  journal     = {SIAM Journal on Scientific Computing},
  volume      = {35},
  number      = {5},
  pages       = {A2316--A2340},
  year        = {2013},
}

@article{ladeveze2010latin,
  author      = {Ladev{\`e}ze, Pierre and Passieux, Jean-Charles and N{\'e}ron, David},
  title       = {The {LATIN} multiscale computational method and the proper generalized decomposition},
  journal     = {Computer Methods in Applied Mechanics and Engineering},
  volume      = {199},
  number      = {21--22},
  pages       = {1287--1296},
  year        = {2010},
}

@article{nouy2010priori,
  author      = {Nouy, Anthony},
  title       = {A priori model reduction through proper generalized decomposition
                 for solving time-dependent partial differential equations},
  journal     = {Computer Methods in Applied Mechanics and Engineering},
  volume      = {199},
  number      = {23--24},
  pages       = {1603--1626},
  year        = {2010},
}

\appendix

\section{Regularization parameter selection: discrepancy principle}
\label{sec:discrepancy}

For the numerical assessment of the method, we generate the observations $y_{\mathrm{obs}}$ using the full-order forward map:
\begin{equation}\label{eq:noisyObs}
y_{\mathrm{obs}} = f(\alpha_{\mathrm{true}}) + \varepsilon_{\mathrm{meas}} = y_{\mathrm{true}} + \varepsilon_{\mathrm{meas}}
\end{equation}

\noindent with $y_{\mathrm{true}} = f(\alpha_{\mathrm{true}})$ denoting the noiseless observations and measurement noise $\varepsilon_{\mathrm{meas}}$. The noise contribution is controlled by the noise-to-signal ratio parameter $\delta_{\mathrm{meas}} > 0$.

For our  numerical experiments, we compute
\[
\varepsilon_{\mathrm{meas}}
= \delta_{\mathrm{meas}} \frac{\|y_{\mathrm{true}}\|_2}{\|\varepsilon_{\mathrm{noise}}\|_2} \varepsilon_{\mathrm{noise}},
\qquad \varepsilon_{\mathrm{noise}}\sim \mathcal{N}(0,I).
\]
\smallskip

The discrepancy between model predictions $\widehat f(\alpha)$ and observations $y_{\mathrm{obs}}$ can be attributed to two sources of error: the measurement noise and the ROM/numerical approximation error. We define the relative noise levels as
\[
\delta_{\mathrm{meas}} = \frac{\|\varepsilon_{\mathrm{meas}}\|_2}{\|y_{\mathrm{true}}\|_2}, \qquad
\delta_{\mathrm{rom}} = \frac{\|\varepsilon_{\mathrm{rom}}\|_2}{\|y_{\mathrm{true}}\|_2},
\]

\noindent and the total noise level as $\delta_{\mathrm{tot}} = \sqrt{\delta_{\mathrm{meas}}^2+\delta_{\mathrm{rom}}^2}$. The discrepancy principle~\cite{morozov2012methods} for the choice of $\lambda$ states that one should not fit the data beyond the noise level to avoid overfitting. Therefore a candidate $\lambda$ is accepted  when
\begin{equation}\label{e:discrepancy}
\frac{\|\widehat f(\alpha_{\lambda}) - y_{\mathrm{obs}}\|_2}{\|y_{\mathrm{true}}\|_2} \leq \tau \delta_{\mathrm{tot}}
\end{equation}

\noindent with a \emph{safety factor} $\tau \geq 1$. Here,
\[
\alpha_{\lambda} = \argmin_{\alpha \in \mathcal{A}} \left\{\tfrac{1}{2}\|\widehat f(\alpha) - y_{\mathrm{obs}} \|_2^2 + \tfrac{\lambda}{2}\|\alpha - \alpha_{\mathrm{ref}}\|_2^2\right\}
\]

\noindent is a family of solutions of the optimization problem in~\eqref{eq:inverse_problem} parameterized by $\lambda$. That is, we solve~\eqref{eq:inverse_problem} for multiple $\lambda$ in some range $[\lambda_{\mathrm{min}}, \lambda_{\mathrm{max}}] \subset \mathbb{R}_{+}$ and select as the optimal regularization parameter the largest $\lambda$ for which~\eqref{e:discrepancy} holds. If there is no measurement noise $\varepsilon_{\mathrm{meas}}$ (i.e., $\delta_{\mathrm{meas}} = 0$), the noise level $\delta_{\mathrm{tot}}$ solely depends on the ROM error $\delta_{\mathrm{rom}}$.

The criterion~\eqref{e:discrepancy} requires access to $\|y_{\mathrm{true}}\|_2$, which is generally unknown. In our synthetic experiments, $y_{\mathrm{true}}$ is available since we generate the data. In practice, one often replaces $\|y_{\mathrm{true}}\|_2$ by $\|y_{\mathrm{obs}}\|_2$.

Note that~\eqref{e:discrepancy} can be rewritten as
\begin{equation}\label{e:discrepancy2}
\|\widehat f(\alpha_\lambda)-y_{\mathrm{obs}}\|_2
\le
\rho,
\qquad
\rho \defeq \tau\delta_{\mathrm{tot}} \|y_{\mathrm{true}}\|_2 .
\end{equation}

\noindent In the numerical section, we refer to~\eqref{e:discrepancy2} and call $\rho$ the \emph{prescribed discrepancy level}.

\end{document}